\documentclass[11pt,a4paper,twoside]{book}
\usepackage{setspace}
\usepackage[hmarginratio=1:1, vmarginratio =5:6,
 textheight=22cm,bindingoffset=1cm, textwidth=14.6cm]{geometry}
\usepackage{times}
\usepackage{euscript}
\usepackage[nottoc]{tocbibind}
\usepackage{xcolor}
\usepackage{microtype}
\usepackage{enumerate}
\usepackage[comma,sort&compress]{natbib}
\usepackage{graphicx}
\usepackage{fancyhdr,latexsym,amsmath,amsfonts,amssymb,amsbsy,amsthm,url,extarrows,bbm}
\allowdisplaybreaks
\newtheorem{theorem}{Theorem}[section]
\newtheorem{lemma}{Lemma}[section]
\newtheorem{proposition}{Proposition}[section]
\newtheorem{corollary}{Corollary}[section]

\theoremstyle{definition}
\newtheorem{definition}{Definition}[section]
\newtheorem{example}{Example}[section]

\theoremstyle{remark}
\newtheorem{remark}{Remark}[section]

\numberwithin{equation}{section}
\addtocounter{tocdepth}{1}

\let\myenumi\theenumi
\renewcommand{\labelenumi}{\upshape{(\roman{enumi})}}
 \renewcommand{\theenumi}{\upshape{(\roman{enumi})}}

\newlength{\myfboxsep}
\newlength{\mywidth}
\newcommand{\boxv}{\setlength{\myfboxsep}{\fboxsep}
\setlength{\fboxsep}{0pt}
\,\framebox[\mywidth]{$\vee$}\,
\setlength{\fboxsep}{\myfboxsep}}

\DeclareMathOperator{\E}{E}

\DeclareMathOperator{\prob}{P}

\DeclareMathOperator{\e}{e}
\DeclareMathOperator{\lito}{o}
\DeclareMathOperator{\bigo}{O}
\newcommand{\ol}{\overline}
\newcommand{\Egamma}{\mathsf{E}^{(\gamma)}}
\newcommand{\Ebar}{\overline{\mathsf{E}}^{(\gamma)}}
\newcommand{\vaguec}{\stackrel{\mathrm{v}}{\rightarrow}}
\newcommand{\gmu}{G_{\mu}}
\newcommand{\comment}[1]{}

\newcommand{\xinf}{X_{(\infty)}}
\newcommand{\xupk}{\sum_{t=m+1}^\infty \Theta_t X_t}
\newcommand{\xdnk}{\sum_{t=1}^m \Theta_t X_t}
\newcommand{\xtheta}{\Theta_tX_t}

\include{psfig}
\begin{document}
\pagestyle{empty}
\vspace*{1cm}
\begin{center}
\begin{huge}
\textbf{Essays on regular variations\\ in classical and free setup:\\ randomly weighted sums,\\ products in CEVM and\\ free subexponentiality\\ }
\end{huge}
\end{center}

\vspace{3cm}

\begin{center}
\textbf{\Large Rajat Subhra Hazra\\}
\end{center}

\vspace{1cm}
\begin{figure}[h]
\begin{center}
\includegraphics[width=1in]{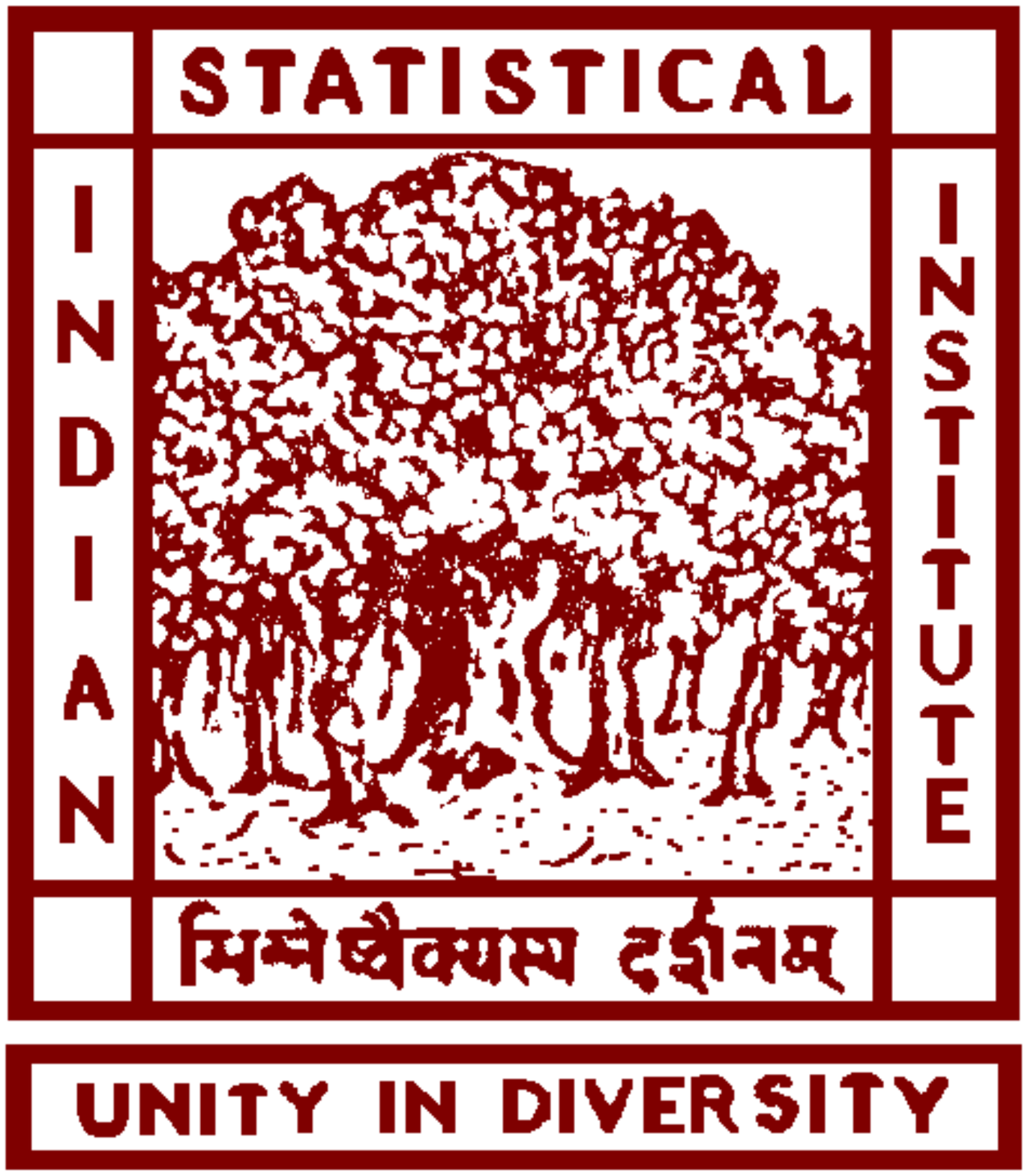}
\end{center}
\end{figure}
\vspace{1cm}

\begin{center}
\begin{large}
\textbf{Indian Statistical Institute, Kolkata\\[1ex]
April, 2011}
\end{large}
\end{center}

\cleardoublepage

\pagestyle{empty}
\begin{center}
\begin{huge}
\textbf{Essays on regular variations\\ in classical and free setup:\\ randomly weighted sums,\\ products in CEVM and\\ free subexponentiality\\}
\end{huge}
\end{center}

\vspace{3cm}

\begin{center}
\textbf{\Large Rajat Subhra Hazra\\}
\end{center}

\vspace{3cm}
\begin{center}
\small{\textbf{Thesis submitted to the Indian Statistical Institute\\
in partial fulfillment of the requirements\\
for the award of
the degree of\\
Doctor of Philosophy\\[1ex]
April, 2011}}
\end{center}

%\vspace{1cm}
\begin{figure}[h]
\begin{center}
\includegraphics[width=1cm]{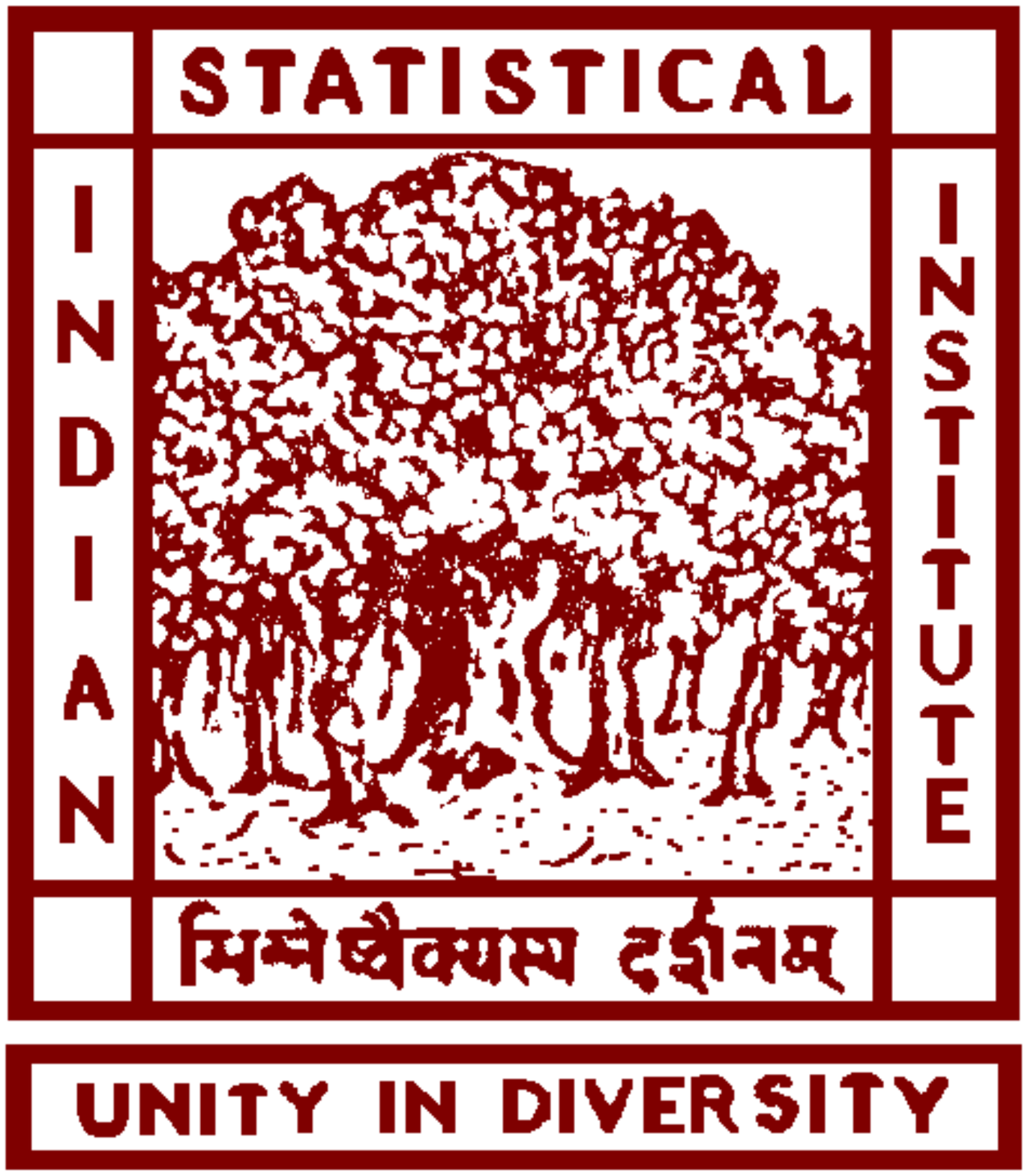}
\end{center}
\end{figure}

\begin{center}
\begin{large}
\textbf{Indian Statistical Institute\\
203, B.T. Road, Kolkata, India.}
\end{large}
\end{center}

\cleardoublepage
\pagenumbering{roman}
\pagestyle{plain}
\pagestyle{empty}
\vspace*{8cm}
\begin{center}
{\it{\Large To \\  Baba, Maa and Monalisa }\\}
\end{center}
\cleardoublepage
%\vspace*{1cm}
%\def\baselinestretch{1.2}
%\include{thanks}
\begin{center}
\textbf{\Large Acknowledgements}\\
\end{center}
%\vspace{2cm}
 
\vspace{2cm}
The time spent in writing a Ph.D.\ thesis provides a unique opportunity to look back and get a new perspective on some years of work. This in turn, makes me feel the need to acknowledge many people who have helped me in many different ways to the realization of this work. I like to begin by thanking my favorite poet Robert Frost, for  the wonderful lines:
\begin{center}
\textit{Woods are lovely, dark and deep.\\
 But I have promises to keep,\\
 And miles to go before I sleep,\\
 And miles to go before I sleep.}
\end{center}
These four lines from `Stopping by Woods on a snowy evening' have been very special in my life and provided me inspiration since my high school days.

I would like to express my sincerest gratitude to my thesis advisor Krishanu Maulik for being supportive and for taking immense pain to correct my incorrectness in every step. His constructive suggestions and criticisms have always put me in the right track. Moreover, I am also grateful to him for being more like an elder brother to me.

I am also grateful to Arup Bose for many things which I cannot pen down completely. In Autumn of 2006, he infused in me the desire to learn probability. It has been a very enjoyable experience since then. I would also thank him for introducing me to the interesting area of Random matrices and for allowing me to be a part of his adventure with Koushik. I would like to acknowledge B.V. Rao, whose mere presence in ISI was a big motivation for me. His humbleness and propensity to talk at length about many mathematical and non-mathematical things have helped me to develop as a human being. Also, I am grateful that he has spent a huge amount of time to explain me various details in many subjects. His classes will remain one of the best guided tours I have ever been to. His moving to CMI was a great loss not only to me but to whole of ISI. 

I would also like to thank Debashish Goswami and Alok Goswami for helping me with important suggestions regarding my work. I am also indebted to the teachers who have taught me. Special thanks to Amites Dasgupta, S.M. Srivastava, S.C. Bagchi, Gautam Mukherjee, Amiya Mukherjee, Antar Bandyopadhyay, Gopal Basak, Rudra Sarkar, Mahuya Dutta, Anil Ghosh and Probal Chaudhuri for showing me the broader spectrum. I also thank Sudip Acharya, Manab Mukerjee, Mihir Chakraborty and Biswajit Mitra for helping me during my M.Sc. I am also thankful to  Partha Pratim Ghosh, Satyajit Sir, Ananda Dasgupta for conducting classes till late evening during our college days and making us feel what research is all about. I would also like to thank my class teachers in school for loving me and teaching me discipline. 

During my research I have also got good suggestions from Siva Athreya, Antar Bandyopadhyay, Parthanil Roy, K. R. Parthasarathi, Arijit Chakraborty, Roland Speicher, Manjunath and many others. Thanks to all of them and a special thanks to the LPS team for organizing a very fruitful meet every year.

How can a journey be possible without friends? I have always had a wonderful set of friends. I take this opportunity to apologize to all those friend with whom  I have not been able to meet during the last few years. I would like to thank a Yahoo group called ``Hostelmasti'' which has been a part of my life since my college days. I sincerely thank all college friends specially Arya, Ravi, Pralay, Deba, Mithun, Arnab, Projjal, Sukesh, Soumyadeep and Gopal for staying with me during my good and bad days. I also thank Bodhisattva, Subhadeep, Arunava, Suhel, Mainak, Rishi and Apratim for being by my side since the school days.

The entry to ISI life won't have been easy without my friend Jyotishman. He has been a very close companion to me since my college days. I hope this camaraderie will remain the rest of my life. I cant completely show my gratitude to him by mere words. After he left ISI for a post-doc, `the fun' has gone down in ISI. I also thank Koushik for allowing me to be a part of his research and for bearing up with my eccentricity for five years.  Biswarup's songs and `kurkur'  are unforgettable  and thanks to him for being a great friend. I am also thankful to Pustida, Ashishda, Abhijitda, Palda, Debashisda, Prosenjitda, Radhe, Santanu, Subhro, Subhajit, Palash, Buddhananda, Subhabrata, Soumenda, Qaiser, Swagatadi, Suparna, Dibyenduda, Sedigeh, Zinna, Soumalya, Arup, Suman and many others for making my research life stress free. 

I wish this to be an occasion to thank my family. I would like to mention their everlasting support and sincere commitments which have helped me to arrive at this stage. My father's encouraging words are always great source of inspiration. I thank Bappa for silently supporting me. I am also indebted to my mother and father in law for their continuous support and encouraging words. Finally I want to thank my wife Monalisa, for being a part of my tension, sorrow, joy and many more things. Thanks to her for being by my side and having faith in me. Her enduring love has always been a privilege for me. I regard every achievement as a success for two of us.

Last but not the least, I thank God for all the blessings He has showered on me and without whose help the journey would not have been possible.
%\bigskip
%\noindent8 November, 2010\hfill Rajat Subhra Hazra

\cleardoublepage
\cleardoublepage
\tableofcontents
\cleardoublepage
\addtolength{\headheight}{15pt}
\pagenumbering{arabic}
\setcounter{page}{1}
\pagestyle{myheadings}
\pagestyle{fancy}
\renewcommand{\chaptermark}[1]{\markboth{\chaptername
\ \thechapter:\,\ #1}{}}
\renewcommand{\sectionmark}[1]{\markright{\thesection\,\ #1}}
\fancyhf{}
\fancyhead[RE]{\sl\leftmark}
\fancyhead[RO,LE]{\rm\thepage}
\fancyhead[LO]{\sl\rightmark}
\fancyfoot[C,L,E]{}
%\doublespacing
%\include{intro}
\chapter{An introduction to heavy tailed distributions}\label{chap:introduction}

\section{Introduction}
% In this thesis we shall mainly be concerned with application of regular variations in different branches of mathematics and probability. The theory of regular variations has already established it position in various branches of mathematics like Harmonic analysis, number theory, complex analysis and many more. The Tauberian theorems play a very crucial role in different disciplines. In probability theory one mainly encounters this when one considers domain of attraction of stable random variables. It also plays a major role in renewal theory, queuing theory and asymptotics of Laplace transform. The power laws which are the basic blocks in regular variations are of great importance in modelling various natural phenomenon in modern times. It has also found applications in newly emerging areas like free probability and random matrices. In this thesis we shall mainly be concerned with sums and products of regularly varying random variables in different setups. 
% 
In this thesis, we shall be focusing  on  some problems in probability theory involving regularly varying functions. The theory of regular variations has played an important role in probability theory, harmonic analysis, number theory, complex analysis and many more areas of mathematics. For an encyclopedic treatment of the subject, we refer to \cite{bingham:goldie:teugels:1987}. In probability theory, the limiting behavior of the sums of independent and identically distributed (i.i.d.)\ random variables is closely related to regular variation. The books by \cite{feller:1971} and \cite{gnedenko:kolmogorov} give  characterizations of random variables in the domains of attraction of stable distributions in terms of regularly varying functions. The study of extreme value theory was first initiated by \cite{gnedenko:1943,fisher:tippett:1928}.  The use of regular variation in extreme value theory is now very well known due to the works of \cite{dehaan:1970,dehaan:1971}.   The Tauberian theorems involving regularly varying functions play a very crucial role in different disciplines. \cite{bingham:goldie:teugels:1987} gives a very good account of the Tauberian theorems for Laplace and Mellin transforms.  The use of regularly varying functions is also very popular in insurance and risk theory. For a comprehensive study of different applications, we refer to the book by~\cite{embrechts:kluppelberg:mikosch:book}. The study of record values and record times also uses various properties of extreme value theory and regularly varying functions \citep[Chapter 4]{Resnick:1987}. The theory of multivariate regular variation is useful in modelling various telecommunication systems and Internet traffic  \citep{mikosch:resnick:rootzen:stegeman:2002, heath:resnick:samorodnitsky:1998, Maulik:Resnick:Rootzen:2002}. The study of certain environmental issues is also facilitated when one considers the theory of multivariate regular variations. See~\cite{dehaan:ronde:1998,heffernan:tawn:2004} for some recent applications in this area.

The theory of regular variations has also found an important place in free probability. Free probability was introduced by \cite{voiculescu1986addition}. While free probability has an important application in the study of random matrices, its connection with various topics like the study of the free groups and the factor theory have made it  a subject of its own interest. \cite{bercovici2000functions,bercovici1999stable} studied the domain of attraction of free stable laws and showed that the regularly varying functions play a very crucial role like they do in the classical setup. In the final chapter of this thesis, we show another application of regular variation in free probability theory. 

In Section~\ref{chap1:intro:defs} we give the definitions of regularly varying functions, distribution functions and measures with regularly varying tails and some of their properties. In Subsections~\ref{Some well known facts about regular variations} and~\ref{chap1:intro:products and sum} we state some well known results on the regularly varying functions which we shall use in the later chapters. In Section~\ref{chap1:sec:heavy tailed} we recall the definitions of subexponential and long tailed distributions and some of their properties. In Section~\ref{chap1:sec:extreme values} we give a brief introduction to one dimensional extreme value theory and also point out its connections with regularly varying functions. In Section~\ref{summary} we give a brief overview of the results in the later chapters.

\subsection{Notations}

The space of natural numbers, integers, real numbers and complex numbers will be denoted by $\mathbb N$, $\mathbb Z$, $\mathbb R$ and $\mathbb C$ respectively. By $\mathbb R^+$ we shall mean the space of non-negative reals, that is, $[0,\infty)$. For a complex number $z$, $\Re z$ and $\Im z$ will denote its real and imaginary parts respectively.  We shall denote the upper and the lower halves of the complex plane respectively by $\mathbb C^+$ and $\mathbb C^-$ , namely, $\mathbb C^+ = \{z\in\mathbb C: \Im z>0\}$ and $\mathbb C^- = - \mathbb C^+$.

For a real valued functions $f$ and $g$, we shall write $f(x)\approx g(x)$, $f(x)=\lito(g(x))$ and $f(x)=\bigo(g(x))$ as $x\to\infty$ to mean that $f(x)/g(x)$ converges to a non-zero limit, ${f(x)}/{g(x)}\rightarrow 0$ and $f(x)/g(x)$ stays bounded as $x\to\infty$ respectively. If the non-zero limit is $1$ in the first case, we write $f(x) \sim g(x)$ as $x\to\infty$. For complex valued functions the notations are explained in Section~\ref{subsec: non-tang} of Chapter~\ref{free sub chapter}.

By $x^+$  and  $x^-$ we shall mean $\max\{0,x\}$ and $\max\{0,-x\}$ respectively. For a distribution function $F$ the tail of the distribution function will be denote by $\overline{F}(x)=1-F(x)$.

If $\mathsf{S}$ is a topological space with $\mathcal{S}$ being its Borel $\sigma$-field, then for non-negative Radon measures $\mu_t$, for $t>0$, and $\mu$ on $(\mathsf S, \mathcal S)$, we say $\mu_t$ converges vaguely to $\mu$ and denote it by $\mu_t \vaguec \mu$, if for all relatively compact sets $C$, which are also $\mu$-continuity sets of $\mu$, that is, $\mu( \partial C ) = 0$, we have $\mu_t (C) \to \mu(C)$, as $t \to \infty$. By $M_+(\mathsf S)$ we shall mean the space of all Radon measures on $\mathsf{S}$ endowed with the topology of vague convergence.

\section{Regularly varying functions and random variables}\label{chap1:intro:defs}

Regularly varying functions are an integral part in the study of heavy tailed distributions. A regularly varying function asymptotically looks like a power function. The use of regularly varying functions is extensive in literature. They are heavily used in characterizing the domains of attraction, stable laws, modeling  long range dependence and extreme value theory. In this section we collect some  basic properties of the regularly varying functions, which we shall use throughout this thesis. Most of the proofs are very common and are well available in literature. For more detailed theory of regular variations and related classes, we refer the readers to \cite{bingham:goldie:teugels:1987}. We now recall the definition of a regularly varying function.

\begin{definition}
 A measurable function $f: \mathbb{R}_+\rightarrow \mathbb{R}_+$ is called \textit{regularly varying (at infinity) with index $\alpha$}, if for $t>0$,
\begin{equation}
\label{chap1:eq:rv}
 \lim_{x\rightarrow\infty}\frac{f(tx)}{f(x)}=t^{\alpha}.
\end{equation}
If $\alpha=0$, $f$ is said to be \textit{slowly varying.}

The space of regularly varying functions (at infinity) with index $\alpha$ will be denoted by $RV_{\alpha}$.
\end{definition}

In analogy to regular variation at infinity one can talk about regular variation at zero with an obvious change in the above definition. A function $f(x)$ is said to be regularly varying at zero if and only if $f(x^{-1})$ is regularly varying at infinity. In the subsequent chapters,  regular variation will always mean the behavior at infinity unless otherwise specified.

From the above definition it is clear that one can always represent a regularly varying function with index $\alpha$ as $x^{\alpha}L(x)$ for some slowly varying function $L$.

\begin{example}
$x^{\alpha}$, $x^{\alpha}\ln(1+x)$, $x^\alpha \ln(\ln(e+x))$ are some examples of regularly varying function at infinity with index $\alpha$.
\end{example}

\subsection{Some well known results about regular variations}\label{Some well known facts about regular variations}
This subsection provides some basic properties of  regularly varying functions, which shall be useful in later chapters of this thesis. The proofs of these results are available in Chapter 1 of \cite{bingham:goldie:teugels:1987}.

\begin{enumerate}
 \item \label{chap1:karamata rep} A positive function $L$ on $[x_0,\infty)$ is slowly varying if and only if it can be written in the form 
\begin{equation}
 \label{chap1:eq:karamata rep} 
L(x)=c(x)\exp\left(\int_{x_0}^x\frac{\epsilon(y)}{y}dy\right),
\end{equation}
where $c(\cdot)$ is a measurable nonnegative function such that $\lim_{x\rightarrow\infty}c(x)=c_0 \in (0,\infty)$ and $\epsilon(x)\rightarrow 0$ as $x\rightarrow\infty$.

This is known as Karamata's Representation for slowly varying functions and we refer to Corollary of Theorem 0.6 of \cite{Resnick:1987}.

\item\label{chap1:infinite:zero} For a regularly varying $f$ with index $\alpha\neq 0$, as $x\rightarrow\infty$,
$$f(x)\rightarrow\begin{cases}
                  \infty & \text{if}\quad \alpha >0\\
0 &\text{if}\quad \alpha<0.
                 \end{cases}$$

Moreover, if $L$ is slowly varying then for every $\epsilon>0$, $x^{-\epsilon}L(x)\rightarrow0$ and $x^{\epsilon}L(x)\rightarrow\infty$, as $x\rightarrow\infty$.

\item\label{chap1:uniformconv} If $f$ is regularly varying with index $\alpha$ (if $\alpha>0$, we also assume $f$ is bounded on each interval $(0,x]$, $x>0$) then for $0<a\leq b<\infty$, the equation~\eqref{chap1:eq:rv} holds uniformly in $t$
\begin{enumerate}
 \item on each $[a,b]$ if $\alpha=0$,
\item on each $(0,b]$ if $\alpha>0$,
\item on each $[a,\infty)$ if $\alpha<0.$
\end{enumerate}
\end{enumerate}

\begin{theorem}[Karamata's Theorem, \citealp{Resnick:1987}, Theorem 0.6]
 \label{chap1:theorem:karamata}
Let $L$ be a slowly varying function and locally bounded in $[x_0,\infty)$ for some $x_0\geq 0$.
\begin{enumerate}
 \item If $\alpha>-1$, then
\begin{equation*}
 \int_{x_0}^x t^{\alpha}L(t)dt \sim \frac{x^{\alpha+1}L(x)}{\alpha+1}, \quad \text{as } x\rightarrow\infty.
\end{equation*}
\item If $\alpha<-1$, then
\begin{equation*}
 \int_{x}^{\infty} t^{\alpha}L(t)dt \sim -\frac{x^{\alpha+1}L(x)}{\alpha+1}, \quad \text{as } x\rightarrow\infty.
\end{equation*}

\item If $\alpha=-1$, then 
 $$\frac1{L(x)}\int_{x_0}^{x} \frac{L(t)}{t}dt\rightarrow\infty \quad \text{as } x\rightarrow\infty$$ and $\int_{x_0}^x \frac{L(t)}{t}dt$ is slowly varying.

\item If $\alpha=-1$ and $\int_{x_0}^\infty \frac{L(t)}{t}dt<\infty$, then
$$\frac1{L(x)}\int_{x}^{\infty} \frac{L(t)}{t}dt\rightarrow\infty \quad \text{as } x\rightarrow\infty$$ and $\int_{x}^\infty \frac{L(t)}{t}dt$ is slowly varying.

\end{enumerate}

\end{theorem}
% \begin{remark}
%  There are other variants of Karamata's Theorem which are sometimes very useful. We refer to Theorems~1.6.4, 1.6.5 and 1.7.2 of \cite{bingham:goldie:teugels:1987} for details of these variants. 
% \end{remark}

\begin{theorem}[Potter's bound, \citealp{Resnick:1987}, Proposition 0.8]\label{chap1:theorem:potter:1}

 Suppose $f\in RV_{\rho}$, $\rho\in \mathbb{R}$. Take $\epsilon>0$. Then there exists $t_0$ such that for $x\geq 1$, and $t\geq t_0$, we have 
\begin{equation}\label{potter1}
(1-\epsilon)x^{\rho-\epsilon}<\frac{f(tx)}{f(t)}<(1+\epsilon)x^{\rho+\epsilon}.
\end{equation}
\end{theorem}

We now state  Karamata's Tauberian theorem for Laplace-Stieltjes  and Stieltjes transforms. 
\begin{definition}For a non-decreasing right continuous function $U$ on $[0,\infty)$ define its \textit{Laplace-Stieltjes} transform 
$$\widehat U(s):=\int_0^\infty e^{-sx}dU(x).$$
\end{definition}
\begin{theorem}[Karamata's Tauberian theorem, \cite{bingham:goldie:teugels:1987}, Theorem 1.7.1]
\label{chap1:theo:tauberian:bingham}
Let $U$ be a nondecreasing, right continuous function defined on $[0,\infty)$. If $L$ is slowly varying, $c\geq 0$, $\alpha\geq 0$, then the following are equivalent:
\begin{enumerate}
\item $U(x)\sim \frac{cx^{\alpha}L(x)}{\Gamma(1+\alpha)}$, as $x\rightarrow\infty$
\item $\widehat{U}(s)\sim cs^{-\alpha}L(1/s)$,  as $s\downarrow0$. 
\end{enumerate}

If $c>0$ then any one of the above implies $U(x)\sim \widehat{U}(1/x)/\Gamma(1+\alpha).$

\end{theorem}

A similar Tauberian result for the Stieltjes transform of order $\rho$ can be derived using the above result on Laplace-Stieltjes transform. 

\begin{definition}If $U$ is non-decreasing function on $[0,\infty)$ and $\rho>0$, let $S_\rho(U;\cdot)$, the Stieltjes transform of order $\rho$ be  defined as,
 \begin{equation*}
  S_\rho(U;x):=\int_0^\infty (x+y)^{-\rho}dU(y), \quad x>0.
 \end{equation*}
\end{definition}
The following result will be useful in Chapter~\ref{free sub chapter}.

 \begin{theorem}[\citealp{bingham:goldie:teugels:1987}, Theorem 1.7.4]
 \label{chap1:theo:karamata:stieltjes}
Suppose $0< \sigma\leq \rho$, $c\geq 0$ and $L$ is slowly varying,
\begin{equation*}
 U(x)\sim \frac{c \Gamma(\rho)}{\Gamma(\rho)\Gamma(\rho-\sigma+1)} x^{\rho-\sigma}L(x) \quad \text{as $x\rightarrow\infty$}
\end{equation*}
   if and only if
\begin{equation*}
 S_\rho(U,x)\sim x^{-\sigma}L(x) \quad \text{as $x\rightarrow\infty$}.
\end{equation*}

\end{theorem}

\subsection{Sums and products of regularly varying random variables}\label{chap1:intro:products and sum}
In this subsection, we define random variables with regularly varying tails and discuss some of the well known results on sums and product of such random variables. For a nice detailed overview of the properties of  sum and products of random variables with regularly varying tails, we refer to \cite{jessen:mikosch:2006}.

\begin{definition}
\label{def: reg var rv}
 \begin{enumerate}
  \item [\upshape{(a)}] We say a nonnegative (real) random variable $X$ with distribution function $F$ has a regularly varying tail of index $-\alpha$,  if the tail of the distribution function $1-F(\cdot):=\overline{F}(\cdot)=\prob[X>\cdot]\in RV_{-\alpha}$.
\item [\upshape{(b)}] A finite measure $\mu$ on $\mathbb{R}_+$ is said to have regularly varying tail of index $-\alpha$,  if $\mu(\cdot,\infty)\in RV_{-\alpha}$. 
 \end{enumerate}
\end{definition}

\begin{remark}
\begin{enumerate}
 \item It is implicit in the definition that we require $F(x)$ and $\mu(x,\infty)$ to be positive for all $x\in \mathbb{R}_+$ to make sense of the definition.

\item  It is clear from~\ref{chap1:infinite:zero} of Subsection~\ref{Some well known facts about regular variations} that, if $F$ or $\mu$ has regularly varying tail of index $-\alpha$, then we necessarily have $\alpha\geq 0$. 

\item\label{chap1:momentsrv} If $X$ is nonnegative random variable having regularly varying tail of index $-\alpha$ with $\alpha>0$, then 
$\E[X^{\beta}]< \infty$ for  $\beta<\alpha$ and $\E[X^{\beta}]=\infty$ for $\beta>\alpha$.

\item By Theorem 3.6 of~\cite{resnick:2007}, Definition~\ref{def: reg var rv}(a) is equivalent to the existence of a positive function $a(\cdot)\in RV_{1/\alpha}$, such that $t\prob[X/a(t)\in \cdot]$ has a vague limit in $M_+((0,\infty])$, where the limit is a nondegenerate Radon measure. The limiting measure necessarily takes values $cx^{-\alpha}$ on set $(x,\infty]$, for some constant $c>0$.
\end{enumerate}
\end{remark}

The definition for regular variation of not necessarily positive random variable extends in a very natural way to accommodate  both the tails of random variable. 
\begin{definition}
A random variable $X$ (not necessarily positive) has a regularly varying tail of index $-\alpha$ with $\alpha\geq 0$ if $|X|>\cdot]\in RV_{-\alpha}$ and 
\begin{equation}
\label{eq:condition:tail balance}
\frac{\prob[X>x]}{\prob[|X|>x]}\rightarrow p \quad \text{and} \quad \frac{\prob[X< -x]}{\prob[|X|>x]}\rightarrow q \text{ as $x\rightarrow\infty$},  
\end{equation}
with $0<p<1$ and $p+q=1$. We shall often call~\eqref{eq:condition:tail balance} as the tail balance condition.
\end{definition}

Now we state some properties of random variables with regularly varying tails which shall be useful later. The next result is a reformulation of  Karamata's Theorem~\ref{chap1:theorem:karamata} in terms of the distribution functions.
\begin{theorem}[\citealp{bingham:goldie:teugels:1987}, Theorems~1.6.4, 1.6.5 and 1.7.2]\label{lemma:karamata:df}\hfill
\begin{enumerate}

\item \label{Kara:1} Suppose $F$ has  regularly varying tail of index $-\alpha$ with $\alpha>0$ and $\beta\geq \alpha$. Then
\begin{equation}
\label{chap1:eq:karamata:distributions}
\lim_{x\rightarrow\infty}\frac{x^{\beta}\overline{F}(x)}{\int_0^x y^{\beta}F(dy)}=\frac{\beta-\alpha}{\alpha}.
\end{equation}
Conversely, if $\beta>\alpha$ and~\eqref{chap1:eq:karamata:distributions} holds then $F$ has a regularly varying tail of index $-\alpha$. If $\beta=\alpha$ then $\overline{F}(x)=\lito(x^{-\alpha}L(x))$ for some slowly varying function $L$.

\item \label{Kara:2}Suppose $F$ has  regularly varying tail of index $-\alpha$ with $\alpha>0$ and $\beta< \alpha$. Then
\begin{equation}
\label{chap1:eq:karamata:distributions:new}
\lim_{x\rightarrow\infty}\frac{x^{\beta}\overline{F}(x)}{\int_x^{\infty} y^{\beta}F(dy)}=\frac{\beta-\alpha}{\alpha}.
\end{equation}
Conversely, if $\beta<\alpha$ and~\eqref{chap1:eq:karamata:distributions:new} holds, then $F$ has a regularly varying tail of index~$-\alpha$.

\item The function $x\mapsto\int_0^x y^\gamma F(dy)$ is slowly varying if and only if $\overline{F}(x)=\lito\left(x^{-\gamma}\int_0^x y^\gamma F(dy)\right)$ for some $\gamma>0$.

% \item For $\eta>0$, $\overline{F}(x)$ is regularly varying with index $-\alpha$, $\alpha>0$ then 
% \begin{equation*}
% \lim_{x\rightarrow\infty}\frac{x^{\eta}\overline{F}(x)}{\int_x^{\infty}u^{-\eta}F(dx)}=\frac{\alpha}{\eta-\alpha}.
% \end{equation*}

\end{enumerate}
\end{theorem}
The following result is a reformulation of the Potter's bound described in Theorem~\ref{chap1:theorem:potter:1} in terms of distribution functions and random variables.
\begin{theorem}[\citealp{resnick:willekens:1991}, Lemma 2.2]\label{chap1:theorem:potter}
 Let $Z$ be a nonnegative random variable with distribution  function $F$ which has regularly varying tail of index $-\alpha$ with $\alpha>0$. Given $\epsilon>0$, there exists $x_0=x_0(\epsilon)$, $K=K(\epsilon)>0$ such that, for any $c>0$:

\begin{equation}\label{potter2} \frac{\overline{F}(x/c)}{\overline{F}(x)}\leq 
\begin{cases}
(1+\epsilon)c^{\alpha+\epsilon}& \text{ if } c\geq 1, x/c\geq x_0 \\
(1+\epsilon)c^{\alpha-\epsilon}& \text{ if } c<1, x\geq x_0. 
\end{cases}\end{equation}
and
\begin{equation} \label{potter3}
\E[(cZ\wedge x)^{\alpha+\epsilon}]\leq 
\begin{cases}
(1+\epsilon)c^{\alpha+\epsilon}x^{\alpha+\epsilon}\overline{F}(x) & \text{ if } c\geq 1, x/c\geq x_0 \\
(1+\epsilon)c^{\alpha-\epsilon}x^{\alpha+\epsilon}\overline{F}(x) & \text{ if } c<1, x\geq x_0. 
\end{cases}\end{equation}

\end{theorem}

We now state an important result about sums of i.i.d.\ random variables with regularly varying tails, which is extended to the free probability setup in Chapter~\ref{free sub chapter}.

\begin{lemma}[\citealp{embrechts:kluppelberg:mikosch:book}, Lemma 1.3.1]
\label{lem:one large jump}
If $\{X_i\}$ are i.i.d.\ nonnegative random variables with regularly varying tails of index $-\alpha$ with $\alpha\geq 0$, then for each $n\in \mathbb N$,
\begin{equation}
\label{eq:large jump}
\prob\left[\sum_{i=1}^n X_i >x\right]\sim n\prob[X_1>x] \quad \text{as}\quad x\rightarrow\infty.
\end{equation}
\end{lemma}
This property is often referred to as the principle of one large jump, since~\eqref{eq:large jump} implies $\prob[S_n>x]\sim \prob[M_n>x]$ where $S_n=\sum_{i=1}^n X_i$ and $M_n=\max\{X_1,\cdots, X_n\}$.

\begin{proof} We now briefly indicate the proof of the above result for $n=2$. As~$\{X_1+X_2>x\}\supset \{X_1>x\}\cup\{X_2>x\}$ we get,
$$\prob\left[X_1+X_2>x\right]\geq \prob\left[X_1>x\right]+\prob\left[X_2>x\right](1+\lito(1)).$$
% 
% First note that the lower bound follows from the inequality,
% \begin{multline*}
% \prob[X_1+X_2>x]\geq \prob[X_1>x, X_2\leq x]+\prob[X_1\leq x, X_2>x ]\\ =2\left(\prob[X_1>x]-\prob[X_1>x]^2\right).\end{multline*}

For the upper bound choose $\delta\in(0,1)$ and observe,
\begin{align*}
\prob[X_1+X_2>x]&\leq \prob[X_1>(1-\delta)x]+\prob[X_2>(1-\delta)x]+\prob[X_1>\delta x]\prob[X_2>\delta x].\\
&=\prob[X_1>(1-\delta)x]+\prob[X_2>(1-\delta)x](1+\lito(1)),
\end{align*}
and the result now follows by taking $\delta\downarrow0.$
\end{proof}

\begin{remark}
Note that the above proof also shows that is $X$ and $Y$ are nonnegative random variables with regularly varying tails of index $-\alpha$ and $-\beta$, $\alpha>0$ and $\beta>0$ with $\alpha<\beta$ then $$\prob[X+Y>x]\sim \prob[X>x]\quad \text{as}\quad x\rightarrow\infty.$$ So the random variable with heavier tail dominates the sum.
\end{remark}
The following result considers weighted sum of i.i.d.\ (not necessarily nonnegative) random variables with regularly varying tails.
\begin{lemma}[\citealp{embrechts:kluppelberg:mikosch:book}, Lemma A3.26] Let $\{Z_i\}$ be an i.i.d.\ sequence of random variables having regularly varying tails of index $-\alpha$ with $\alpha\geq 0$, where they satisfy the tail balance condition~\eqref{eq:condition:tail balance}. Then for any sequence of real constants $\{\psi_i\}$ and $m\geq 1$, 
$$\prob[\psi_1Z_1+\cdots+\psi_mZ_m>x]\sim \prob[|Z_1|>x]\sum_{i=1}^m\left[p (\psi_i^+)^{\alpha}+q(\psi_i^-)^{\alpha}\right],$$

\end{lemma}

The above result can be extended to infinite weighted series of  i.i.d.\  random variables with regularly varying tails. The following result considers $X=\sum_{j=1}^{\infty} \psi_j Z_j$ for an i.i.d.\ sequence $\{Z_i\}$ of  random variables with regularly varying tails of index $-\alpha$. This kind of infinite series appears in the study of the  extreme value properties of linear processes and autoregressive processes. Before proceeding with the tail behavior of this series, one needs to consider the almost sure finiteness of the series, which follows from  Three Series Theorem. The almost sure finiteness and tail behavior of $X$ were considered by \cite{cline:1983}. See also Theorem 2.2 of \cite{kokoszka:taqqu:1996}.

\begin{theorem}
 \label{chap1:theo:infiniteseries:non random weights}
Let $\{Z_i\}$ be an i.i.d.\ sequence of  random variables with regularly varying tails of index $-\alpha$ with $\alpha>0$, which satisfy the tail balance condition~\eqref{eq:condition:tail balance}. Let $\{\psi_i\}$ be a sequence of real valued weights. Assume that one of the following condition holds:
\begin{enumerate}
 \item $\alpha>2$, $\E[Z_1]=0$ and $\sum_{i=1}^{\infty}\psi_i^2<\infty$;
\item $\alpha\in(1,2]$, $\E[Z_1]=0$ and $\sum_{i=1}^{\infty}|\psi_i|^{\alpha-\epsilon}<\infty$ for some $\epsilon>0$;
\item $\alpha\in(0,1]$, $\sum_{i=1}^{\infty}|\psi_i|^{\alpha-\epsilon}<\infty$,  for some $\epsilon>0.$
\end{enumerate}
Then $X=\sum_{j=1}^{\infty} \psi_j Z_j$ converges almost surely and 
\begin{equation*}
 \prob[X>x]\sim \prob[|Z_1|>x]\sum_{i=1}^{\infty}\left(p (\psi_i^+)^{\alpha}+q(\psi_i^-)^{\alpha}\right).
\end{equation*}

\end{theorem}

The product behavior of the random variables with regularly varying tails is as important as the sum of such random variables. The product behavior is a bit more delicate than the sums. For a review of the results on product of random variables we refer to Section 4 of \cite{jessen:mikosch:2006}.
\begin{theorem}[\citealp{jessen:mikosch:2006}, Lemma 4.2]
 \label{chap1:theo:product:item 2}
  Let $X_1,X_2\cdots, X_n$  be i.i.d.\ random variables, with $\prob[X_1>x]\sim c^\alpha x^{-\alpha}$ for some $c>0$. Then 
\begin{equation}
 \prob[X_1X_2\cdots X_n>x]\sim \frac{\alpha^{n-1}c^{n\alpha}}{(n-1)!}x^{-\alpha}\log ^{n-1}x.
\end{equation}
\end{theorem}

\begin{theorem}
 \label{chap1:theo:product}
Suppose $X$ and $\Theta$ are independent nonnegative random variables and $X$ has regularly varying tail of index $-\alpha$ with $\alpha>0$. 
\begin{enumerate}
 \item If $\Theta$ has regularly varying tail of index $-\alpha$ with $\alpha>0$, then $\Theta X$ also has regularly varying tail of index $-\alpha$.\label{chap1:theo:product:item 1}

\item \label{chap1:theo: breiman} If there exists an $\epsilon>0$, such that $\E[\Theta^{\alpha+\epsilon}]<\infty,$ then
\begin{equation}
\label{eq:chap1:breiman's asymptotics}
 \prob[\Theta X>x]\sim \E[\Theta^{\alpha}]\prob[X>x].
\end{equation}
\item  \label{chap1:theo:product:item 4} Under the assumptions of part~\ref{chap1:theo: breiman},
$$\sup_{x\geq y}\left|\frac{\prob[\Theta X>x]}{\prob[X>x]}-\E[\Theta^{\alpha}]\right|\rightarrow 0 \quad \text{as } y\rightarrow\infty.$$

\item  \label{chap1:theo:product:item 5}If $\prob[X>x]=x^{-\alpha}$ for $x\geq 1$ and $\E[\Theta^{\alpha}]<\infty$, then~\eqref{eq:chap1:breiman's asymptotics} holds. 

\end{enumerate}

\end{theorem}

Proof of part~\ref{chap1:theo:product:item 1} can be found on page 245 of \cite{Embrechts:goldie:1980}. The proof of part~\ref{chap1:theo:product:item 4} is indicated in Lemma 4.2 of \cite{jessen:mikosch:2006}. In the subsequent chapters we come across the results of part~\ref{chap1:theo: breiman} and part~\ref{chap1:theo:product:item 5}. Part~\ref{chap1:theo: breiman} is also known as Breiman's Theorem and was derived in \cite{breiman:1965}. Part~\ref{chap1:theo:product:item 5} was used in various places but an explicit proof appears in Lemma 5.1 of  \cite{maulik:zwart:2006}.

% \begin{proof}[Proof of Theorem~\ref{chap1:theo:product} ~\ref{chap1:theo: breiman}] Choose a $x_0$ as in Theorem~\ref{chap1:theorem:potter}~\ref{potter2}. Let the distribution function of $\Theta$ be denoted by $G$ and choose $x>x_0$. 
% \begin{align*}
%  \frac{\prob[\Theta X>x]}{\prob[X>x]}&=\int_0^\infty \frac{\prob[X>x/t]}{\prob[X>x]}dG(t)
% =\int_0^1+\int_1^{x/x_0}+\int_{x/x_0}^\infty\frac{\prob[X>x/t]}{\prob[X>x]}dG(t).
% \end{align*}
% 
% Now by Theorem~\ref{chap1:theorem:potter}~\ref{potter2} we have the first integral bounded by $t^{\alpha-\epsilon}$ and the second integral bounded by $t^{\alpha+\epsilon}$. For the third integral we have
% \begin{equation*}
%  \int_{x/x_0}^\infty\frac{\prob[X>x/t]}{\prob[X>x]}dG(t)\leq \frac{\prob[\Theta>x/x_0]}{\prob[X>x]}\leq \frac{\E[\Theta^{\alpha+\epsilon}]}{(\frac{x}{x_0})^{\alpha+\epsilon}\prob[X>x]}
% \end{equation*}
% and since $X$ is regularly varying $-\alpha$, the last term goes to $0$ and $x\rightarrow\infty$. Now apply Dominated Convergence Theorem to get the result.
% 
% 
% \end{proof}
% 
% \begin{proof}[Proof of Theorem~\ref{chap1:theo:product}~\ref{chap1:theo:product:item 5}] Denote the distribution function of $\Theta$ by $G$ and observe that,
% $$\prob[\Theta X>x]=\int_0^\infty \prob[X>x/t]dG(t)=\frac{1}{x^\alpha}\int_0^x t^{\alpha}dG(t)+\prob[\Theta>x].$$
% So we have that $x^\alpha\prob[\Theta X>x]=\int_0^x t^{\alpha}dG(t)+x^\alpha\prob[\Theta>x]$. Now the result follows since $\E[\Theta^\alpha]<\infty$ implies $x^\alpha\prob[\Theta>x]\rightarrow0$. 
% \end{proof}
% 
\section{Class of heavy tailed distributions and its subclasses}\label{chap1:sec:heavy tailed}

In Lemma~\ref{lem:one large jump} we saw that the sum of random variables with regularly varying tails satisfy the principle of one large jump. This property is in general exhibited by a larger class. The random variables which satisfy the principle of one large jump are known as the subexponential random variables. We now give a formal definition of the subexponential distribution functions and look into some of their properties. Throughout this section we assume that the distribution functions has support unbounded above, that is, $\overline{F}(x)>0$ for all $x$.

\begin{definition}
\label{chap1:def:subexponential}
 A distribution function $F$ on $(0,\infty)$ is called \textit{subexponential} if for all $n\geq 2$,
\begin{equation}
\label{chap1:eq:subexponential}
 \overline{F^{*n}}(x)\sim n\overline{F}(x) \quad \text{as}\quad x\rightarrow\infty.
\end{equation}
Here $F^{*n}$ denotes the $n$-fold convolution power of $F$. The class of subexponential distribution functions is denoted by $\mathcal S$.
\end{definition}

By Lemma~\ref{lem:one large jump} any distribution function with regularly varying tail satisfies the above equation~\eqref{chap1:eq:subexponential} and hence they form a subclass of the subexponential distributions. Some other examples are

\begin{enumerate}
 \item Lognormal: $\prob[X>x]=\prob[e^{\mu+\sigma N}>x]$, $\mu\in \mathbb R$, $\sigma>0$ with $N$ as standard normal random variable.
\item Weibull: $\prob[X>x]=e^{-ax^\alpha}$, $a>0$ and $0<\alpha<1$.
\end{enumerate}

As stated above one of the most important features of the subexponential distributions is the principle of single large jump, which can be briefly summarized as follows: if $X_1,\ldots, X_n$ are i.i.d.\ random variables with subexponential distribution functions, then the probability that their sum will be large is of the same order as that their maximum will be large. We formalize this in the following proposition.
\begin{proposition}[Principle of one large jump] \label{prop: one jump}
If $X_1, \ldots, X_n$ are i.i.d.\ random variables with subexponential distribution functions, then as $x\to\infty$,
$$\prob[X_1+\ldots+X_n>x] \sim \prob[X_1\vee\ldots\vee X_n>x].$$
\end{proposition}
\begin{proof}
Observe that, by the subexponential property,  we have, as $x\to\infty$,
$$\prob[X_1+\cdots+X_n>x] = \ol{F^{*n}}(x) \sim n \ol F(x) \sim 1 - (\ol F(x))^n = \prob[X_1+\cdots+X_n>x].$$
\end{proof}
The principle of one large jump is very important in modeling ruins under catastrophic events. While in the usual scenarios, individually small shocks build up to create a large aggregate shock, in catastrophic models, one single large shock is the cause of ruin. Such a property makes the subexponential distributions an ideal choice for modelling ruin and insurance problem and has caused wide interest in the probability literature \citep[cf.][]{embrechts:kluppelberg:mikosch:book, rolski:teugels:book}. Historically, this class was first studied by \cite{chistyakov:1964} where he showed an application to branching processes. The connection with branching process was later elaborately studied by \cite{chover:ney:wainger:1973a,chover:ney:wainger:1973b}.

The above definition is also closely related to two other concepts originating in reliability theory, namely hazard function and hazard rate.
\begin{definition} \label{defn: hazard}
For a distribution function $F$, we define the \textit{hazard function} $$R(x)= - \log \overline F(x).$$ If, further, the distribution function is absolutely continuous with density function $f$, we define the \textit{hazard rate} $r$ as the derivative of $R$. In particular, we have $$r(x)=\frac{f(x)}{\ol F(x)}.$$
\end{definition}
Since we assume that the distribution function $F$ has support unbounded above, the hazard function is well defined.

 As shown in the subexponentiality of random variables with regularly varying tails, the lower bound is almost trivial and does not require the regularly varying property. The next Lemma provides a sufficient condition on subexponentiality. The proof of the next result appears as Lemma 1.3.4 of \cite{embrechts:kluppelberg:mikosch:book}.

\begin{lemma}
\label{chap1:lem:upper bound subexp condition}
 If $$\limsup_{x\rightarrow\infty}\frac{\overline{F^{*2}}(x)}{\overline{F}(x)}\leq 2,$$  then $F$ is subexponential.
\end{lemma}

\cite{pitman:1980} provided a necessary and sufficient condition for identifying subexponential distribution functions.
\begin{theorem} \label{thm: subexp char}
Let $F$ be an absolutely continuous distribution function supported on the nonnegative real line with hazard rate function $r$, which is eventually non-increasing and goes to $0$. Then $F$ is a subexponential distribution function if and only if 
$$\lim_{x\to\infty} \int_0^x \e^{y r(x)} F(dy) = 1.$$
If $R$ denotes the hazard function, then a sufficient condition for the subexponentiality of $F$ is the integrability of $\exp(x r(x) - R(x)) r(x)$ over the nonnegative real line.
\end{theorem}
An application of the sufficient part of the theorem shows that Weibull distribution with shape parameter $\alpha<1$ is subexponential. Since the definition of subexponential property does not depend on the scale factor $a$, we assume  $a=1$. Then $\ol F(x) = \exp(-x^\alpha)$ for $x>0$ giving $R(x) = x^\alpha$ and $r(x)=\alpha x^{\alpha-1}$. Since $\alpha<1$, the hazard function decreases to $0$. So $\exp(x r(x) - R(x)) r(x) = \alpha x^{\alpha-1} \exp(-(1-\alpha) x^\alpha)$, which is integrable over the positive real line, as $0<\alpha<1$.

Although the definition of subexponential distributions require the equation~\eqref{chap1:eq:subexponential} to hold for all $n$ (or, for $n=2$ in Lemma~\ref{chap1:lem:upper bound subexp condition}), but  was shown in \cite{Embrechts:goldie:1980} that it is enough to check for some $n\geq 2$.

\begin{proposition}[\citealp{embrechts:kluppelberg:mikosch:book}, Lemma A3.14] \label{prop: conv root subexp}
Let $F$ be a distribution function supported on the nonnegative real line, such that, for some $n\in\mathbb N$, we have,
$$\limsup_{x\to\infty} \frac{\ol{F^{*n}}(x)}{\ol F(x)} \le n.$$ 
Then $F$ is  a subexponential distribution function.
\end{proposition}
% \begin{proof}
% Observe that, for any natural number $m$,
% \begin{equation} \label{eq: tail conv}
% \ol{F^{*(m+1)}}(x) = 1 - F^{*(m+1)}(x) = 1 - \int_0^x F(x-t) F^{*m}(dt) = \ol{F^{*m}}(x) + \int_0^x \ol F(x-t) F^{*m}(dt)
% \end{equation}
% So we get from~\eqref{eq: tail conv},
% $$\ol{F^{*(m+1)}}(x) = \ol{F^{*m}}(x) + \int_0^x \ol F(x-t) F^{*m}(dt) \ge \ol{F^{*m}}(x) + \ol F(x) F^{*m}(x)$$
% and dividing by $\ol F(x)$ and taking limit superior we have,
% $$\limsup_{x\to\infty} \frac{\ol{F^{*(m+1)}}(x)}{\ol F(x)} \ge \limsup_{x\to\infty} \frac{\ol{F^{*m}}(x)}{\ol F(x)} + 1.$$
% Hence, under the hypothesis of the Proposition we obtain $\limsup_{x\to\infty} \frac{\ol{F^{*2}}(x)}{\ol F(x)} \le 2$, which gives the required result from Lemma~\ref{chap1:lem:upper bound subexp condition}.
% \end{proof}
\cite{chistyakov:1964} showed that the subexponential distributions form a  part of larger class of distributions which are called heavy tailed distributions. We now define heavy tailed distributions and analogously light tailed distributions.
\begin{definition}\label{chap1:heavy:def}
 We say a random variable $X$ is \textit{heavy tailed} if, for all $\lambda>0$, $\E[e^{\lambda X}]=\infty$, or equivalently, if for all $\lambda>0$, $e^{\lambda x}\prob[X>x]\rightarrow\infty$ as $x\rightarrow\infty$.
\end{definition}

 Also we give a definition for light tailed distribution in view of the above definition. 
 \begin{definition} 
A random variable $X$ is called \textit{light tailed} if $\E[e^{\lambda X}]<\infty$ for some $\lambda>0$.  
 \end{definition}
 For light tailed distributions on $[0,\infty)$ we have all moments finite.

Clearly, Definition~\ref{chap1:heavy:def} requires that $\ol F(x)>0$ for all $x\in\mathbb R$, which is the standing assumption. The definition requires that the tail of the distribution function decays slower than any exponential function. This is the reason that sometimes such distribution functions are also called subexponential distribution functions. However, as per the convention introduced by the seminal work of \cite{teugels:1975}, we have reserved that terminology for the class defined in Definition~\ref{chap1:def:subexponential}.  The heavy-tailed property, the hazard function and the decay of the tail are related in the following theorem:
\begin{theorem}[\citealp{foss:korshunov:zachary:2009}, Theorem 2.6]\label{thm: heavy tail}
For a distribution function $F$ with support unbounded above, the following statements are equivalent:
\begin{enumerate}
  \item $F$ is a heavy-tailed distribution. \label{item: ht1}
  \item $\liminf_{x\to\infty} R(x)/x = 0$. \label{item: ht2}
  \item For all $\lambda>0$, $\limsup_{x\to\infty} \e^{\lambda x} \overline F(x) = \infty$. \label{item: ht3}
\end{enumerate}
\end{theorem}

Definition~\ref{chap1:def:subexponential} defines subexponential distribution functions, when they are supported on the nonnegative real line only. If we allow the distribution function to be supported on the entire real line,~\eqref{chap1:eq:subexponential} may hold even for light-tailed distribution. See Example~3.3 of \cite{foss:korshunov:zachary:2009} for one such distribution function. One easy way to extend the notion to all distribution functions is to restrict the distribution function to the nonnegative real line by considering
$$F_+(x)=
\begin{cases}
  F(x), &\text{if $x\ge 0$,}\\
  0, &\text{otherwise}
\end{cases}$$
and requiring~\eqref{chap1:eq:subexponential} to hold for $F_+$.

The class of subexponential distributions satisfies the important property of being long tailed and long tailed distributions in turn have the property of being heavy tailed. 
\begin{definition}
\label{chap1:def:long tailed}
A distribution function $F$ on $\mathbb R$ is called \textit{long tailed}, if $\overline{F}(x)>0$ for all $x$ and, for any real number $y$,
\begin{equation}
\label{eq: longtail def}
 \overline{F}(x-y)\sim \overline{F}(x) \quad \text{as} \quad x\rightarrow\infty.
\end{equation}
The class of long-tailed distributions is denoted by $\mathcal L$.
\end{definition}

Since the function $\ol F$ is monotone, it can be easily checked that the convergence is uniform in $y$ on every compact interval.

\begin{remark}
\label{chap1:rem:long:slowly varying}
 It is important to note that the long tailed distributions are related to the slowly varying functions in the following manner
\begin{equation}
 F\in \mathcal L \quad \text{if and only if} \quad \ol F(\ln (\cdot))\in RV_0.
\end{equation}

\end{remark}

% 
% \begin{proof}
% \ref{item: ht1} $\Rightarrow$ \ref{item: ht2}: Suppose not. Then there exists $\epsilon>0$ and $x_0<\infty$, such that, for all $x>x_0$, we have $\ol F(x) < \e^{-\epsilon x}$. Then, for any $\lambda<\epsilon$, we have
% $$\int_{\mathbb R} \e^{\lambda x} F(dx) = \int_{-\infty}^{x_0} \e^{\lambda x} F(dx) + \e^{\lambda x_0} \ol F(x_0) + \lambda \int_{x_0}^\infty \int_{x_0}^x \e^{\lambda u} du F(dx).$$
% Clearly, the first two terms on the right side are finite. Interchanging the order of the integration, the last term on the right side becomes
% $\lambda \int_{x_0}^\infty \e^{\lambda u} \ol F(u) du \le \lambda \int_{x_0}^\infty \e^{-(\epsilon-\lambda) u} du < \infty$.
% 
% \ref{item: ht2} $\Rightarrow$ \ref{item: ht3}: Fix any $\lambda>0$. From~\ref{item: ht2}, we also have, for any $c>0$, $\liminf_{x\to\infty} (R(x)+\log c)/x = 0$. Then there exists a sequence $\{x_n\}$, such that $x_n\to\infty$ and $R(x_n) \le -\log c + \lambda x_n$ for all $n$. Thus, we have, simplifying, $\e^{\lambda x_n} \ol F(x_n) \ge c$. The result follows, since $c>0$ is arbitrary.
% 
% \ref{item: ht3} $\Rightarrow$ \ref{item: ht1}: Suppose not. Then for some $\lambda>0$, $c:=\int \e^{\lambda x} F(dx) < \infty$. Hence, for any $x>0$, we have $c > \int_x^\infty \e^{\lambda u} F(du) > \e^{\lambda x} \ol F(x)$, which contradicts~\ref{item: ht3}.
% \end{proof}

 \cite{chistyakov:1964} introduced the classes of subexponential, heavy tailed and long tailed distributions and showed the following containment.

\begin{proposition}[\citealp{chistyakov:1964}, Lemma~2 and Theorem~2] \label{prop: contain}
Any long-tailed distribution function is heavy-tailed. Any subexponential distribution function (supported on the nonnegative real line) is long-tailed.
\end{proposition}

The class of subexponential distributions has found wide application in different branches of probability theory other than branching process, for which it was originally introduced by \cite{chistyakov:1964}. The following result shows that if two distribution functions have tails of comparable order and one of them is subexponential, then the other is also subexponential.
\begin{definition} \label{def: tail equiv}
We say that two distribution functions $F$ and $G$ are \textit{tail equivalent}, if there exists a number $c\in(0,\infty)$, such that $\ol G(x) \sim c \ol F(x)$.
\end{definition}
The proof of the following result can be found as Lemma~A3.15 of \cite{embrechts:kluppelberg:mikosch:book}.
\begin{proposition} \label{prop: tail eq subexp}
If $F$ is a subexponential distribution function and $G$ is tail equivalent to $F$, then $G$ is also subexponential.
\end{proposition}
\begin{remark}
We can take $F$ and $G$ to be supported even on the entire real line. The subexponentiality will depend on the corresponding distribution functions $F_+$ and $G_+$, which will also be tail equivalent.
\end{remark}

Proposition~\ref{prop: tail eq subexp} and Theorem~\ref{thm: subexp char} together can be used to show that lognormal distribution is subexponential. While the tail of the lognormal distribution is not easy to understand, Mill's ratio can be used to easily approximate it by a nice nonincreasing continuous function. So we first consider a distribution function which has this tail and hence is tail equivalent to the lognormal distribution. The corresponding hazard rate and hazard function do not satisfy the sufficient integrability condition of Theorem~\ref{thm: subexp char} around $0$. However, we can still alter the distribution function in a neighborhood of $0$ to make things integrable, yet maintain tail equivalence, which is a behavior around infinity. Thus the new distribution function will be subexponential and hence, by Proposition~\ref{prop: tail eq subexp}, the lognormal distribution will also be subexponential.

Proposition~\ref{prop: tail eq subexp} shows that the class of subexponential distribution functions is closed under tail equivalence. Closure of the class of long-tailed and heavy-tailed distribution functions under tail equivalence is trivial. Similar questions are of general interest in this area of probability theory. One of such operations for the question of closure is convolution. Clearly if $F$ and $G$ are heavy-tailed, then so is $F*G$. Similar result is true for long-tailed distributions as well; see Section~2.7 of \cite{foss:korshunov:zachary:2009}. In fact, it is also proved there that if $F$ is long-tailed and $\ol G(x) = \lito(\ol F(x))$, then $F*G$ is also long-tailed. However, in general, the class of subexponential distributions is not closed under convolution; see \cite{leslie:1989} for a counterexample. Other interesting operations include finite mixture and product. In particular, if $F$ and $G$ are heavy-tailed (long-tailed respectively) distribution functions, then for any $p\in(0,1)$, the mixture distribution function  $pF+(1-p)G$ and $FG$ (the distribution function of the maximum of two independent random variables with distributions $F$ and $G$) are also heavy-tailed (long-tailed respectively). However, such closures fail for the class of subexponential distributions. The failures are not coincidental, but they are closely related. To state the corresponding result, we need to introduce the following notion from \cite{Embrechts:goldie:1980}.
\begin{definition} \label{def: max-sum}
Two distribution functions $F$ and $G$ are called \textit{max-sum equivalent} and is denoted by $F\sim_M G$, if $\ol{F*G}(x)\sim \ol F(x) + \ol G(x)$.
\end{definition}
Note that, if $X$ and $Y$ are two independent random variables with distribution functions $F$ and $G$ respectively, where $F\sim_MG$, then $\prob[X+Y>x]=\ol{F*G}(x)\sim\ol F(x) + \ol G(x) \sim\ol F(x) + \ol G(x) - \ol F(x) \ol G(x) = \prob[X\vee Y>x]$, which suggests the terminology. Also, if $F\sim_M F$, then $F$ is subexponential.

The following theorem relates the closure of the class of subexponential distribution functions under convolution, finite mixture and product with max-sum equivalence.
\begin{theorem}[\citealp{Embrechts:goldie:1980}, Theorem 2] \label{eq: subexp conv}
Let $F$ and $G$ be two subexponential distribution functions. Then the following statements are equivalent:
\begin{enumerate}
\item $F*G$ is subexponential.
\item $pF+(1-p)G$ is subexponential for all $p\in[0,1]$.
\item $pF+(1-p)G$ is subexponential for some $p\in(0,1)$.
\item $F\sim_MG$.
\item $FG$ is subexponential.
\end{enumerate}
\end{theorem}
Thus, the counterexample from \cite{leslie:1989} for closure under convolution will also work for finite mixture and product.

The product of independent random variables with regularly varying tails were studied by \cite{breiman:1965}. Further details are given in Theorem~\ref{chap1:theo:product}. The product of independent random variables with subexponential distribution functions is not as well behaved as the class of random variables with regularly varying tails. It was extensively studied in \cite{cline:somorodnitsky:1994}. We now state two results from \cite{cline:somorodnitsky:1994} without proofs. 

\begin{proposition}[\citealp{cline:somorodnitsky:1994}, Theorem 2.1]
 Assume $X$ and $Y$ to be independent positive random variables with distribution functions $F$ and $G$ respectively and let the distribution of $XY$ be denoted by $H$. Suppose $F$ is subexponential. 
\begin{enumerate}
\item Suppose there exists a function $a:(0,\infty)\to (0,\infty)$ satisfying:
\begin{enumerate}
 \item $a(t)\uparrow\infty$ as $t\rightarrow\infty$;
\item $\frac{t}{a(t)}\rightarrow\infty$ as $t\rightarrow\infty$;
\item $\ol F(t-a(t))\sim \ol F(t)$ as $t\rightarrow\infty$;
\item $\ol G(a(t))=\lito (\ol H(t))$ as $t\rightarrow\infty$.
\end{enumerate}
Then $H$ is subexponential.
\item If $Y$ is bounded random variable, then $H$ is subexponential.
\end{enumerate}
\end{proposition}

\subsection{Some other useful classes of distributions}

The light tailed counterparts of subexponential and long tailed classes are the classes $\mathcal S(\gamma)$ and $\mathcal L(\gamma)$ respectively, which we describe now. They will be useful in dealing with the weighted  sums of random variables with regularly varying tail in the Chapter~\ref{chap2}. In Remark~\ref{chap1:rem:long:slowly varying} we saw the relationship between the random variables with slowly varying  tails and random variables with long tailed distribution functions. In general the class of distribution having regularly varying tails of index $-\gamma$ with $\gamma\geq 0$ can be related to a class $\mathcal{L}(\gamma)$.

\begin{definition}
\label{chap1:def:lgamma}
 A distribution function $F$ on $(0,\infty)$ belongs to the class $\mathcal L(\gamma)$ with $\gamma\geq 0$, if for all real $u$,
\begin{equation}
 \label{chap1:eq:def: lgamma}
\lim_{x\rightarrow\infty}\frac{\ol F(x-u)}{\ol F(x)}=e^{\gamma u}.
\end{equation}
Alternatively, one can say $\ol F(\ln(\cdot))\in RV_{-\gamma}$ if and only if $F\in \mathcal L(\gamma)$.
\end{definition}

The analogous counterpart of the subexponential class  is the class $\mathcal S(\gamma)$.
\begin{definition}
\label{chap1:def:sgamma}
 A distribution function $F$ on $(0,\infty)$ belongs to the class $\mathcal S(\gamma)$ with $\gamma\geq 0$, if $F\in\mathcal L(\gamma)$ and
\begin{equation}
 \label{chap1:eq:def: sgamma}
\lim_{x\rightarrow\infty}\frac{\ol{ F*F}(x)}{\ol F(x)}=2 \int_0^\infty e^{\gamma y}F(dy).
\end{equation}
\end{definition}

When $\gamma=0$ we get back the long tailed and subexponential classes. Note that when $\gamma=0$ in Definition~\ref{chap1:def:sgamma},  the condition that $F\in \mathcal L(\gamma)$ is automatically satisfied due to Proposition~\ref{prop: contain}. These classes were initially introduced by \cite{chover:ney:wainger:1973a,chover:ney:wainger:1973b} in the theory of branching process and they were further studied  by \cite{kluppelberg:1988, kluppelberg:1989} . \cite{kluppelberg:1988, kluppelberg:1989} also studied the densities of the class $\mathcal S(\gamma)$. To describe them we now introduce two more classes $\mathcal S_d(\gamma)$ and $\mathcal S^*$.

\begin{definition}
 A function $f:\mathbb R\rightarrow \mathbb R_+$  such that $f(x)>0$ on $[A,\infty)$ for some $A\in \mathbb R_+$ belongs to the class $\mathcal S_d(\gamma)$ with $\gamma\geq 0$ if
%\begin{enumerate}[(i)]
\begin{eqnarray}
&(i)&\qquad     \lim_{x\rightarrow\infty}\int_0^x\frac{f(x-y)}{f(x)}f(y)dy= 2\int_0^\infty f(u)du<\infty;\nonumber\\ 
&(ii)&\qquad       \lim_{x\rightarrow\infty}\frac{f(x-y)}{f(x)}=e^{\gamma y} \quad \text{for all }y\in\mathbb R.\nonumber 
      \end{eqnarray}
%\end{enumerate}
\end{definition}

\begin{definition}
 A distribution function $F$ belongs to the class $\mathcal S^*$ if it has finite expectation $\mu$ and 
\begin{equation}
 \lim_{x\rightarrow\infty}\int_0^x\frac{\ol F(x-y)}{\ol F(x)}\ol F(y)dy= 2\mu.
\end{equation}

\end{definition}

The classes $\mathcal S_d(\gamma)$ and $S(\gamma)$ can be linked in the following way. For $f\in\mathcal S_d(\gamma)$ define a distribution concentrated on $(0,\infty)$ by 
$$F(x)=\frac{\int_0^xf(y)dy}{\int_0^\infty f(y)dy}.$$ Then $F\in \mathcal S(\gamma)$ \citep[see][Theorem 1.1]{kluppelberg:1989}. It was also shown in Theorem 3.2 of \cite{kluppelberg:1988} that $\mathcal S^*$ is a proper subclass of the subexponential distributions.

\begin{example}[Generalized inverse Gaussian distribution (GIGD), \citealp{kluppelberg:1989}]
For  $a>0$, $b\geq 0$ and $c<0$, the density of GIGD is given by 
$$f(x)=\left(\frac ba\right)^{c/2}\left(2K_c(\sqrt{ab})\right)^{-1}x^{c-1}\exp\left(-\frac12 (ax^{-1}+bx)\right), \quad x\in\mathbb R_+$$ 
 where $K_c$ is the modified Bessel function of third kind with index $c$. Then $f\in \mathcal S_d(b/2)$ and hence the distribution function $F\in \mathcal S(b/2)$.
\end{example}

In the next theorem we summarize the relationship between the classes described before. To describe the containment we need to define another class of distribution functions first.

\begin{definition}
 A distribution function on $[0,\infty)$ belongs to the class $\mathcal D$ of dominated variations if
\begin{equation}
 \limsup_{x\rightarrow\infty}\frac{\ol F(x)}{\ol F(2x)}<\infty.
\end{equation}

\end{definition}

\begin{theorem}[\citealp{goldie:1978, embrechts:omey:1984, kluppelberg:1988}]\hfill

\noindent
$(1)$ The class of distribution functions $RV_{-\alpha}$, $\mathcal D$, $\mathcal L$, $\mathcal S$ satisfy $$RV_{-\alpha}\subset \mathcal D \cap \mathcal L \subset \mathcal S \subset \mathcal L$$
and all the classes are contained in the class of heavy tailed distribution functions.\\
\noindent
$(2)$ If $F$ has finite expectation and 
$F\in D \cap \mathcal L$, then $F\in\mathcal S^*\subset \mathcal S$.

\end{theorem}

\section{Extreme value theory and regular variations}\label{chap1:sec:extreme values}

The regularly varying functions also play an important role in the theory of extreme values. In this section we briefly point out the role of regularly varying functions in the extreme  value theory.

Let $X_1, X_2,\cdots, X_n$ be i.i.d.\ random variables with a nondegenerate distribution function $F$ on $\mathbb R$. The study of $M_n=\max(X_1, X_2,\cdots, X_n)$ is often known as the extreme value theory. It is easy to see that if we denote $x_F$ as the right end point of $F$, that is, $x_F:= \sup \{x: F(x)<1\}$, then $M_n\rightarrow x_F$ in probability and since $M_n$ is increasing, we also have $M_n\rightarrow x_F$ almost surely. Since one does not arrive at anything deeply interesting in terms of almost sure convergence or convergence in probability so it is of major importance to study the distributional convergence of $M_n$.

\begin{definition}
\label{chap1:def:doa}
A distribution function $F$, belongs to the maximum domain of attraction of $G$,  if $G$ is nondegenerate and there exists $a_n>0$ and $b_n\in\mathbb{R}$ such that
\begin{equation}
\label{chap1:def:eq:doa}
\prob\left[\frac{M_n-b_n}{a_n}\leq x\right]=F^n(a_n x+b_n)\rightarrow G(x).
\end{equation}
If this happens then we write $F\in D(G)$.
\end{definition}

Now a major challenge is to classify the possible distributions $G$ that can arise in the above limit. The distribution function $G$ can be classified into three types, namely Fr\'echet, Weibull and Gumbel. The initial study was carried out by  \cite{fisher:tippett:1928} and \cite{gnedenko:1943} and later elaborated by \cite{dehaan:1970, dehaan:1971,balkema:dehaan:1972}. A nice treatment of this can be found in the monographs by \cite{Resnick:1987, Haan:Ferreira:2006, embrechts:kluppelberg:mikosch:book, galambos:extreme:book, teugels:book}. The major tool  used in the proof of the above result is the convergence of types. Since we shall use the result, we state it  for future use.

\begin{theorem}[Convergence of types]\label{theo:convergence of types}
Suppose $U(x)$ and $V(x)$ are two nondegenerate distribution functions. Suppose for $n\geq 1$, there exists distribution function $F_n$, scaling constants $\alpha_n, a_n>0$ and centering constants $\beta_n, b_n\in \mathbb R$ such that
$$F_n(a_n x+b_n)\rightarrow U(x) \qquad \text{and} \qquad F_n(\alpha_nx+\beta_n)\rightarrow V(x).$$

Then as $n\rightarrow\infty$,

$$\frac{\alpha_n}{a_n}\rightarrow A>0,\quad \frac{\beta_n-b_n}{a_n}\rightarrow B\in\mathbb R \quad \text{and} \quad V(x)=U(Ax+B).$$

\end{theorem}

\begin{definition}
 A nondegenerate distribution function $F$ is called \textit{max-stable} if for all $n\in \mathbb N$, there exists $a_n>0$, $b_n\in\mathbb R$ such that $F(x)^n=F(a_nx+b_n)$.
\end{definition}

The next theorem characterizes the class of extreme value distributions. This result was proved by \cite{fisher:tippett:1928} and~\cite{gnedenko:1943}.

\begin{theorem}
 Suppose there exists $a_n>0$, $b_n\in\mathbb R$, $n\geq 1$ such that
$$\prob\left[\frac{M_n-b_n}{a_n}\leq x\right]=F^n(a_n x+b_n)\rightarrow G(x)$$
where $G$ is nondegenerate, then $G$ can be suitable scaled and centered to be one of the following forms:
\begin{enumerate}
 \item Fr\'echet : \quad $\Phi_\alpha(x)=\begin{cases}
                                          0, &\text{if } x<0,\\
                                          \exp(-x^{-\alpha}), &\text{if } x>0;\\
                                         \end{cases}$
\item Weibull : \quad $\Psi_\alpha(x)=\begin{cases}
                                       \exp(-(-x)^{\alpha}, &\text{if } x<0,\\
                                        1, &\text{if } x>0;\\
                                      \end{cases}$
\item Gumbel : \quad $\Lambda(x)=\exp(-e^{-x}),$  $x\in\mathbb R$.

\end{enumerate}

\end{theorem}

These three types of distributions are known as the extreme value distributions. It is well known that the class of extreme value distributions is same as the class of max-stable distributions.  A different choice of centering gives another parametrization of the above limits called Generalized Extreme Value distributions. 

\begin{definition}[Generalized Extreme Value Distributions]
\label{def:gevd}
For any $\gamma\in\mathbb R$, define the generalized extreme value distribution function $G_\gamma$ as follows:
$$G_\gamma(x) =
\begin{cases}
\exp\left(-(1+\gamma x)^{-\frac1\gamma}\right), &\text{for $\gamma \neq 0$ and $1+\gamma x >0$},\\
\exp\left(-\e^{-x}\right), &\text{for $\gamma = 0$}.
\end{cases}$$

% \begin{align}
%  \text{For $\gamma\neq0$, }& G_\gamma(x)=\exp(-(1+\gamma x)^{-1/\gamma}), \quad \text{where } 1+\gamma x >0,\label{chap1:def:GEVD}\\
%  \text{for $\gamma=0$, }& G_0(x)=\exp(-e^{-x}).\nonumber
%  \end{align}
\end{definition}

Note that for $\gamma=0$, we have $G_0(x)=\Lambda(x)$. For $\gamma>0$, $1+\gamma x>0$ implies $x>-1/\gamma$ and after a shift it is of the same type as $\Phi_\alpha$ with $\alpha=1/\gamma$. When $\gamma<0$, $1-|\gamma|x>0$ implies $x<1/|\gamma|$ and again after a shift, this of the same type as $\Psi_\alpha$ with $\alpha=1/|\gamma|$. We often use this parametrization of the limits and write $F\in D(G_\gamma)$.

Note that as $1-\Lambda(x)\sim e^{-x}$, it has all moments finite and hence it is not heavy tailed. Now as $1-\Phi_\alpha(x)\sim x^{-\alpha}$, this is heavy tailed. For $\Psi_\alpha$ we have the right end point $x_F$ to be finite. At this stage we want to point out the relation between extreme value theory and regular variation. The next result not only gives the relation, but also gives the properties of the centering and scaling needed in~\eqref{chap1:def:eq:doa}. For stating the result, we need to introduce the notation for left continuous inverse of a nondecreasing function.

\begin{definition}
 \label{chap1:def:inverse}
Suppose $f$ is a nondecreasing function on $\mathbb R$. The (left continuous) inverse of $f$ is defined as $$f^{\leftarrow}(y)=\inf \{s: \ f(s)\ge y\ \}.$$
\end{definition}
 
Suppose $f(\infty)=\infty$ and $f\in RV_{\alpha}$ with $0<\alpha<\infty$, then it is known \citep[Proposition 0.8]{Resnick:1987} that $f^{\leftarrow}\in RV_{1/\alpha}$.

\begin{theorem}
 \label{chap1:theo:evt:rv}
A distribution function $F\in D(G_\gamma)$ if and only if
\begin{enumerate}
 \item ($\gamma>0$): $x_F=\infty$ and $\ol F\in RV_{-\frac1\gamma}$. In this case 
\begin{equation*}
 F^n(a_n x+b_n)\rightarrow \Phi_{\frac1\gamma}(x) \quad \text{with } a_n=\left(\frac1{\ol F}\right)^{\leftarrow}(n), \ b_n=0.
\end{equation*}

\item ($\gamma<0$): $x_F<\infty$ and $\ol F\left(x_F-\frac1x\right)\in RV_\frac1\gamma.$ In this case 

\begin{equation*}
 F^n(a_n x+b_n)\rightarrow \Psi_{-\frac1\gamma}(x) \quad \text{with } a_n=x_F-\left(\frac1{\ol F}\right)^{\leftarrow}(n), \ b_n=x_F.
\end{equation*}

\item ($\gamma=0$): $x_F\leq \infty$ and
\begin{equation}
 \label{chap1:theo:evt:rv:eq:gumbel}
\lim_{t\uparrow x_F}\frac{\ol F(t+xf(t))}{\ol F(t)}=e^{-x}, \quad x\in \mathbb R
\end{equation}
for some positive function $f$. If~\eqref{chap1:theo:evt:rv:eq:gumbel} holds for some  positive function $f$, then $\int_t^{x_F} \ol F(s)ds <\infty$ for $t<x_F$ and $$f(t)=\frac{\int_t^{x_F}\ol F(s)ds}{\ol F(t)}, \quad t<x_F. $$ In this case,
$$F^n(a_n x+b_n)\rightarrow \Lambda(x) \quad \text{with } b_n=\left(\frac{1}{\ol F}\right)^{\leftarrow}(n) \text{ and } a_n=f(b_n).$$

\end{enumerate}

\end{theorem}

\section{Summary of the thesis}\label{summary}
This thesis considers three problems where regularly varying functions play an important role. The problems are from classical as well as free probability theory. The first two problems described in Subsections~\ref{chap1:subsec1} and~\ref{chap1:subsec2} involve the tail behavior of randomly weighted sums and the tail behavior of products of random variables from conditional extreme value model. The third problem described in Subsection~\ref{chap1:subsec3} provides the extension of the notion of subexponentiality to the free probability theory. It shows that the probability distribution functions with regularly varying tails play as important a role in free probability theory as in classical probability.
\begin{enumerate}
\item Suppose $\{X_t\}_{t\geq 1}$ is a sequence i.i.d.\ random variables having regularly varying tails of index $-\alpha$ with $\alpha>0$ and $\{\Theta_t\}_{t\geq 1}$ is a sequence of positive random variables independent of $\{X_t\}_{t\geq 1}$. We obtain sufficient conditions, which ensure that $\xinf = \sum_{t=1}^\infty \Theta_t X_t$ has regularly varying tails,  under  reduced conditions involving the summability of only $\alpha$-th moments of $\{\Theta_t\}_{t\ge 1}$. Motivated by \cite{denisov:zwart:2007}, we reduce the moment conditions required in determining the series behavior.
 For a converse result to the problem above, we suppose that $\xinf=\sum_{t=1}^{\infty}\Theta_t X_t$ converges with probability $1$ and has regularly varying tail of index $-\alpha$, where $\alpha>0$. Also assume $\{\Theta_t\}_{t\geq 1}$ are positive random variables independent of the positive i.i.d.\ sequence $\{X_t\}_{t\geq1}$. We obtain sufficient moment conditions on $\{\Theta_t\}$, so that the regularly varying tail of $\xinf$ guarantees the same for $X_1$.
\item In much of the probability literature on the product of  random variables with regularly varying tails, the random variables are taken to be independent. We extend the results to a suitable dependence structure. The conditional extreme value models were introduced by \cite{heffernan:tawn:2004} and later elaborated upon by \cite{heffernan:resnick:2007} to model the multivariate regular variation and accommodate the notions of asymptotic independence and dependence. We explore the behavior of $XY$, when $(X,Y)$ follow the conditional extreme value model.
\item Free probability theory is a fast emerging area of interest and it is worthwhile to look into the role of regularly varying functions in the free probability theory. It is already known that they play a very crucial role in determining the behavior of domains of attraction of stable laws \citep{bercovici1999stable, bercovici2000free}. In classical probability theory, the subexponential random variables play an important role in the class of heavy-tailed random variables. We extend the notion of subexponentiality to the free probability setup. We show the class is nonempty. In fact, we show that it contains all random variables with regularly varying tails. In the process, we prove results giving the relation between the error terms in Laurent series expansions of Voiculescu and Cauchy transforms, which can be of independent interest.
\end{enumerate}

\begin{subsection}{Tail behavior of randomly weighted sums}\label{chap1:subsec1}
We consider a problem about the tail behavior of the randomly weighted sum and its converse. The details of this problem is available in Chapter~\ref{chap2} and is based on \cite{hazra:maulik:2010b}.

If $\{c_t\}_{t\geq 1}$ is a sequence of positive real numbers satisfying suitable summability conditions and $\{X_t\}$ are i.i.d.\ random variables with regularly varying tails of index $-\alpha$ with $\alpha>0$, then $\sum_{t=1}^{\infty}c_tX_t$ has regularly varying tails with same index \citep[see][Section~4.5, for details]{Resnick:1987}. One may replace the constant coefficients by a sequence of positive random variables $\{\Theta_t\}_{t\geq 1}$ independent of $\{X_t\}_{t\geq 1}$ to consider the tail behavior of $\xinf=\sum_{t=1}^{\infty}\Theta_tX_t$. Observe that the tail behavior of the product in one particular term can be controlled by the moment conditions on $\Theta_t$ according to Breiman's theorem. Recall from Theorem~\ref{chap1:theo:product}~\ref{chap1:theo: breiman} that Breiman's result gives: if $\E[\Theta^{\alpha+\epsilon}]<\infty$ for some $\epsilon>0$ and $X$ has regularly varying tail with index $-\alpha$ where $\alpha>0$ then $\Theta X$ also has regularly varying tail of index $-\alpha$. In fact in this case, $\prob[\Theta X>x] \sim \E[\Theta^\alpha] \prob[X>x]$ holds. When one considers the series $\xinf$, then both the small and the large values of $\Theta$ need to be managed. The large values are controlled by Breiman's result. \cite{resnick:willekens:1991} assumed the following moment conditions on $\{\Theta_t\}_{t\geq 1}$: 
\begin{enumerate}
\item When $0<\alpha<1$, then for some $\epsilon>0$,
\begin{equation}\label{chap1:condition1}
\sum_{t=1}^{\infty}\left(E[\Theta_t^{\alpha-\epsilon}+\Theta_t^{\alpha+\epsilon}]\right)<\infty.
\end{equation}
\item When $\alpha\geq1$, then for some $\epsilon>0$,
\begin{equation}\label{chap1:condition2}
\sum_{t=1}^{\infty}\left(E[\Theta_t^{\alpha-\epsilon}]^{\frac{1}{\alpha+\epsilon}}
+\E[\Theta_t^{\alpha+\epsilon}]^{\frac{1}{\alpha+\epsilon}}\right)<\infty.
\end{equation}
\end{enumerate}
Then they showed that $\xinf$ has regularly varying tail of index $-\alpha.$ Note that the conditions~\eqref{chap1:condition1} and~\eqref{chap1:condition2} also hold for $\alpha\ge 1$ and $0<\alpha<1$  respectively.
 \cite{zhang:weng:shen:2008} extended the result to random variables with extended regularly varying tails. Further review of the existing literature in the topic is available in Chapter~\ref{chap2}. It can be easily verified that the series converges almost surely under the above moment conditions \citep[see][]{jessen:mikosch:2006, mikosch:somorodnitsky:2000}.

In many cases, the existence of moments of $\Theta$ of order strictly greater that $\alpha$ becomes too strong an assumption in Breiman's result. In fact, the condition $\sum_{t=1}^{\infty}E[\Theta_t^{\alpha+\epsilon}]<\infty$ can become too restrictive for regularly varying tail property of $\xinf$.  For example, if we take a sequence of random variables such that $\sum_{t=1}^{\infty}\E[\Theta_t^{\alpha+\epsilon}]=\infty$, but $\sum_{t=1}^{\infty}\E[{\Theta_t}^{\alpha}]<\infty$ and $\{X_t\}$ to be standard Pareto with parameter $-\alpha$ where $0<\alpha<1$, then it can be shown that $\xinf$ still has regularly varying tail of index $-\alpha$. \cite{denisov:zwart:2007} relaxed the moment condition in Breiman's result by assuming the finiteness of $\alpha$-th moment of $\Theta$ alone. They also assumed the natural condition that $\Theta$ has lighter tail than $X$, that is, $\prob[\Theta>x]=\text{o} (\prob[X>x])$ as $x\rightarrow\infty$. They replaced the requirement of the existence of higher moments by some further sufficient conditions on the slowly varying part in $\prob[X>x]$. We exhibit that these conditions have a natural extension when we consider the randomly weighted sums and series.

We consider a sequence of identically distributed random variables $\{X_t\}_{t\geq 1}$ with regularly varying tails of index $-\alpha$, where $\alpha>0$. We reduce the independence assumption to that of pairwise asymptotically independence, that is, 
\begin{equation}\label{eq:pairwiseasympind}
\frac{\prob[X_t>x,X_s>x]}{\prob[X_s>x]}\rightarrow 0 \text{ as } x\rightarrow\infty \text{ for } s\neq t.
\end{equation}
 The nonnegativity of the summands can also be replaced by the weaker condition of the negligibility of the left tail compared to the right tail, that is, $\prob[X<-x]=\lito(\prob[X>x])$ as $x\rightarrow\infty$. We also consider another sequence of non-negative random variables $\{\Theta_t\}_{t\geq1}$ independent of $\{X_t\}$ such that  $\prob[\Theta_t>x]=\lito(\prob[X_1>x])$ for all $t\geq 1$ and satisfying :
\begin{enumerate}
    \item When $ 0 <\alpha<1$, \begin{equation} \label{chap1:alpha less than 1} \sum_{t=1}^{\infty} \E[{\Theta_t}^{\alpha}]<\infty.  \end{equation}
    \item When $1\leq\alpha<\infty$,  \begin{equation}\label{chap1:alpha greater than 1}\sum_{t=1}^{\infty} (\E[{\Theta_t}^{\alpha}])^{\frac{1}{\alpha+\epsilon}}<\infty  \mbox{~~~~~~~for some  $\epsilon > 0$.}\end{equation}
\end{enumerate}

Note again that the conditions~\eqref{chap1:alpha less than 1}  and~\eqref{chap1:alpha greater than 1} also hold for $\alpha\ge 1$ and $0<\alpha<1$  respectively.
In Chapter~\ref{chap2}, we show that under the natural uniform extensions of the conditions on $\{X_t, \Theta_t\}_{t\geq 1}$ given in \cite{denisov:zwart:2007}, we have
\begin{equation} \label{chap1:infinite asymptotics}
\prob\left[\max_{1\leq k <\infty}\sum_{t=1}^{k}\Theta_t X_t > x\right]\sim \prob[\xinf > x]\sim \prob[X_1>x]\sum_{t=1}^{\infty}\E[\Theta_t^{\alpha}].
\end{equation}
It should be noted above that in conditions (\ref{chap1:alpha less than 1}) and (\ref{chap1:alpha greater than 1}) we reduce the assumption of existence of moments of $\Theta$ of order strictly greater than $\alpha$ as well as the extra summability conditions of the moments.
%\end{subsection}

%\begin{subsection}{Problem 2: Converse problem for the tail behaviour of the randomly weighted sums}
% Let $\{X_t\}_{t\geq 1}$ be a  sequence of pairwise asymptotically independent and identically distributed positive random variables. If $X_1$ has regularly varying tail of index $-\alpha$ with $\alpha>0$ and if $\{\Theta_t, t\geq1\}$ is a positive sequence of random variables independent of $\{X_t\}$, then, from~\eqref{chap1:infinite asymptotics}, we have that under some moment conditions,
% \begin{equation}\label{chap1:tailcomparison1}
% \lim_{x\rightarrow\infty}\frac{\prob[\xinf>x]}{\prob[X_1>x]}=\sum_{t=1}^{\infty}\E[{\Theta_t}^{\alpha}].
% \end{equation}
Till now, we have discussed how the regular variation of the tail of $X_1$ forces the same for $\xinf$ under various weights and dependence structure.
\cite{JMRS:2009} discussed a converse problem. They considered a sequence of i.i.d.\ positive random variables $\{X_t\}_{t\geq 1}$ and a non-random sequence $\{c_t\}$ satisfying some summability assumptions, so that $\sum_{t=1}^{\infty}c_tX_t<\infty$ almost surely and has regularly varying tail of index $-\alpha$ with $\alpha>0$. Then they showed under suitable assumptions that $X_1$ has regularly varying tail with same index. Motivated by this converse problem, we obtain sufficient conditions, which  guarantee a regularly varying tail of $X_1$, whenever $\xinf$ has one. We show in Theorem~\ref{main result}, if  $\{X_t, t\geq1\}$ is a sequence of positive, identically distributed and pairwise asymptotically independent random variables,  $\{\Theta_t, t\geq1\}$ is a sequence of positive random variables independent of $\{X_t, t\geq 1\}$ satisfying the moment condition~\eqref{chap1:condition1} or \eqref{chap1:condition2} such that $\xinf=\sum_{t=1}^{\infty}\Theta_tX_t$ is finite almost surely and has regularly varying tail of index $-\alpha$ with $\alpha>0$, and
\begin{equation}\label{chap1:condition3} \sum_{t=1}^{\infty}\E[\Theta_t^{\alpha+i\phi}]\neq 0 \text{ for all }  \phi\in \mathbb{R},\end{equation}
 then $X_1$ has regularly varying tail of index $-\alpha$.
% and $\prob[\xinf>x]\sim \sum_{t=1}^{\infty}\E[{\Theta_t}^{\alpha}]\prob[X_1>x]$ as $x\rightarrow\infty$ holds.

The condition~\eqref{chap1:condition3} is necessary for the above theorem. In fact, in Theorem~\ref{thm: converse}, we show that whenever there exists a sequence of positive random variables $\{\Theta_t, t\geq1\}$  which for some $\alpha>0$, satisfies the appropriate moment condition \eqref{chap1:condition1} or \eqref{chap1:condition2},  but the condition \eqref{chap1:condition3} fails, that is,  $\sum_{t=1}^{\infty}\E[\Theta_t^{\alpha+i\phi_0}]=0$ for some $\phi_0$, then there exists a sequence of i.i.d.\ random variables $\{X_t\}$, whose common marginal distribution function does not have regularly varying tail, yet the series $\xinf=\sum_{t=1}^{\infty}\xtheta$ converges almost surely and $\xinf$ has regularly varying tail of index $-\alpha$.

\end{subsection}

\begin{subsection}{Products in the Conditional Extreme Value model}\label{chap1:subsec2}
The joint distribution of two random variables is said to be \textit{asymptotically independent}, if the suitably centered and scaled coordinatewise maximum of $n$ i.i.d.\ observations from the distribution has a non-degenerate limit, which is a product measure. It follows form Proposition 5.27 of \cite{Resnick:1987} that this  is equivalent to the pairwise asymptotic independence between the two marginals as described in \eqref{eq:pairwiseasympind}. However, this concept is too weak to conclude anything useful about the product of random variables, which is the main issue in this problem. So this concept was replaced by a stronger condition in \cite{Maulik:Resnick:Rootzen:2002}. They assumed that
\begin{equation}\label{chap1:asympind-strong}
t\prob\left[\left(\frac{X}{b(t)}, Y\right)\in\cdot\right]\vaguec(\nu\times G)(\cdot) \text { on } (0,\infty]\times[0,\infty],
\end{equation}
where $X$ and $Y$ are strictly positive random variables and $b(t)=\inf\{x:\prob[X>x]\leq 1/t\}$, $\nu$ is a Radon measure on $(0, \infty]$ and $G$ is a probability measure with $G(0,\infty)=1$.  Note that this implies that $X$ has regularly varying tail of index $-\alpha$, for some $\alpha>0$, $\nu(x,\infty)=cx^{-\alpha}$ for some $c\in(0,\infty)$ and $(X,Y)$ are asymptotically independent in the sense defined above. It was shown in \cite{Maulik:Resnick:Rootzen:2002} that if $(X,Y)$ has the above dependence structure~\eqref{chap1:asympind-strong}, then, under some moment conditions, the product has regularly varying tail, whose behavior is similar to that of the heavier of the two factors.

Suppose that $(X, Y)$ are multivariate regularly varying in the sense that there exists regularly varying functions $b(\cdot)$ and $a(\cdot)$ and a Radon measure $\mu(\cdot)$ such that $$t\prob\left[\left(\frac{X}{b(t)}, \frac{Y}{a(t)}\right)\in\cdot\right]\vaguec\mu(\cdot) \text{ on } [0,\infty]^{2}\setminus\{(0,0)\}.$$ If the limit measure $\mu$ satisfies $\mu((0, \infty]^2) > 0$, $(X, Y)$ are said to be \textit{asymptotically dependent}. \cite{Maulik:Resnick:Rootzen:2002} showed that, for asymptotically dependent $(X, Y)$, the random variables $X$, $Y$ and $XY$ have regularly varying tails of indices $-\alpha$, $-\beta$ and $-\alpha\beta/(\alpha+\beta)$, for some $\alpha, \beta > 0$. Thus, the product behavior for jointly multivariate regularly varying random variables is in contrast to the case when $(X, Y)$ follow (\ref{chap1:asympind-strong}). This causes an interest about the behavior of the product of multivariate regularly varying random variables, which have other types of joint distribution.

The conditional extreme value model (CEVM) provides a rich source of dependence structure for multivariate regularly varying random variables. The model extends~\eqref{chap1:asympind-strong} to include the limit measures which are not in product form. This model was introduced as the usual ones in multivariate extreme value theory suffer a lot either from the presence of asymptotic independence or the absence of one or more components in domain of attraction of an univariate extreme value. The model was first proposed by \cite{heffernan:tawn:2004} and then further formalized and analyzed by \cite{heffernan:resnick:2007}. We now briefly describe the model.

Let $(X, Y)$ be a two-dimensional real-valued random vector and let $F$ denote the marginal distribution function of $Y$. $(X, Y)$ is said to be from a conditional extreme value model if the following conditions hold:
\begin{enumerate}
\item The marginal distribution function $F$ is in the domain of attraction of an extreme value distribution $G_{\gamma}$, for some $\gamma\in \mathbb{R}$ as in Definition~\ref{chap1:def:doa}.
\item There exists a positive valued function $\alpha$ and a real valued function $\beta$ and a non-null Radon measure $\mu$ on Borel subsets of $[-\infty,\infty] \times \Ebar$, where $\Ebar = \{ y\in \mathbb R: 1+\gamma y \ge 0\}$, such that
    \begin{enumerate}
      \item \label{chap1:basic1} $$t \prob \left[ \left(\frac{X-\beta(t)}{\alpha(t)}, \frac{Y-b(t)}{a(t)}\right) \in \cdot \right] \vaguec \mu(\cdot)$$ on $[-\infty, \infty] \times \Ebar$, and
      \item for each $y \in \Egamma$, $\mu ((-\infty, x] \times (y, \infty))$ is a non-degenerate distribution function in $x$.
    \end{enumerate}
\item The function $H(x)=\mu((-\infty,x]\times(0,\infty))$ is a probability distribution.
\end{enumerate}

The above conditions imply that if $(x,0)$ is a continuity point of $\mu(\cdot)$, then 
$$\prob\left[\frac{X-\beta(t)}{\alpha(t)}\leq x \mid Y>b(t)\right]\rightarrow H(x), \quad \text{as}\quad t\rightarrow \infty,$$ which suggests the name of this class of distribution functions.
\cite{heffernan:resnick:2007} showed that for a conditional extreme value model, there exists functions $\psi_1(\cdot)$, $\psi_2(\cdot)$ such that,
$$\lim_{t\rightarrow\infty}\frac{\alpha(tx)}{\alpha(t)}=\psi_1(x) \qquad \text{and} \qquad \lim_{t\rightarrow\infty}\frac{\beta(tx)-\beta(t)}{\alpha(t)}=\psi_2(x)$$
hold uniformly on compact subsets of $(0,\infty)$. We must also necessarily have, for some $\rho \in \mathbb R$, $\psi_1(x)=x^{\rho}, x>0$ and either $\psi_2$ is $0$ or, for some $k\in\mathbb R$, on $x>0$,
$$\psi_2(x) =
\begin{cases}
  \frac k\rho (x^\rho - 1), &\text{when $\rho \neq 0$},\\
  k \log x, &\text{when $\rho = 0$}.
\end{cases}$$

In Chapter~\ref{chap3}, we study the tail behavior of the product $XY$, when $(X, Y)$ belongs to the conditional extreme value model. The discussion in Chapter~\ref{chap3} is based on \cite{hazra:maulik:2011}. Throughout this thesis we assume $(\psi_1(x),\psi_2(x))\neq(1,0)$.  More precisely, we considered the following cases:
\begin{enumerate}
\item When $\gamma>0$ and $\rho>0$, then under some tail conditions $XY$ has regularly varying tail of index $-1/{(\gamma+\rho)}$. This situation was similar to the case when $(X,Y)$ satisfy multivariate regular variation with asymptotic dependence.
\item When $\gamma<0$ and $\rho<0$, the situation is more complicated. Here $b(t)\rightarrow b(\infty)<\infty$, and $\beta(t)\rightarrow\beta(\infty)<\infty$ as $t\rightarrow\infty$, and the regularly varying property of $Y$ occurs at $\beta(\infty)$. We further divided this into following subcases:
  \begin{enumerate}
  \item Suppose $\gamma \le \rho$, $\beta(\infty)>0$ and $b(\infty)>0$ and also assume $X$ and $Y$ are strictly positive. Then under some moment conditions on $X$, we have $(b(\infty)\beta(\infty)-XY)^{-1}$ has regularly varying tail of index $-{1}/{|\rho|}$.
  \item If $\beta(\infty)=0=b(\infty)$ and $X$ and $Y$, then $(XY)^{-1}$ has regularly varying tail of index ${-{1}/{(|\gamma|+|\rho|)}}$.
  \item Suppose the $\beta(\infty)=0$ and $b(\infty)=1$ with $Y>0$. Then $-(XY)^{-1}$ has regularly varying tail of index $-{1}/{|\rho|}$.
  \item If both $\beta(\infty)$ and $b(\infty)$ are strictly less than zero, then $(XY-\beta(\infty)b(\infty))^{-1}$ has regularly varying tail of index ${-1/|\rho|}$.
  \end{enumerate}
  \item Let $\rho>0$ and $\gamma<0$ and assume that $b(\infty)>0$ and $Y>0$. The product $XY$ has regularly varying tail of index $-{1}/{|\gamma|}$, provided some moment conditions on $X$ gets satisfied.
\item If $\rho<0$ and $\gamma>0$ then  $XY$ has a regularly varying tail of index $-1/\gamma$.
\end{enumerate}

\end{subsection}

\begin{subsection}{Sums of free random variables with regularly varying tails}\label{chap1:subsec3}
Free probability theory is the non-commutative analogue of the classical probability theory. In free probability theory, the notion of independent random variables is replaced by freely independent operators on some suitable Hilbert space. A non-commutative probability space is a pair $(\mathcal{A}, \tau)$, where $\mathcal{A}$ is a unital complex algebra and $\tau$ is a linear functional on $\mathcal{A}$ satisfying $\tau(1)=1$. A non-commutative analogue of independence, based on free products, was introduced by \cite{voiculescu1986addition}. A family of unital subalgebras $\{\mathcal{A}_i\}_{i\in I}\subset\mathcal{A}$ is called a family of {\it free algebras}, if $\tau(a_1\cdots a_n)=0$, whenever $\tau(a_j)=0,\  a_j\in \mathcal{A}_{i_j}$ and $i_j\neq i_{j+1}$ for all $j$. The above setup is suitable for dealing with bounded random variables. In order to deal with unbounded random variables we need to consider a tracial $W^*$-probability space $(\mathcal{A}, \tau)$, where $\mathcal{A}$ is a von Neumann algebra and $\tau$ is a normal faithful tracial state. A self-adjoint operator $X$ is said to be affiliated to $\mathcal{A}$, if $u(X)\in\mathcal{A}$ for any bounded Borel function $u$ on the real line $\mathbb{R}$. A self-adjoint operator affiliated to $\mathcal{A}$ will also be called a random variable. The notion of freeness was extended to this context by~\cite{bercovici1993free}. The self-adjoint operators $\{X_i\}_{1\leq i\leq k}$ affiliated with a von Neumann algebra $\mathcal{A}$ are called freely independent, or simply free, if and only if the algebras generated by the operators, $\{f(X_i): \text{ $f$ bounded measurable} \}_{1\leq i\leq p}$ are free. Given a self-adjoint operator $X$ affiliated with $\mathcal{A},$ the law of $X$ is the unique probability measure $\mu_X$ on $\mathbb{R}$ satisfying $$\tau(u(X))=\int_{-\infty}^{\infty}u(t)d\mu_X(t)$$ for every bounded Borel function $u$  on $\mathbb{R}$. If $e_A$ denote the projection valued spectral measure associated with $X$ evaluated at the set $A$, then it is easy to see that $\mu_{X}(-\infty,x]=\tau(e_{(-\infty,x]}(X)).$ The distribution function of $X$, denoted by ${F}_X$, is given by ${F}_{X}(x)=\mu_X(-\infty,x]$.

The measure $\mu*\nu$ is the classical convolution of the measures $\mu$ and $\nu$, which also corresponds to the probability law of random variable $X+Y$ where $X$ and $Y$ are independent and have laws $\mu$ and $\nu$ respectively. Now it is well known that given two measures $\mu$ and $\nu$, there exists a tracial $W^*$-probability space $(\mathcal{A}, \tau)$ and free random variables $X$ and $Y$ affiliated to $\mathcal{A}$ such that $\mu$ and $\nu$ are the laws of $X$ and $Y$ respectively. Now $\mu\boxplus\nu$ denotes the law of $X+Y$. This is well defined as the law of $X+Y$ does not depend on the particular choices of $X$ and $Y$, except for the fact that $X$ and $Y$ have laws $\mu$ and $\nu$ respectively and they are free. The free convolution was first introduced in \cite{voiculescu1986addition} for compactly supported measures, extended by \cite{maassen1992addition} to measures with finite variance and by \cite{bercovici1993free} to measures with unbounded support.

The relationship between $*$ and $\boxplus$ convolution is very striking. There are many similarities, for example, in characterizations of infinitely divisible and stable laws \citep{bercovici1999stable, bercovici2000free}, weak law of large numbers \citep{bercovici1996law} and central limit theorem \citep{pata1996central,Voiculescu:1985,maassen1992addition}. Analogues of many other classical theories have also been derived. In recent times, links with extreme value theory \citep{arous2006free, arous2009free} and de Finetti type theorems \citep{banica907finetti} are of major interest. However, there are differences as well. Cramer's theorem \citep{bercovici1995superconvergence} and Raikov's theorem \citep{benaych2004failure} fail to extend to the non-commutative setup. Further details about free probability and free convolution is provided in Chapter~\ref{free sub chapter}.

% We are interested in studying some heavy tailed properties of the distributions under non-commutative setup. A measure $\mu$ on $(0,\infty)$ is said to be \textit{subexponential} if $\mu^{*n}(x,\infty)\sim n\mu(x,\infty)$ as $x\rightarrow\infty$ or all and, equivalently, some $n$. A random variable $X$ with law $\mu$ is said to be subexponential, if $\mu$ satisfies the above property. Random variables with regularly varying tails of index $-\alpha$, $\alpha\geq 0$ form a large subclass of the subexponential random variables. Weibull distribution with shape parameter less than $1$ and lognormal distribution are some other well known examples of subexponential random variables. We extend the definition of subexponentiality in the free setup.
In Chapter~\ref{free sub chapter}, we study some heavy tailed properties of the distributions under non-commutative setup. \cite{hazra:maulik:2010a} extended the definition of subexponentiality given in Definition~\ref{chap1:def:subexponential} to the free setup.
\begin{definition}\label{def: free sub}
A probability measure $\mu$ on $(0,\infty)$ is said to be \textit{free subexponential} if for all $n$,
\begin{equation}\label{chap1:freesub}
\mu^{\boxplus n}(x,\infty)=\underbrace{(\mu\boxplus\cdots\boxplus\mu)}_{n-\text{times}}(x,\infty)\sim n\mu(x,\infty), \ \ \text{as}\ \ x\rightarrow\infty.
\end{equation}
\end{definition}

In Theorem~\ref{main theorem-1}, we show that if ${\mu}$ has regularly varying tail of index ${-\alpha}$ with $\alpha>0$, then ${\mu}$ is free subexponential. Thus probability measures with regularly varying tails not only form a subclass of the classical subexponential probability measures, but they also form a subclass of the free subexponential. Theorem~\ref{main theorem-1} and the related results in Chapter~\ref{free sub chapter} are based on~\cite{hazra:maulik:2010a}.

 The above definition can be rewritten in terms of distribution functions as well. A distribution function $F$ is called free subexponetial if for all $n \in \mathbb N$, $\overline {F^{\boxplus n}}(x) \sim n \overline F(x)$ as $x\to\infty$. A random element $X$ affiliated to a tracial $W^*$-probability space is called free subexponential, if its distribution is so. 
%One immediate consequence of the definition of free subexponentiality is the principle of one large jump.
% 
% \cite{arous2006free} showed that for two distribution functions $F$ and $G$, there exists a unique measure $F\boxv G$, such that whenever $X$ and $Y$ are two free random elements on a tracial $W^*$-probability space, $F\boxv G$ will become the distribution of $X\vee Y$. Here $X\vee Y$ is the maximum of two self-adjoint operators defined using the spectral calculus via the projection-valued operators, see \cite{arous2006free} for details. \cite{arous2006free} showed that $F\boxv G(x) = \max((F(x)+G(x)-1),0)$, and hence $F^{\boxv n}(x) = \max((nF(x)-(n-1)),0)$. Then, we have, for each $n$, $\overline{F^{\boxv n}} (x) \sim n \overline F(x)$ as $x\to\infty$. Thus, by definition of free subexponentiality, we have
% \begin{proposition}[Free one large jump principle] \label{chap1:prop: free one large jump}
% Free subexponential distributions satisfy the principle of one large jump, namely, if $F$ is freely subexponential, then, for every $n$,
% $$\overline{F^{\boxplus n}}(x) \sim \overline{F^{\boxv n}}(x) \text{ as $x\to\infty$.}$$
% \end{proposition}

Due to the lack of coordinate systems and expressions for joint distributions in terms of measures, the proofs of the above results deviates from the classical ones. When dealing with convolutions in the free setup, the best way is to go via Cauchy and Voiculescu transforms. The proof of the above theorem involves obtaining a relationship between Cauchy and Voiculescu transforms. Since, probability distribution function with regularly varying tails do not have all moments finite, their Cauchy and Voiculescu transforms will have Laurent series expansions of only finite order. We obtain a relationship between the remainder terms in two expansions, which is new in literature and is an extension of \eqref{eq: phi mu} proved in \cite{bercovici1999stable}, where only one term expansion was considered. In Theorems~\ref{thm: error equiv}--\ref{thm: error equiv new} we obtain sharper results by considering higher order expansions of Cauchy and Voiculescu transforms of distribution functions with regularly varying tails. This relation can be of independent interest in free probability theory.

To study the relationship between the remainder terms in the expansion of Cauchy and Voiculescu transforms, we derive two interesting results from complex analysis. We consider Taylor series expansion of finite order of an analytic function, as well as the remainder term, which is assumed to have suitable regularly varying properties besides other regularity conditions. We show in Theorems~\ref{fraction-taylor} and~\ref{inverse-taylor} that such properties are inherited by the remainder terms of of the reciprocals and inverses of the analytic functions respectively. We also use Karamata's Tauberian theorems and other results to relate the remainder term in the expansion of Cauchy transform and the regularly varying tail of the distribution function.

 \end{subsection}

\cleardoublepage
\chapter{Tail behavior of randomly weighted sums}\label{chap2}

\begin{section}{Introduction}\label{section:intro}
 Let $\{X_t, t \geq 1\}$ be a sequence of identically distributed, pairwise asymptotically independent, cf.~\eqref{asymptotic independence}, random variables and $\{ \Theta_t, t\geq 1 \}$ be a sequence of positive random variables independent of the sequence $\{X_t, t \geq 1\}$. We shall discuss the tail probabilities and almost sure convergence of $\xinf=\sum_{t=1}^{\infty}\Theta_t X_t^{+}$ (recall that $X^+=\max\{0,X\}$)  and $\max_{1\leq k<\infty}\sum_{t=1}^{k}\Theta_t X_t$, in particular, when $X_t$'s belong to the class of random variables with regularly varying tail and $\{\Theta_t, t\geq 1\}$ satisfies some moment conditions.
% We shall say that a random variable $X$ with distribution function $F$ has \textit{regularly varying tail of index $-\alpha$}, if $\overline F(x) := 1-F(x)$ is a regularly varying function of index $-\alpha$, that is, for any $t>0$, as $x\to\infty$, $\overline F(tx) \sim t^{-\alpha} \overline F(x)$. Here and later, for two positive functions $a(x)$ and $b(x)$, we write $a(x) \sim b(x)$ as $x\to\infty$, if $\lim_{x\to\infty} a(x)/b(x) = 1$. For $\alpha>0$, the convergence in the limit of the ratio of the tail probabilities is uniform in $t$ on the intervals of the form $[a,\infty)$ with $a>0$. Note that, we require the upper endpoint of the support of $X$ to be $\infty$.
 In recent times, there have been quite a few articles devoted to the asymptotic tail behavior of randomly weighted sums and their maxima. See, for example,\cite{resnick:willekens:1991, tang:chen:ng:2005, tang:wang:2006, zhang:weng:shen:2008, hult:somorodnitsky:2008}.

The question about the tail behavior of the infinite series $\xinf$ with non-random $\Theta_t$ and i.i.d.\ $X_t$ having regularly varying tails has been studied well in the literature, as it arises in the context of the linear processes, including ARMA and FARIMA processes. We refer to \cite{jessen:mikosch:2006} for a review of the results. The case, when $\Theta_t$'s are random, arises in various areas, especially in actuarial and economic situations and stochastic recurrence equation. For various applications, see \cite{hult:somorodnitsky:2008, zhang:weng:shen:2008}.

\cite{resnick:willekens:1991} showed that if $\{X_t\}$ is a sequence of i.i.d.\ nonnegative random variables with regularly varying tail of index $-\alpha$, where $\alpha>0$ and $\{\Theta_t\}$ is another sequence of positive random variables independent of $\{X_t\}$, the series $\xinf$ has regularly varying tail under the following conditions, which we shall call the RW conditions:
\renewcommand{\labelenumi}{(RW\arabic{enumi})}
\renewcommand{\theenumi}{RW\arabic{enumi}}
\begin{enumerate}
\item \label{RW1} If $0 <\alpha<1$, then for some $\epsilon\in(0,\alpha)$, $\sum_{t=1}^{\infty} \E[\Theta_t^{\alpha+\epsilon} + \Theta_t^{\alpha-\epsilon}]<\infty$.
\item \label{RW2} If $1\leq\alpha<\infty$, then for some $\epsilon\in(0,\alpha)$, $\sum_{t=1}^{\infty} (\E[\Theta_t^{\alpha+\epsilon} + \Theta_t^{\alpha-\epsilon}])^{\frac{1}{\alpha+\epsilon}}<\infty$.
\end{enumerate}
In this case, we have $\prob[\xinf>x] \sim \sum_{t=1}^\infty \E[\Theta_t^\alpha] \prob[X_1>x]$ as $x\to\infty$.
\begin{remark} \label{rem: RW}
Each of the RW conditions implies the other for the respective ranges of $\alpha$. In particular, if $0<\alpha<1$, choose $\epsilon^\prime<\epsilon$ such that $\alpha+\epsilon^\prime<1$. Note that
\begin{multline*}
 \sum_{t=1}^{\infty} \E[\Theta_t^{\alpha+\epsilon^\prime} + \Theta_t^{\alpha-\epsilon^\prime}] \le 2 \sum_{t=1}^{\infty} \E[\Theta_t^{\alpha+\epsilon^\prime} \mathbbm 1_{[\Theta_t\ge 1]} + \Theta_t^{\alpha-\epsilon^\prime} \mathbbm 1_{[\Theta_t< 1]}]\\
 \le 2 \sum_{t=1}^{\infty} \E[\Theta_t^{\alpha+\epsilon} \mathbbm 1_{[\Theta_t\ge 1]} + \Theta_t^{\alpha-\epsilon} \mathbbm 1_{[\Theta_t< 1]}] \le 2\sum_{t=1}^{\infty} \E[\Theta_t^{\alpha+\epsilon} + \Theta_t^{\alpha-\epsilon}]<\infty.
 \end{multline*}
Further, since $\alpha+\epsilon^\prime<1$, we also have $\sum_{t=1}^{\infty} (\E[\Theta_t^{\alpha+\epsilon^\prime} + \Theta_t^{\alpha-\epsilon^\prime}])^{\frac{1}{\alpha+\epsilon^\prime}}<\infty$. On the other hand, if $\alpha\ge 1$ and $\epsilon>0$, then $\alpha+\epsilon>1$ and the condition~\eqref{RW2} implies $\sum_{t=1}^{\infty} \E[\Theta_t^{\alpha+\epsilon} + \Theta_t^{\alpha-\epsilon}]<\infty$.
\end{remark}

\cite{zhang:weng:shen:2008} considered the tails of $\sum_{t=1}^n \Theta_tX_t$ and the tails of their maxima, when $\{X_t\}$ are pairwise asymptotically independent and have extended regularly varying and negligible left tail and $\{\Theta_t\}$ are positive random variables independent of $\{X_t\}$. The sufficient conditions for the tails to be regularly varying are almost similar.

While the tail behavior of $\xinf$ requires only the $\alpha$-th moments of $\Theta_t$'s, we require existence and summability of some extra moments in the RW conditions. Note that $\Theta_t^{\alpha+\epsilon}$ acts as a dominator for $[\Theta_t\ge1]$ and $\Theta_t^{\alpha-\epsilon}$ acts as a dominator for $[\Theta_t\le1]$. In some cases, the assumption of existence and summability of extra moments can become restrictive. We now consider such an example.

\begin{example} Consider $\{\Theta_t\}$ such that $\sum_{t=1}^{\infty}\E[\Theta_t^{\alpha+\epsilon}]=\infty$ for all $\epsilon>0$ but $\sum_{t=1}^{\infty} \E[\Theta_t^{\alpha}]<\infty$. (A particular choice of such $\{\Theta_t\}$, for $\alpha<1$ is as follows: $\Theta_t$ takes values ${2^t}/{t^{2/\alpha}}$ and $0$ with probability ${2^{-t\alpha}}$ and $1-2^{-t\alpha}$ respectively.) Also let  $\{X_t\}$ be i.i.d.\ Pareto with parameter $\alpha<1$, independent of $\{\Theta_t\}$. Then it turns out, after some easy calculations, that $\sum_{t=1}^{\infty}\Theta_t X_t$ is finite almost surely and has regularly varying tail of index $-\alpha$. 
 \end{example}
This leads to the question whether we can reduce the moment conditions on $\Theta_t$ to obtain the regular variation of the tail for $\xinf$.

The situation becomes clearer when we consider a general term of the series. It involves the product $\Theta_t X_t$. Using Breiman's theorem, see Theorem~\ref{chap1:theo:product}~\ref{chap1:theo: breiman},  the tail behavior of the product depends on the moments of $\Theta_t$. Recall that Breiman's theorem states, if $X$ is a random variable with regularly varying tail of index $-\alpha$ for some $\alpha> 0$ and is independent of a positive random variable $\Theta$ satisfying $\E[\Theta^{\alpha+\epsilon}]<\infty$ for some $\epsilon>0$, then,
\begin{equation}\label{breiman}
\lim_{x\to\infty} \prob[\Theta X > x]\sim \E[\Theta^{\alpha}] \prob[X>x].
\end{equation}
Note that, in this case, we work with a probability measure $\prob [\Theta_t \in \cdot]$, unlike in the problem of the weighted sum, where a $\sigma$-finite measure $\sum_{t=1}^\infty \prob [\Theta\in\cdot]$ is considered. In this case, we can consider the dominator as $1$ on $[\Theta\le1]$ instead of $\Theta^{\alpha-\epsilon}$, since $1$ is integrable with respect to a probability measure.

\cite{denisov:zwart:2007} relaxed the existence of $(\alpha+\epsilon)$ moments in Breiman's theorem to $\E[\Theta^{\alpha}] < \infty$. They also made the further natural assumption that $\prob[\Theta>x]=\lito(\prob[X>x])$. However, to obtain~\eqref{breiman}, the weaker moment assumption needed to be compensated. They obtained some sufficient conditions for~\eqref{breiman} to hold. In Sections~\ref{section:notation} and~\ref{sec:direct}, we find conditions similar to those obtained by \cite{denisov:zwart:2007}, which will guarantee the regular variation of $\xinf$.

In the above discussion, we considered the effect of the tail of $X_1$ in determining the tail of $\xinf$. However, the converse question is also equally interesting. More specifically, let $\{X_t\}$ be a sequence of identically distributed, asymptotically independent, positive random variables, independent of another sequence of positive random variables $\{\Theta_t\}$. As before, we define $\xinf=\sum_{t=1}^{\infty}\Theta_t X_t$. Assume that $\xinf$ converges with probability one and has regularly varying tail of index $-\alpha$ with $\alpha>0$. In Section~\ref{sec: converse} we obtain sufficient conditions which will ensure that $X_1$ has a regularly varying tail of index $-\alpha$ as well.

Similar converse questions regarding Breiman's theorem~\eqref{breiman} have recently been considered in the literature. Suppose $X$ and $Y$ are positive random variables with $\E[Y^{\alpha+\epsilon}]<\infty$ and the product $XY$ has regularly varying tail of index $-\alpha$, $\alpha>0$. Then it was shown in \cite{JMRS:2009} that $X$ has regularly varying tail of same index and hence~\eqref{breiman} holds. They have also obtained similar results for the weighted series, when the weights $\{\Theta_t\}$ are nonrandom. We shall extend this result for product to the case of randomly weighted series under appropriate conditions.

In Section~\ref{section:notation} we describe the conditions imposed by \cite{denisov:zwart:2007} and  study the tail behavior when finite weighted sums are considered. In Section~\ref{sec:direct} we describe the tail behavior of the series of randomly weighted sums. In Section~\ref{sec: converse} we consider the converse problem described above. We prove the converse result is true under the RW conditions and the extra assumption of nonvanishing Mellin transform. We also show the necessity of this extra assumption.
\end{section}

\begin{section}{Some preliminary results} \label{section:notation}

We call two random variables $X_1$ and $X_2$ to be \textit{asymptotically independent} if
\begin{equation} \label{asymptotic independence}
\lim_{x\rightarrow\infty}\frac{\prob[X_1 > x,X_2 > x]}{\prob[X_t > x]} = 0, \text{ for $t=1, 2$.}
\end{equation}
\begin{remark}
It is should be noted that the above notion of asymptotic independence is useful when $X_1$ and $X_2$ have similar tail behavior, that is, $\lim_{x\to\infty}\prob[X_1>x]/\prob[X_2>x]$ exists and is positive. In fact we use it for such cases only.
\end{remark}
See \cite{ledford:tawn:1996, ledford:tawn:1997} or Chapter~6.5 of \cite{resnick:2007} for discussions on asymptotic independence. Note that, we require $\overline F_t(x)>0$ for all $x>0$ and $t=1,2$. Observe that if $X_1$ and $X_2$ are independent, then they are also asymptotically independent. Thus the results under pairwise asymptotic independence condition continue to hold in the independent setup.

A random variable $X$ is said to have \textit{negligible left tail with respect to the right one}, if
\begin{equation} \label{right tail}
\lim_{x\rightarrow\infty}\frac{\prob[X<-x]}{\prob[X > x]}=0.
\end{equation}
Note that we require $\prob[X>x]>0$ for all $x>0$.

% The random variables with regularly varying tails will play a central role in this article. Note that, if $X$ has a regularly varying tail of index $-\alpha$, then $x^\alpha \prob[X>x]$ is a slowly varying function, that is, a regularly varying function with index $0$. 
Now, we gather a few results important for the present chapter. Our next Lemma states that by Karamata's representation, a slowly varying function $L$ can be one of the following form.

\begin{lemma}[\citealp{denisov:zwart:2007}, Lemma 2.1]
   Let $L$ be slowly varying. Then $L$ admits precisely one of the following four representations:
  \begin{enumerate}
  \item \label{L1}$L(x) = c(x)$,
 \item \label{L2}$L(x) = c(x)/\prob[ V > \log x]$,	
  \item \label{L3}$L(x) = c(x)\prob[ U > \log x]$,
  \item \label{L4}$L(x) = c(x)\prob[ U > \log x]/\prob[ V > \log x]$.
 \end{enumerate}
  In the above representations, $c(x)$ is a function converging to $c\in(0,\infty)$, and $U$ and $V$ are two independent long-tailed random variables with hazard rates converging to $0$.
\end{lemma}
 We shall refer to a slowly varying function $L$ as of type 1, type 2, type 3 or type 4, according to the above representations.

% \begin{proof}
%  By Karamata's representation for slowly varying function (see \eqref{chap1:eq:karamata rep}) we get
% $$L(x)=c(x)\exp\left(\int_{x_0}^x\frac{\epsilon(y)}{y}dy\right),$$
% where $c(\cdot)$ is a measurable nonnegative function such that $\lim_{x\rightarrow\infty}c(x)=c \in (0,\infty)$ and $\epsilon(x)\rightarrow 0$ as $x\rightarrow\infty$.
% Write $\epsilon(x)=\epsilon^+(x)-\epsilon^{-}(x)$ where $\epsilon^+, \epsilon^-$ are positive and negative parts of $\epsilon(x)$. Clearly both parts converge to zero. The first case is when $\int_1^x\frac{\epsilon^i(y)}y dy$ converges, where $i=+,-$. If this case holds, then we plug it in $c(x)$. So now we assume that either $\epsilon^i(x)=0$ or the corresponding integral diverges. This leads to the four cases.
% Suppose $\int_1^x \frac{\epsilon^+(y)}y dy$ diverges. Then there exists a long tailed random variable such that
% $$\exp(-\int_1^x \frac{\epsilon^+(y)}y dy)=\exp(-\int_0^{\log x} \epsilon^+(e^v)dv)=\prob[V>\log x].$$
% 
% Similarly $\epsilon ^-$ can be considered to get the other cases.
% \end{proof}

\cite{denisov:zwart:2007} introduced the following sufficient conditions on the slowly varying part $L$ of the regularly varying tail of index $-\alpha$ of a random variable $X$ with distribution function $\overline F(x)=x^{-\alpha} L(x)$ for Breiman's theorem~\eqref{breiman} to hold:
\let\myenumi\theenumi
\renewcommand{\labelenumi}{(DZ\arabic{enumi})}
\renewcommand{\theenumi}{DZ\arabic{enumi}}
\begin{enumerate}
   \item \label{DZ1} Assume $\lim_{x\rightarrow\infty}\sup_{y\in[1,x]} {L(y)}/{L(x)}:=D_1 < \infty$.
   \item \label{DZ2} Assume $L$ is of type 3 or type 4 and  $L(e^{x}) \in \mathcal{S}_{d}$.
   \item \label{DZ3} Assume $L$ is of type 3 or type 4, $U\in \mathcal{S^{*}}$ and $\prob[\Theta > x] = \lito(x^{-\alpha}\prob[U>\log x])$.
   \item \label{DZ4} When $\E[U]=\infty$ or equivalently $\E[X^\alpha]=\infty$, define $m(x)=\int_{0}^{x} v^{\alpha} F(dv) \to \infty$. Assume $\limsup_{x\rightarrow\infty} \sup_{\sqrt{x}\leq y \leq x} {L(y)}/{L(x)} :=D_2 < \infty$ and $\prob[\Theta > x] = \lito({\prob[X>x]}/m(x))$.
\end{enumerate}
\renewcommand{\theenumi}{\myenumi}
We shall refer to these conditions as the DZ conditions. A short discussion on several classes of distribution functions involved in the statements of the DZ conditions are available in Section~\ref{chap1:sec:heavy tailed}. For further discussions on the DZ conditions, we refer to \cite{denisov:zwart:2007}. Denisov and Zwart proved the following lemma:
\begin{lemma}[\citealp{denisov:zwart:2007}, Section~2] \label{breiman new}
Let $X$ be a nonnegative random variable with regularly varying tail of index $-\alpha$, $\alpha \geq 0$ and $\Theta$ be a positive random variable independent of $X$ such that $\E[\Theta^{\alpha}]<\infty$ and $\prob[\Theta > x]= \lito(\prob[X>x])$. If $X$ and ${\Theta}$ satisfy any one of the DZ conditions, then~\eqref{breiman} holds.
\end{lemma}

The next result shows that asymptotic independence is preserved under multiplication, when the DZ conditions are assumed.
\begin{lemma} \label{joint asymptotic independence}
Let $X_1, X_2$ be two positive, asymptotically independent, identically distributed random variables with common regularly varying tail of index $-\alpha$, where $\alpha>0$. Let $\Theta_1$ and $\Theta_2$ be two other positive random variables independent of the pair $(X_1, X_2)$ satisfying $\E[\Theta_t^{\alpha}] < \infty$,  $t=1,2$. Also suppose that $\prob[\Theta_t > x]= \lito(\prob[X_1>x])$ for $t=1,2$ and the pairs $(\Theta_1, X_1)$ and $(\Theta_2, X_2)$ satisfy any one of the DZ conditions. Then $\Theta_1 X_1$ and $\Theta_2 X_2$ are asymptotically independent.
\end{lemma}
\begin{proof}
Here and later $G$ will denote the joint distribution function of $(\Theta_1,\Theta_2)$ and $G_t$ will denote the marginal distribution functions of $\Theta_t$.
\begin{multline*}
\frac{\prob[\Theta_1 X_1 > x,\Theta_2X_2 > x]}{\prob[X_1>x]} = \iint_{u\leq v}+\iint_{u>v} \frac{\prob[X_1>x/u,X_2>x/v]}{\prob[X_1>x]} G(du,dv)\\
\leq \int_0^\infty \frac{\prob [X_1>x/v, X_2>x/v]}{\prob[X_1>x/v]} \frac{\prob[X_1>x/v]}{\prob[X_1>x]} (G_1+G_2)(dv).
\end{multline*}
The integrand converges to $0$. Also, the first factor of the integrand is bounded by $1$ and hence the integrand is bounded by the second factor, which converges to $v^\alpha$. Further, using Lemma~\ref{breiman new}, we have
\begin{multline*}
\int_0^\infty \frac{\prob[X_1>x/v]}{\prob[X_1>x]} (G_1+G_2)(dv) = \frac{\prob[\Theta_1X_1>x] + \prob[\Theta_2X_1>x]}{\prob[X_1>x]}\\
\to \E\left[ \Theta_1^\alpha \right] + \E\left[ \Theta_2^\alpha \right] = \int_0^\infty v^\alpha (G_1+G_2)(dv).
\end{multline*}
Then the result follows using Pratt's lemma, cf. \cite{pratt:1960}.
\end{proof}

The next lemma shows that if the left tail of $X$ is negligible when compared to the right tail then the product has also such a behavior.
\begin{lemma} \label{left tail negligibility}
Let $X$ have regularly varying tail of index $-\alpha$, for some $\alpha>0$ satisfying~\eqref{right tail} and $\Theta$ be independent of $X$ satisfying $\E[\Theta^{\alpha}]<\infty$ and $\prob[\Theta>x]= \lito(\prob[X > x])$. Also suppose that $(\Theta,X)$ satisfy one of the DZ conditions. Then, for any $u>0$,
$$\lim_{x\rightarrow\infty}\frac{\prob[\Theta X < -ux]}{\prob[\Theta X > x]}=0.$$
\end{lemma}
The proof is exactly similar to that of Lemma~\ref{joint asymptotic independence}, except for the fact that the first factor in the integrand is bounded, as, using~\eqref{right tail}, $\prob[X<-x]/\prob[X>x]$ is bounded. We skip the proof.

The following result from \cite{davis:resnick:1996} considers a simple case of the tail of sum of finitely many random variables.
\begin{lemma}[\citealp{davis:resnick:1996}, Lemma~2.1] \label{finite sum asymptotics1}
Suppose $Y_1, Y_2,\ldots,Y_k$ are nonnegative, pairwise asymptotically independent $($but not necessarily identically distributed$)$ random variables with regularly varying tails of common index $-\alpha$, where $\alpha>0$. If, for $t=1, 2, \ldots, k$,
${\prob[Y_t>x]}/\prob[Y_1>x]\rightarrow c_t$, then
$$\frac{\prob[\sum_{t=1}^k Y_t>x]}{\prob[Y_1>x]}\rightarrow c_1+c_2+\cdots+c_k.$$
\end{lemma}

We have the following corollary by applying Lemma~\ref{finite sum asymptotics1} with $Y_t=\Theta_tX_t^+$ and the modified Breiman's theorem in Lemma~\ref{breiman new} under the DZ conditions.
\begin{corollary} \label{fin sum cor}
Let $\{X_t\}$ be a sequence of pairwise asymptotically independent, identically distributed random variables with common regularly varying tail of index $-\alpha$, where $\alpha>0$, which is independent of another sequence of positive random variables $\{\Theta_t\}$ satisfying $\E[\Theta_t^\alpha]<\infty$, for all $t$. Also assume that, for all $t$, $\prob [\Theta_t > x] = \lito(\prob[X_1>x])$ and the pairs $(\Theta_t, X_t)$ satisfy one of the DZ conditions. Then we have
$$\prob\left[\sum_{t=1}^k\Theta_tX_t^+>x\right]\sim \prob[X_1>x]\sum_{t=1}^k\E[\Theta_t^{\alpha}].$$
\end{corollary}

Using Lemmas~\ref{breiman new}--\ref{finite sum asymptotics1} and Corollary~\ref{fin sum cor}, we obtain the following result, which is an extension of Theorem 3.1(a) of \cite{zhang:weng:shen:2008}. (Note that the proof of Theorem~3.1(a) of \cite{zhang:weng:shen:2008} require only the results obtained in Lemmas~\ref{breiman new}--\ref{finite sum asymptotics1} and Corollary~\ref{fin sum cor}.) 
\begin{proposition} \label{finite sum}
Let $\{X_t\}$ be a sequence of pairwise asymptotically independent, identically distributed random variables with common regularly varying tail of index $-\alpha$, for some $\alpha>0$ satisfying~\eqref{right tail}, which is independent of another sequence of positive random variables $\{\Theta_t\}$. Further assume that, for all $t$, $\prob[\Theta_t>x]= \lito(\prob[X_1>x])$ and $\E[\Theta_t^{\alpha}]<\infty$. Also assume that the pairs $(\Theta_t, X_t)$ satisfy one of the DZ conditions. Then,
$$\prob\left[\max_{1\leq k \leq n}\sum_{t=1}^{k}\Theta_t X_t > x\right] \sim \prob\left[\sum_{t=1}^{n}\Theta_t X_t^{+}> x\right] \sim \prob[X_1>x] \sum_{t=1}^{n}\E[\Theta_t^{\alpha}].$$
\end{proposition}

\begin{proof}[Proof of Propositon~\ref{finite sum}] The proof is similar to that of Theorem 3.1(a) of \cite{zhang:weng:shen:2008}. We provide a sketch for the completeness.
 Since $\{\Theta_t\}_{t\geq 1}$ are positive, we have
\begin{equation*}
 \sum_{t=1}^n\xtheta \leq \max_{1\leq k\leq n}\sum_{t=1}^n\xtheta \leq \sum_{t=1}^n\Theta_tX_t^+, \quad n\geq 1.
\end{equation*}

Thus it suffices to show for $n\geq 1$,
\begin{equation}
 \label{chap2:proof:finite sum:claim1}
\prob\left[\sum_{t=1}^n\Theta_tX_t^+>x\right]\sim\prob[X_1>x]\sum_{t=1}^n\E[\Theta_t^\alpha],\quad \text{as } x\rightarrow\infty
\end{equation}
\begin{equation}
 \label{chap2:proof:finite sum:claim2}
\limsup_{x\rightarrow\infty}\frac{\prob\left[\sum_{t=1}^n\xtheta>x\right]}{\prob[X_1>x]}\geq \sum_{t=1}^n\E[\Theta_t^\alpha].
\end{equation}

Note that~\eqref{chap2:proof:finite sum:claim1} immediately follows from Corollary~\ref{fin sum cor}.  Also note that~\eqref{chap2:proof:finite sum:claim2} holds for $n=1$ by the modified Breiman's theorem given in Lemma~\ref{breiman new}. Suppose $n\geq 2$. Let $v>1$ be a constant and set $y=(v-1)/(n-1)$. Clearly $y>0$. 

\begin{align}
 \prob\left[\sum_{t=1}^n\xtheta>x\right]&\geq \prob\left[\sum_{t=1}^n\xtheta>x,\max_{1\leq s\leq n}\Theta_sX_s >vx\right]\nonumber \\
&\geq \sum_{s=1}^n\prob\left[\sum_{t=1}^n\xtheta>x,\Theta_sX_s >vx\right]\nonumber \\
&\qquad-\sum_{1\leq k\neq l\leq n}\prob\left[\sum_{t=1}^n\xtheta>x,\Theta_kX_k >vx,\Theta_lX_l>vx\right]\nonumber \\
&\qquad:=\Delta_1-\Delta_2. \label{chap2:proof:finite sum: difference}
\end{align}

For $\Delta_1$ in~\eqref{chap2:proof:finite sum: difference}, we have

\begin{align*}
&\prob\left[\sum_{t=1}^n\xtheta>x, \Theta_sX_s >vx \right]\\ 
&\qquad \geq \prob\left[\sum_{t=1}^n\xtheta>x,\Theta_sX_s >vx, \Theta_jX_j>-yx,1\leq j\leq n,j\neq s\right]\\
&\qquad \geq \prob\left[\Theta_sX_s >vx, \Theta_jX_j>-yx,1\leq j\leq n,j\neq s\right]\\
&\qquad \geq 1-\left(\prob\left[\Theta_sX_s\leq vx\right]+\sum_{\substack{j=1\\ j\neq s}}^n\prob\left[\Theta_jX_j\leq -yx\right]\right)
\end{align*}

It follows from Lemma~\ref{left tail negligibility},
$$\lim_{x\rightarrow\infty}\frac{\prob\left[\Theta_jX_j\leq -yx\right]}{\prob\left[\Theta_jX_j>x\right]}=0 \quad \text{for } 1\leq j\leq n.$$
Therefore we have,
\begin{equation}
 \label{chap2:proof:finite sum: final eq}
\limsup_{x\rightarrow\infty}\frac{\Delta_1}{\prob[X_1>x]}\geq \limsup_{x\rightarrow\infty}\sum_{s=1}^n\frac{\prob\left[\Theta_sX_s>vx\right]}{\prob[X_1>x]}=v^{-\alpha}\sum_{s=1}^n\E[\Theta_t^\alpha].
\end{equation}
For $\Delta_2$ in~\eqref{chap2:proof:finite sum: difference} by Lemma~\ref{joint asymptotic independence} we have for $1\leq k\neq l\leq n$,
\begin{multline*}
 \lim_{x\rightarrow\infty}\frac{\prob\left[\sum_{t=1}^n\xtheta>x,\Theta_kX_k >vx,\Theta_lX_l>vx\right]}{\prob[\Theta_lX_l>x]}\\ \leq v^{-\alpha}\lim_{x\rightarrow\infty}\frac{\prob\left[\Theta_kX_k >vx,\Theta_lX_l>vx\right]}{\prob[\Theta_lX_l>vx]}=0.
\end{multline*}

So this implies, by Lemma~\ref{breiman new}, $\Delta_2=\lito(\prob[X_1>x])$. Now letting $v\rightarrow1$ in~\eqref{chap2:proof:finite sum: final eq}, we get the result.

\end{proof}

\end{section}

\begin{section}{The tail of the weighted sum under the DZ conditions} \label{sec:direct}
In Proposition~\ref{finite sum}, we saw that the conditions on the slowly varying function helps us to reduce the moment conditions on $\{\Theta_t\}$ for the finite sum. However we need some additional hypotheses to handle the infinite series. To study the almost sure convergence of $\xinf=\sum_{t=1}^{\infty} \Theta_t X_t^{+}$, observe that the partial sums $S_n=\sum_{t=1}^{n}\Theta_t X_t^{+}$ increase to $\xinf$. We shall show in the following results that $\prob[\xinf>x]\sim\prob[X_1>x]\sum_{t=1}^{\infty}\E[\Theta_t^{\alpha}]$ under suitable conditions. Thus if $\sum_{t=1}^{\infty}\E[\Theta_t^{\alpha}] < \infty$, then $\lim_{x\rightarrow\infty}\prob[\xinf>x]=0$ and $\xinf$ is finite almost surely.

To obtain the required tail behavior, we shall assume the following conditions, which weaken the moment requirements of $\{\Theta_t\}$ assumed in the conditions~\eqref{RW1} and~\eqref{RW2} given in \cite{resnick:willekens:1991}:
\renewcommand{\labelenumi}{(RW\arabic{enumi}$^\prime$)}
\renewcommand{\theenumi}{RW\arabic{enumi}$^\prime$}
\begin{enumerate}
    \item \label{ERW1} For $ 0 <\alpha<1$, $\sum_{t=1}^{\infty} \E[\Theta_t^{\alpha}]<\infty$.
    \item \label{ERW2} For $1\leq\alpha<\infty$, for some $\epsilon>0$, $\sum_{t=1}^{\infty} (\E[\Theta_t^{\alpha}])^{\frac{1}{\alpha+\epsilon}} <\infty$.
\end{enumerate}
\renewcommand{\theenumi}{\myenumi}
We shall call these conditions modified RW moment conditions.

\begin{remark} \label{almost sure remark}
As observed in Remark~\ref{rem: RW}, for $\alpha\geq 1$ and $\epsilon>0$, the finiteness of the sum in~\eqref{ERW2} implies $\sum_{t=1}^{\infty}(\E[\Theta_t^{\alpha}])<\infty$. Thus to check the almost sure finiteness of $\xinf$, it is enough to check the tail asymptotics condition:
$$\prob[\xinf > x]\sim \prob[X_1>x]\sum_{t=1}^{\infty}\E[\Theta_t^{\alpha}].$$

We shall prove it under the above model together with the assumption that $\prob[\Theta_t>x] = \lito(\prob[X_1>x])$ and one of the DZ conditions. We need to assume an extra summability condition for uniform convergence, when the conditions~\eqref{DZ2},~\eqref{DZ3} or~\eqref{DZ4} hold.

Further note that $\Theta_1 X_1 \le \max_{1\le n<\infty} \sum_{t=1}^n \Theta_t X_t \le X_{(\infty)}$ and hence the almost sure finiteness of $X_{(\infty)}$ guarantees that $\max_{1\le n<\infty} \sum_{t=1}^n \Theta_t X_t$ is a valid random variable.
\end{remark}

\begin{theorem}
Suppose that $\{X_t\}$ is a sequence of pairwise asymptotically independent, identically distributed random variables with common regularly varying tail of index $-\alpha$, where $\alpha>0$, satisfying~\eqref{right tail}, which is independent of another sequence of positive random variables $\{\Theta_t\}$. Also assume that $\prob[\Theta_t>x]=\lito(\prob[X_1>x])$, the pairs $(\Theta_t, X_t)$ satisfy one of the four DZ conditions~\eqref{DZ1}--\eqref{DZ4} and, depending on the value of $\alpha$, the modified RW moment conditions~\eqref{ERW1} or~\eqref{ERW2} holds. If the pairs $(\Theta_t, X_t)$ satisfy DZ condition~\eqref{DZ2},~\eqref{DZ3} or~\eqref{DZ4}, define
\begin{equation} \label{Ct}
C_t =
\begin{cases}
\sup_x \frac{\prob[\Theta_t > x]}{\prob[X_1>x]}, &\text{when~\eqref{DZ2} holds,}\\[1ex]
\sup_x \frac{\prob[\Theta_t > x]}{x^{-\alpha}\prob[U>\log x]}, &\text{when~\eqref{DZ3} holds,}\\[1ex]
\sup_x \frac{\prob[\Theta_t > x]}{\prob[X_1>x]} m(x), &\text{when~\eqref{DZ4} holds,}
\end{cases}
\end{equation}
and further assume that
\begin{align}
&\sum_{t=1}^\infty C_t < \infty, &\text{when $\alpha<1$,} \label{sum less}\\
&\sum_{t=1}^\infty C_t^{\frac1{\alpha+\epsilon}} < \infty, &\text{when $\alpha\ge 1$.} \label{sum more}
\end{align}

Then
$$\prob\left[\max_{1\leq n <\infty}\sum_{t=1}^{n}\Theta_t X_t > x\right]\sim \prob[\xinf > x]\sim \prob[X_1>x]\sum_{t=1}^{\infty}\E[\Theta_t^{\alpha}]$$
and $\xinf$ is almost surely finite.
\end{theorem}
\begin{proof}
For any $m\geq1$, we have, by Proposition~\ref{finite sum},
$$\prob\left[\max_{1\leq n <\infty}\sum_{t=1}^{n}\Theta_t X_t >x\right] \geq \prob\left[\max_{1\leq n\leq m}\sum_{t=1}^{n}\Theta_t X_t > x\right] \sim \prob[X_1>x]\sum_{t=1}^{m}\E[\Theta_t^{\alpha}]$$
leading to
$$\liminf_{x\rightarrow\infty}\frac{\prob[\max_{1\leq n <\infty}\sum_{t=1}^{n}\Theta_t X_t >x]}{\prob[X_1>x]} \geq \sum_{t=1}^{\infty}\E[\Theta_t^{\alpha}].$$
Similarly, comparing with the partial sums and using Proposition~\ref{finite sum}, we also get
$$\liminf_{x\rightarrow\infty}\frac{\prob[\xinf >x]}{\prob[X_1>x]} \geq \sum_{t=1}^{\infty}\E[\Theta_t^{\alpha}].$$

For the other inequality, observe that for any natural number $m$, $0<\delta<1$ and $x\geq 0$,
\begin{multline*}
\prob\left[ \max_{1\leq n <\infty}\sum_{t=1}^{n}\Theta_tX_t > x\right]\\
\leq \prob\left[\max_{1\leq n \leq m}\sum_{t=1}^{n}\Theta_tX_t>(1-\delta)x\right] + \prob\left[\sum_{t=m+1}^{\infty}\Theta_tX_t^{+}>\delta x\right].
\end{multline*}
For the first term, by Proposition~\ref{finite sum} and the regular variation of the tail of $X_1$, we have,
\begin{multline*}
\lim_{x\rightarrow\infty} \frac{\prob\left[\max_{1\leq n \leq m}\sum_{t=1}^{n}\Theta_tX_t>(1-\delta)x\right]}{\prob[X_1>x]}\\ = (1-\delta)^{-\alpha} \sum_{t=1}^m \E[\Theta_t^{\alpha}] \leq (1-\delta)^{-\alpha} \sum_{t=1}^{\infty}\E[\Theta_t^{\alpha}].
\end{multline*}
Also, for $\xinf$, we have,
$$\prob\left[ \xinf > x\right] \leq\ \prob\left[\sum_{t=1}^{m}\Theta_tX_t^+>(1-\delta)x\right] + \prob\left[\sum_{t=m+1}^{\infty}\Theta_tX_t^{+}>\delta x\right]$$
and a similar result holds for the first term.

Then, as $X_1$ is a random variable with regularly varying tail, to complete the proof, it is enough to show that,
\begin{equation} \label{tail negligible}
\lim_{m\rightarrow\infty}\limsup_{x\rightarrow\infty}\frac{\prob[\sum_{t=m+1}^{\infty}\Theta_tX_t^{+} >x]}{\prob[X_1>x]}= 0.
\end{equation}
Now,
\begin{align}
&\prob\left[\sum_{t=m+1}^{\infty}\Theta_t{X_t}^{+} > x\right] \nonumber\\
\leq &\prob\left[\bigvee\limits_{t=m+1}^{\infty}\Theta_t X_t^{+}>x\right] +\prob\left[\sum_{t=m+1}^{\infty}\Theta_tX_t^{+}>x, \bigvee\limits_{t=m+1}^{\infty}\Theta_t X_t^{+}\leq x\right] \nonumber\\
\leq &\sum_{t=m+1}^\infty \prob[\Theta_t X_t > x] + \prob\left[\sum_{t=m+1}^{\infty}\Theta_tX_t^{+}\mathbbm{1}_{[\Theta_t X_t^+ \leq x]}>x\right]. \label{two terms tail negligible}
\end{align}

We bound the final term of~\eqref{two terms tail negligible} separately in the cases $\alpha<1$ and $\alpha\ge 1$. In the rest of the proof, for $\alpha\ge 1$, we shall choose $\epsilon>0$, so that the condition~\eqref{ERW2} holds. We first consider the case $\alpha<1$. By Markov inequality, the final term of~\eqref{two terms tail negligible} gets bounded above by
\begin{align}
\sum_{t=m+1}^{\infty}\frac1x\E\left[\Theta_t X_t^+\mathbbm{1}_{[\Theta_t X_t^+ \leq x]}\right]
&=\sum_{t=m+1}^{\infty}\int_0^{\infty}\frac1{x/v}{\E\left[X_t^+\mathbbm{1}_{[X_t^+\leq x/v]}\right]}G_t(dv) \label{eq: term 2 alpha small}\\
&=\sum_{t=m+1}^{\infty}\int_0^{\infty}\frac{\E[X_t^+\mathbbm{1}_{[X_t^+\leq x/v]}]}{x/v\prob[X_t^+>x/v]}\prob[X_t>x/v]G_t(dv).\nonumber
\end{align}
Now, using Karamata's theorem (see Theorem~\ref{chap1:theorem:karamata}), we have
$$\lim_{x\to\infty} \frac{\E\left[X_t^+\mathbbm{1}_{[X_t^+\leq x/v]}\right]}{x\prob[X_t^+>x]} = \frac\alpha{1-\alpha}$$
and, for $x<1$, we have $\E[X_t^+\mathbbm{1}_{[X_t^+\leq x/v]}]/(x\prob[X_t^+>x])\le 1/\prob[X_t^+>1]$. Thus, $\E[X_t^+\mathbbm{1}_{[X_t^+\leq x/v]}]/(x\prob[X_t^+>x])$ is bounded on $(0,\infty)$. So the final term of~\eqref{two terms tail negligible} becomes bounded by a multiple of $\sum_{t=m+1}^\infty \prob[\Theta_t X_t>x]$.

When $\alpha\geq 1$, using Markov inequality on the final term of~\eqref{two terms tail negligible}, we get a bound for it as
\begin{align}
&\frac1{x^{\alpha+\epsilon}} \E\left[ \left( \sum_{t=m+1}^\infty \Theta_t X_t^+ \mathbbm{1}_{[\Theta_t X_t^+\le x]} \right)^{\alpha+\epsilon} \right],\nonumber
\intertext{and then using Minkowski's inequality, this gets further bounded by}
&\left\{ \sum_{t=m+1}^\infty \left( \E \left[ \frac1{x^{\alpha+\epsilon}} \left(\Theta_t X_t^+\right)^{\alpha+\epsilon} \mathbbm{1}_{[\Theta_t X_t^+\le x]} \right] \right)^{\frac1{\alpha+\epsilon}} \right\}^{\alpha+\epsilon}\nonumber\\
= &\left\{ \sum_{t=m+1}^\infty \left[ \int_0^\infty (x/v)^{-(\alpha+\epsilon)} {\E \left[ \left(X_t^+\right)^{\alpha+\epsilon} \mathbbm{1}_{[X_t^+\le x/v]} \right]} G_t(dv) \right]^{\frac1{\alpha+\epsilon}} \right\}^{\alpha+\epsilon} \label{eq: term 2 alpha large}\\
= &\left\{ \sum_{t=m+1}^\infty \left[ \int_0^\infty \frac{\E \left[ \left(X_t^+\right)^{\alpha+\epsilon} \mathbbm{1}_{[X_t^+\le x/v]} \right]}{(x/v)^{\alpha+\epsilon} \prob[X_t^+ > x/v]} \prob[X_t>x/v] G_t(dv) \right]^{\frac1{\alpha+\epsilon}} \right\}^{\alpha+\epsilon}.\nonumber
\end{align}
Then, again using Karamata's theorem, the first factor of the integrand converges to $\alpha/\epsilon$ and, arguing as in the case $\alpha<1$, is bounded. Thus the final term of~\eqref{two terms tail negligible} is bounded by a multiple of $[ \sum_{t=m+1}^\infty ( \prob[\Theta_t X_t > x] )^{1/(\alpha+\epsilon)} ]^{\alpha+\epsilon}$.

Combining the two cases for $\alpha$, we get, for some $L_1>0$,
$$\frac{\prob[\sum_{t=m+1}^{\infty}\Theta_tX_t^{+} >x]}{\prob[X_1>x]} \le
\begin{cases}
L_1 \sum_{t=m+1}^\infty \frac{\prob[\Theta_t X_t>x]}{\prob[X_1>x]}, &\text{when $\alpha<1$,}\\[2ex]
\sum_{t=m+1}^\infty \frac{\prob[\Theta_t X_t>x]}{\prob[X_1>x]}\\
 + L_1 \left[ \sum_{t=m+1}^\infty \left( \frac{\prob[\Theta_t X_t > x]}{\prob[X_1>x]} \right)^{\frac1{\alpha+\epsilon}} \right]^{\alpha+\epsilon}, &\text{when $\alpha\ge 1$.}
\end{cases}$$
To prove~\eqref{tail negligible}, we shall show
\begin{equation} \label{eq: dir bd}
\frac{\prob[\Theta_t X_t > x]}{\prob[X_1>x]} \le B_t
\end{equation}
for all large values of $x$, where
\begin{equation} \label{eq: bd sum}
\begin{split}
\sum_{t=1}^\infty B_t <\infty, &\text{ for $\alpha<1$,}\\
\sum_{t=1}^\infty B_t^{\frac1{\alpha+\epsilon}} < \infty, &\text{ for $\alpha\ge1$.}
\end{split}
\end{equation}
As mentioned in Remark~\ref{almost sure remark}, for $\alpha\ge 1$ and $\epsilon>0$, $\sum_{t=1}^\infty B_t^{1/{\alpha+\epsilon}} < \infty$ will also imply $\sum_{t=1}^\infty B_t <\infty$. Thus, for both the cases of $\alpha<1$ and $\alpha\ge1$, the sums involved will be bounded by the tail sum of a convergent series and hence~\eqref{tail negligible} will hold.

First observe that
\begin{equation} \label{ratio}
\frac{\prob[\Theta_t X_t>x]}{\prob[X_1>x]}=\int_{0}^{\infty}\frac{\prob[X_1>x/v]}{\prob[X_1>x]}G_t(dv).
\end{equation}
We break the range of integration into three intervals $(0,1]$, $(1,x]$ and $(x,\infty)$, where we choose a suitably large $x$ greater than $1$.

Since $\overline F$ is regularly varying of index $-\alpha$ with $\alpha>0$, $\prob[X_1>x/v]/\prob[X_1>x]$ converges uniformly to $v^\alpha$ for $v\in(0,1)$ or equivalently $1/v\in(1, \infty)$. Hence the integral in~\eqref{ratio} over the first interval can be bounded, for all large enough $x$, as
\begin{equation} \label{eq: first bd}
\int_{0}^{1}\frac{\prob[X_1>x/v]}{\prob[X_1>x]}G_t(dv) \le 2\E[\Theta_t^\alpha].
\end{equation}

For the integral in~\eqref{ratio} over the third interval, we have, for all large enough $x$, by~\eqref{Ct} (for the conditions~\eqref{DZ2},~\eqref{DZ3} and~\eqref{DZ4} only),
\begin{multline} \label{eq: third bd}
\int_{x}^{\infty}\frac{\prob[X_1>x/v]}{\prob[X_1>x]}G_t(dv) \le \frac{\prob\left[\Theta_t>x\right]}{\prob[X_1>x]}\\
\le
\begin{cases}
\frac{\E\left[\Theta_t^\alpha\right]}{L(x)} \leq 2D_1 \frac{\E\left[\Theta_t^\alpha\right]}{L(1)}, &\text{by Markov's inequality, when~\eqref{DZ1} holds,}\\
C_t, &\text{when~\eqref{DZ2} holds,}\\
\frac{\prob\left[\Theta_t>x\right]}{c(x)x^{-\alpha}\prob[U>\log x]} \le \frac2c C_t, &\text{when~\eqref{DZ3} holds,}\\
C_t, &\text{as $m(x)\to\infty$, when~\eqref{DZ4} holds.}
\end{cases}
\end{multline}
Note that, when the condition~\eqref{DZ3} holds and $L$ is of type 4, we can ignore the factor $\prob[V>\log x]$, as it is bounded by $1$.

Finally, we consider the integral in~\eqref{ratio} over the second interval separately for each of the DZ conditions. We begin with the condition~\eqref{DZ1}. In this case, we have, for all large enough $x$,
\begin{multline} \label{eq: second bd DZ1}
\int_{1}^{x}\frac{\prob[X_1>x/v]}{\prob[X_1>x]}G_t(dv) \le \int_{1}^{x} v^\alpha \frac{L(x/v)}{L(x)} G_t(dv)\\ \le \sup_{y\in[1,x]} \frac{L(y)}{L(x)} \E\left[\Theta_t^\alpha\right] \le 2D_1 \E\left[\Theta_t^\alpha\right].
\end{multline}

Next we consider the condition~\eqref{DZ2}. Integrating by parts, we have
$$\int_{1}^{x}\frac{\prob\left[X_1>{x}/{v}\right]}{\prob[X_1>x]}G_t(dv) \leq \prob[\Theta_t > 1] + \int_{1}^{x} \frac{\prob[\Theta_t > v]}{\prob[X_1 > x]} d_{v}\prob\left[X_1 > {x}/{v}\right].$$ Using Markov's inequality and~\eqref{Ct} respectively in each of the terms, we have
$$\int_{1}^{x}\frac{\prob\left[X_1>{x}/{v}\right]}{\prob[X_1>x]}G_t(dv) \leq \E[\Theta_t^\alpha] + C_t \int_1^x \frac{\prob[X_1 > v]}{\prob[X_1 > x]} d_{v}\prob\left[X_1 > {x}/{v}\right].$$
Substituting $u=\log v$, the second term becomes, for all large $x$,
\begin{multline*}
C_t \int_0^{\log x} \frac{\prob[\log X_1 > u]}{\prob[\log X_1 > \log x]} d_u\prob\left[\log X_1 > \log x - u \right]\\
\le 2 C_t \E[\exp(\alpha (\log X_1)^+)] \le 2 C_t \E[X_1^\alpha],
\end{multline*}
where the inequalities follow, since $L(e^x)\in\mathcal S_d$ implies $(\log X_1)^+\in \mathcal S(\alpha)$, cf. \cite{kluppelberg:1989}. Thus,
\begin{equation} \label{eq: second bd DZ2}
\int_{1}^{x}\frac{\prob[X_1>x/v]}{\prob[X_1>x]}G_t(dv) \le \E[\Theta_t^\alpha] + 2 C_t \E[X_1^\alpha].
\end{equation}

Next we consider the condition~\eqref{DZ3}. In this case, we have
\begin{align*}
\int_1^{x}\frac{\prob[X_1>x/v]}{\prob[X_1>x]}G_t(dv)&= \int_1^{x}\frac{L(x/v)}{L(x)}v^{\alpha}G_t(dv)\\
&\leq \sup_{v\in[1,x]}\frac{c(x/v)}{c(x)} \int_1^{x}\frac{\prob[U > \log x -\log v]}{\prob[U > \log x]}v^{\alpha}G_t(dv).
\end{align*}
If $L$ is of type 4, the ratio $L(x/v)/L(x)$ has an extra factor $\prob[V>\log x]/\prob[V>\log x-\log v]$, which is bounded by $1$. Thus the above estimate works if $L$ is either of type 3 or of type 4. Since $c(x)\rightarrow c \in (0,\infty)$, we have $\sup_{v\in[N,x)} {c(x/v)}/{c(x)} := L_2 <\infty$. Integrating by parts, the integral becomes
\begin{multline*}
\int_1^{x}\frac{\prob[U>\log x-\log v]}{\prob[U>\log x]}v^{\alpha}G_t(dv)\\
\leq \prob[\Theta_t >1]+\int_1^{x}\frac{\prob[U>\log x- \log v]\prob[\Theta_t > v]}{\prob[U>\log x]}\alpha v^{\alpha -1}dv\\
+\int_1^{x}\frac{\prob[\Theta_t > v]v^{\alpha}}{\prob[U>\log x]}d_{v}\prob[U>\log x-\log v].
\end{multline*}
The first term is bounded by $\E[\Theta_1^\alpha]$ by Markov's inequality. By~\eqref{Ct}, the second term gets bounded by, for all large enough $x$,
$$\alpha C_t \int_1^x \frac{\prob[U>\log x-\log v]\prob[U>\log v]}{\prob[U>\log x]}d(\log v) \le 2 \alpha C_t \E[U],$$
as $U$ belongs to $\mathcal S^*$. Again, by~\eqref{Ct}, the third term gets bounded by, for all large enough $x$,
$$C_t \int_1^x \frac{\prob[U>\log v] d_v \prob[U>\log x - \log v]}{\prob[U>\log x]} \le 4 C_t,$$
as $U$ belongs to $\mathcal S^*$ and hence is subexponential, cf. \cite{kluppelberg:1988}. Combining the bounds for the three terms, we get
\begin{equation} \label{eq: second bd DZ3}
\int_{1}^{x}\frac{\prob[X_1>x/v]}{\prob[X_1>x]}G_t(dv) \le L_2 \{\E[\Theta_t^\alpha] + 2 (\alpha \E[U] + 2) C_t\}.
\end{equation}

Finally we consider the condition~\eqref{DZ4}. In this case, we split the interval $(1,x]$ into two subintervals $(1,\sqrt{x}]$ and $(\sqrt{x},x]$ and bound the integrals on each of the subintervals separately. We begin with the integral on the subinterval $(1,\sqrt{x}]$.
$$\int_1^{\sqrt x} \frac{L(x/v)}{L(x)}v^{\alpha}G_t(dv)\leq \sup_{v\in (1,\sqrt x]}\frac{L(x/v)}{L(x)}\int_1^{\sqrt x}v^{\alpha}G_t(dv) \leq D_2 \E[\Theta_t^{\alpha}].$$
For the integral over $(\sqrt{x}, x]$, we integrate by parts to obtain
$$\int_{\sqrt x}^{x}\frac{L(x/v)}{L(x)}v^{\alpha}G_t(dv)\leq \prob[\Theta_t>\sqrt x]x^{\alpha/2}\frac{L(\sqrt x)}{L(x)}+\int_{\sqrt x}^x\frac{\prob[\Theta_t>v]}{L(x)}d_v(v^{\alpha}L(x/v)).$$
By Markov's inequality, the first term is bounded by $D_2 \E[\Theta_t^\alpha]$. The second term becomes, using~\eqref{Ct},
\begin{align*}
\int_{\sqrt x}^x \frac{\prob[\Theta_t>v]}{L(x)} x^\alpha d_v(\prob[X_1\le x/v]) &\leq C_t\int_{\sqrt x}^x \frac{\prob[X_1>v]}{L(x)m(v)} x^\alpha d_v(\prob[X_1\le x/v])\\
&\leq \frac{C_t}{m(\sqrt x)}\int_{\sqrt x}^x\frac{L(v)}{L(x)}\left(\frac{x}{v}\right)^{\alpha}d_v(\prob[X_1\leq x/v])\\
&\leq \frac{D_2 C_t}{m(\sqrt x)} \int_{1}^{\sqrt x}y^{\alpha}d_y(\prob[X_1\leq y]) \le D_2 C_t.
\end{align*}
Combining the bounds for the integrals over each subinterval, we get
\begin{equation} \label{eq: second bd DZ4}
\int_{1}^{x}\frac{\prob[X_1>x/v]}{\prob[X_1>x]}G_t(dv) \le D_2 (2 \E[\Theta_t^\alpha] + C_t).
\end{equation}

Combining all the bounds in~\eqref{eq: first bd}--\eqref{eq: second bd DZ4}, for some constant $B$, we can choose the bound in~\eqref{eq: dir bd} as
$$B_t =
\begin{cases}
B \E[\Theta_t^\alpha], &\text{when the condition~\eqref{DZ1} holds,}\\
B (\E[\Theta_t^\alpha] + C_t), &\text{when the conditions~\eqref{DZ2}, \eqref{DZ3} or~\eqref{DZ4} hold.}
\end{cases}$$
Then, for $\alpha<1$, the summability condition~\eqref{eq: bd sum} follows from the condition~\eqref{ERW1} alone under the condition~\eqref{DZ1} and from the condition~\eqref{ERW1} together with~\eqref{sum less} under the conditions~\eqref{DZ2},~\eqref{DZ3} or~\eqref{DZ4}. For $\alpha\ge 1$, under the condition~\eqref{DZ1}, the summability condition~\eqref{eq: bd sum} follows from the condition~\eqref{ERW2}. Finally, to check the summability condition~\eqref{eq: bd sum} for $\alpha\ge1$, under the condition~\eqref{DZ2},~\eqref{DZ3} or~\eqref{DZ4}, observe that as $\alpha\ge 1$ and $\epsilon>0$, we have
$$(\E[\Theta_t^\alpha] + C_t)^{\frac1{\alpha+\epsilon}} \le (\E[\Theta_t^\alpha])^{\frac1{\alpha+\epsilon}} + C_t^{\frac1{\alpha+\epsilon}}$$ and we get the desired condition from the condition~\eqref{ERW2}, together with~\eqref{sum more}.
\end{proof}
\end{section}

\begin{section}{The tails of the summands from the tail of the sum} \label{sec: converse}
In this section, we address the converse problem of studying the tail behavior of $X_1$ based on the tail behavior of $\xinf$. For the converse problem, we restrict ourselves to the setup where the sequence $\{X_t\}$ is positive and pairwise asymptotically independent and the other sequence $\{\Theta_t\}$ is positive and independent of the sequence $\{X_t\}$, such that $\xinf$ is finite with probability one and has regularly varying tail of index $-\alpha$. Depending on the value of $\alpha$, we assume the usual RW moment conditions~\eqref{RW1} or~\eqref{RW2} for the sequence $\{\Theta_t\}$, instead of the modified ones. Then, under a further assumption of the non-vanishing Mellin transform along the vertical line of the complex plane with the real part $\alpha$, we shall show that $X_1$ also has regularly varying tail of index $-\alpha$.

We use the extension of the notion of product of two independent positive random variables to the product convolution of two measures on $(0,\infty)$, which we allow to be $\sigma$-finite. For two $\sigma$-finite measures $\nu$ and $\rho$ on $(0,\infty)$, we define the product convolution as
$$\nu \circledast \rho(B)= \int_0^{\infty}\nu(x^{-1}B)\rho(dx),$$
for any Borel subset $B$ of $(0,\infty)$. We shall need the following result from \cite{JMRS:2009}.
\begin{theorem}[\citealp{JMRS:2009}, Theorem~2.3] \label{JMRS}
Let a non-zero $\sigma$-finite measure $\rho$ on $(0,\infty)$ satisfies, for some $\alpha>0$, $\epsilon\in(0,\alpha)$ and all $\beta\in\mathbb R$,
\begin{align}
\int_{0}^{\infty} \left(y^{\alpha-\epsilon}\vee y^{\alpha+\epsilon}\right) \rho(dy)&<\infty \label{conditionjmrs0}
\intertext{and}
\int_{0}^{\infty}y^{\alpha+i\beta}\rho(dy)&\neq 0. \label{conditionjmrs1}
\end{align}

Suppose, for another $\sigma$-finite measure $\nu$ on $(0,\infty)$, the product convolution measure $\nu\circledast\rho$ has a regularly varying tail of index $-\alpha$ and
\begin{equation} \label{conditionjmrs} \lim_{b\rightarrow0}\limsup_{x\rightarrow\infty}\frac{\int_0^b\rho(x/y,\infty)\nu(dy)}{(\nu\circledast\rho)(x,\infty)}=0.
\end{equation}
Then the measure $\nu$ has a regularly varying tail of index $-\alpha$ as well and
$$\lim_{x\rightarrow\infty}\frac{\nu\circledast\rho (x,\infty)}{\nu(x,\infty)}=\int_0^{\infty}y^{\alpha}\rho(dy).$$

Conversely, if~\eqref{conditionjmrs0} holds but~\eqref{conditionjmrs1} fails for the measure $\rho$, then there exists a $\sigma$-finite measure $\nu$ without regularly varying tail, such that $\nu\circledast\rho$ has regularly varying tail of index $-\alpha$ and~\eqref{conditionjmrs} holds.
\end{theorem}

\begin{remark} \label{rem: jmrs}
\cite{JMRS:2009} gave an explicit construction of the $\sigma$-finite measure $\nu$ in Theorem~\ref{JMRS} above. In fact, if~\eqref{conditionjmrs1} fails for $\beta=\beta_0$, then, for any real number $a$ and $b$ satisfying $0<a^2+b^2\le 1$, we can define $g(x) = 1 + a \cos(\beta_0 \log x) + b \sin(\beta_0 \log x)$ and $d\nu = g d\nu_\alpha$ will qualify for the measure in the converse part, where $\nu_\alpha$ is the $\sigma$-finite measure given by $\nu_\alpha(x,\infty)=x^{-\alpha}$ for any $x>0$.

It is easy to check that $0\le g(x)\le2$ for all $x>0$ and hence
\begin{equation} \label{eq: bd nu}
\nu(x,\infty)\le2x^{-\alpha}.
\end{equation}
Also, it is known from Theorem~2.1 of \cite{JMRS:2009} that
\begin{equation} \label{eq: prod conv}
\nu\circledast\rho=\|\rho\|_\alpha\nu_\alpha,
\end{equation}
where $\|\rho\|_\alpha = \int_0^\infty y^\alpha \rho(dy) <\infty$, by~\eqref{conditionjmrs0}.
\end{remark}

We are now ready to state the main result of this section.
\begin{theorem}\label{main result}
Let $\{X_t,t\geq1\}$ be a sequence of identically distributed, pairwise asymptotically independent positive random variables and $\{\Theta_t,t\geq1\}$ be a sequence of positive random variables independent of $\{X_t\}$, such that $\xinf=\sum_{t=1}^{\infty}\Theta_tX_t$ is finite with probability one and has regularly varying tail of index $-\alpha$, where $\alpha>0$. Let $\{\Theta_t,t\geq1\}$ satisfy the appropriate RW condition~\eqref{RW1} or~\eqref{RW2}, depending on the value of $\alpha$. If we further have, for all $\beta\in \mathbb{R}$,
\begin{equation} \label{eq: mellin}
\sum_{t=1}^{\infty}\E[\Theta_t^{\alpha+i\beta}]\neq 0,
\end{equation}
then $X_1$ has regularly varying tail of index $-\alpha$ and, as $x\to\infty$,
$$\prob[\xinf > x]\sim \prob[X_1>x]\sum_{t=1}^{\infty}\E[\Theta_t^{\alpha}] \ \text{as}\ x\rightarrow\infty.$$
\end{theorem}

We shall prove Theorem~\ref{main result} in several steps. We collect the preliminary steps, which will also be useful for a converse to Theorem~\ref{main result}, into three separate lemmas. The first lemma controls the tail of the sum $\xinf$.
\begin{lemma} \label{lem: tail}
Let $\{X_t\}$ be a sequence of identically distributed positive random variables and $\{\Theta_t\}$ be a sequence of positive random variables independent of $\{X_t\}$. Suppose that the tail of $X_1$ is dominated by a bounded regularly varying function $R$ of index $-\alpha$, where $\alpha>0$, that is, for all $x>0$,
\begin{equation} \label{eq: domination}
\prob[X_1>x] \le R(x).
\end{equation}
Also assume that $\{\Theta_t\}$ satisfies the appropriate RW condition depending on the value of $\alpha$. Then,
\begin{align*}
\lim_{m\to\infty} \limsup_{x\to\infty} \frac{\prob[\xupk>x]}{R(x)} &= 0 \intertext{and}
\lim_{m\to\infty} \limsup_{x\to\infty} \sum_{t=m+1}^\infty \frac{\prob[\Theta_t X_t>x]}{R(x)} &= 0.
\end{align*}
\end{lemma}
\begin{proof}
From~\eqref{two terms tail negligible}, we have
\begin{equation} \label{eq: tail bd conv}
\prob\left[\xupk>x\right] \le \sum_{t=m+1}^\infty \prob\left[\Theta_tX_t>x\right] + \prob\left[\sum_{t=m+1}^\infty \Theta_tX_t \mathbbm 1_{[\Theta_tX_t\le x]}>x\right].
\end{equation}
Using~\eqref{eq: domination}, the summands of the first term on the right side of~\eqref{eq: tail bd conv} can be bounded as
\begin{equation} \label{eq: first term conv}
\prob[\Theta_t X_t > x] = \int_0^\infty \prob[X_t > x/u] G_t(du) \leq \int_0^\infty R(x/u) G_t(du).
\end{equation}

Before analyzing the second term on the right side of~\eqref{eq: tail bd conv}, observe that, for $\gamma>\alpha$, we have, using Fubini's theorem,~\eqref{eq: domination} and Karamata's theorem successively
$$\E\left[X_t^\gamma \mathbbm 1_{[X_t\le x]}\right] \le \gamma \int_0^x u^{\gamma-1} \prob[X_t>u] du \le \gamma \int_0^x u^{\gamma-1} R(u) du \sim \frac{\gamma}{\gamma-\alpha} x^\gamma R(x).$$
Thus, there exists constant $M \equiv M(\gamma)$, such that, for all $x>0$,
\begin{equation} \label{eq: karamata}
x^{-\gamma} \E\left[X_t^\gamma \mathbbm 1_{[X_t\le x]}\right] \le M R(x).
\end{equation}

We bound the second term on the right side of~\eqref{eq: tail bd conv}, using~\eqref{eq: karamata}, separately for the cases $\alpha<1$ and $\alpha\ge 1$. For $\alpha<1$, we use~\eqref{eq: term 2 alpha small} and~\eqref{eq: karamata} with $\gamma=1$, to get
\begin{equation} \label{eq: second term alpha small}
\prob\left[\sum_{t=m+1}^\infty \Theta_tX_t \mathbbm 1_{[\Theta_tX_t\le x]}>x\right] \le M(1) \sum_{t=m+1}^\infty \int_0^\infty R(x/u) G_t(du).
\end{equation}
For $\alpha\ge 1$, we use~\eqref{eq: term 2 alpha large} and~\eqref{eq: karamata} with $\gamma=\alpha+\epsilon$, to get
\begin{equation} \label{eq: second term alpha large}
\prob\left[\sum_{t=m+1}^\infty \Theta_tX_t \mathbbm 1_{[\Theta_tX_t\le x]}>x\right] \le M(\alpha+\epsilon) \left[ \sum_{t=m+1}^\infty \left( \int_0^\infty R(x/u) G_t(du) \right)^{\frac1{\alpha+\epsilon}} \right]^{\alpha+\epsilon}.
\end{equation}

Combining~\eqref{eq: first term conv},~\eqref{eq: second term alpha small} and~\eqref{eq: second term alpha large} with the bound in~\eqref{eq: tail bd conv}, the proof will be complete if we show
\begin{equation} \label{eq: aim}
\begin{split}
&\lim_{m\to\infty} \limsup_{x\to\infty} \sum_{t=m+1}^\infty \int_0^\infty \frac{R(x/u)}{R(x)} G_t(du) = 0, \text{ for $\alpha<1$},\\
\text{and } &\lim_{m\to\infty} \limsup_{x\to\infty} \sum_{t=m+1}^\infty \left( \int_0^\infty \frac{R(x/u)}{R(x)} G_t(du) \right)^{\frac1{\alpha+\epsilon}} =0, \text{ for $\alpha\ge1$}.
\end{split}
\end{equation}
Note that, for $\alpha\ge 1$, as in Remark~\ref{rem: RW}, the second limit above gives the first one as well.

We bound the integrand using a variant of Potter's bound (see~\eqref{potter2}) . Let $\epsilon>0$ be as in the RW conditions. Then there exists a $x_0$ and a constant $M>0$ such that, for $x>x_0$, we have
\begin{equation} \label{eq: potter}
\frac{R(x/u)}{R(x)}\leq
\begin{cases}
Mu^{\alpha-\epsilon},  &\text{if $u<1$,}\\
Mu^{\alpha +\epsilon},  &\text{if $1\leq u\leq x/x_0$.}
\end{cases}
\end{equation}

We split the range of integration in~\eqref{eq: aim} into three intervals, namely $(0,1]$, $(1,x/x_0]$ and $(x/x_0,\infty)$. For $x>x_0$, we bound the integrand over the first two integrals using~\eqref{eq: potter} and hence the integrals get bounded by a multiple of $\E[\Theta_t^{\alpha-\epsilon}]$ and $\E[\Theta_t^{\alpha+\epsilon}]$ respectively. As $R$ is bounded, by Markov's inequality, the third integral gets bounded by a multiple of $x_0^{\alpha+\epsilon}\E[\Theta_t^{\alpha+\epsilon}]/\{x^{\alpha+\epsilon}R(x)\}$. Putting all the bounds together, we have
$$\int_0^\infty \frac{R(x/u)}{R(x)} G_t(du) \le M \left(\E[\Theta_t^{\alpha-\epsilon}] + \E[\Theta_t^{\alpha+\epsilon}] + \frac{x_0^{\alpha+\epsilon}\E[\Theta_t^{\alpha+\epsilon}]}{x^{\alpha+\epsilon}R(x)}\right).$$
Then,~\eqref{eq: aim} holds for $\alpha<1$ using the condition~\eqref{RW1} and the fact that $R$ is regularly varying of index $-\alpha$. For $\alpha\ge 1$, we need to further observe that, as $\alpha+\epsilon>1$, we have
\begin{multline*}
\left(\int_0^\infty \frac{R(x/u)}{R(x)} G_t(du)\right)^{\frac1{\alpha+\epsilon}} \le M^{\frac1{\alpha+\epsilon}} \left[ \left( \E[\Theta_t^{\alpha-\epsilon}] + \E[\Theta_t^{\alpha+\epsilon}]\right) + \frac{x_0^{\alpha+\epsilon} \E[\Theta_t^{\alpha+\epsilon}]}{x^{\alpha+\epsilon}R(x)}\right]^{\frac1{\alpha+\epsilon}}\\
\le M^{\frac1{\alpha+\epsilon}} \left(\E[\Theta_t^{\alpha-\epsilon}] + \E[\Theta_t^{\alpha+\epsilon}]\right)^{\frac1{\alpha+\epsilon}} + \frac{x_0 \left(\E[\Theta_t^{\alpha+\epsilon}]\right)^{\frac1{\alpha+\epsilon}}}{xR(x)^{\frac1{\alpha+\epsilon}}}
\end{multline*}
and~\eqref{eq: aim} holds using the condition~\eqref{RW2} and the fact that $R$ is regularly varying of index $-\alpha$.
\end{proof}

The next lemma considers the joint distribution of $(\Theta_1 X_1, \Theta_2 X_2)$ and shows they are ``somewhat'' asymptotically independent, if $(X_1, X_2)$ are asymptotically independent.
\begin{lemma} \label{lem: approx indep}
Let $(X_1, X_2)$ and $(\Theta_1, \Theta_2)$ be two independent random vectors, such that each coordinate of either vector is positive. We assume that $X_1$ and $X_2$ have same distribution with their common tail dominated by a regularly varying function $R$ of index $-\alpha$ with $\alpha>0$, as in~\eqref{eq: domination}. We also assume that $R$ stays bounded away from $0$ on any bounded interval. We further assume that both $\Theta_1$ and $\Theta_2$ have $(\alpha+\epsilon)$-th moments finite. Then
$$\lim_{x\to\infty} \frac{\prob[\Theta_1 X_1 > x, \Theta_2 X_2 > x]}{R(x)} = 0.$$
\end{lemma}
\begin{proof}
By asymptotic independence and~\eqref{eq: domination}, we have
\begin{equation} \label{eq: jt neg}
\prob[X_1>x,X_2>x] = \lito(R(x)).
\end{equation}
Further, since $R$ is bounded away from $0$ on any bounded interval, $\prob[X_1>x,X_2>x]$ is bounded by a multiple of $R(x)$. Then,
\begin{align*}
\frac{\prob[\Theta_1X_1>x,\Theta_2X_2>x]}{R(x)} =&\int\limits_0^{\infty}\int\limits_0^{\infty}\frac{\prob[X_1>x/u,X_2>x/v]}{R(x)}G(du,dv)\\
=&\iint\limits_{u>v}+\iint\limits_{u\leq v}\frac{\prob[X_1>x/u,X_2>x/v]}{R(x)}G(du,dv)\\
\leq &\int_0^{\infty} \frac{\prob[X_1>x/u,X_2>x/u]}{R(x)} (G_1+G_2)(du)\\
= &\int_0^\infty \frac{\prob[X_1>x/u,X_2>x/u]}{R(x)} \mathbbm 1_{[0,x/x_0]}(u) (G_1+G_2)(du)\\ &\quad + \frac{x_0^{\alpha+\epsilon} \left( \E[\Theta_1^{\alpha+\epsilon}] + \E[\Theta_1^{\alpha+\epsilon}] \right)}{x^{\alpha+\epsilon} R(x)},
\end{align*}
for any $x_0>0$. The integrand in the first term goes to $0$, using~\eqref{eq: jt neg} and the regular variation of $R$. Further choose $x_0$ as in Potter's bound~\eqref{eq: potter}. Then, the integrand of the first term is bounded by a multiple of $1+u^{\alpha+\epsilon}$, which is integrable with respect to $G_1+G_2$. So, by Dominated Convergence Theorem, the first term goes to $0$. For this choice of $x_0$, the second term also goes to $0$, as $R$ is regularly varying of index $-\alpha$.
\end{proof}

The next lemma compares $\sum_{t=1}^m \prob[\Theta_t X_t > x]$ and $\prob\left[ \sum_{t=1}^m \Theta_t X_t > x \right]$.
\begin{lemma} \label{lem: comp}
Let $\{X_t\}$ and $\{\Theta_t\}$ be two sequences of positive random variables.
Then, we have, for any $\frac12<\delta<1$ and $m\ge 2$,
\begin{align}
\prob\left[ \sum_{t=1}^m \Theta_t X_t > x \right] \geq &\sum_{t=1}^m \prob[\Theta_t X_t > x] - \underset{1\le s\neq t\le m}{\sum\sum} \prob[\Theta_s X_s > x, \Theta_t X_t > x] \label{eq: comp up}
\intertext{and}
\prob\left[ \sum_{t=1}^m \Theta_t X_t > x \right] \leq &\sum_{t=1}^m \prob[\Theta_t X_t > x] \nonumber\\
&\quad +\underset{1\le s\neq t\le m}{\sum\sum} \prob\left[\Theta_s X_s > \frac{1-\delta}{m-1} x, \Theta_t X_t > \frac{1-\delta}{m-1} x\right]. \label{eq: comp dn}
\end{align}
\end{lemma}
\begin{proof}
The first inequality~\eqref{eq: comp up} follows from the fact that
$$\left[\sum_{t=1}^m \Theta_t X_t > x\right] \subseteq \bigcup_{t=1}^m [\Theta_t X_t > x]$$ and Bonferroni's inequality.

For the second inequality~\eqref{eq: comp dn}, observe that
$$\prob\left[\xdnk>x\right] \leq \sum_{t=1}^{m}\prob[\xtheta>\delta x]+\prob\left[\sum_{t=1}^{k}\xtheta>x, \bigvee_{t=1}^m\xtheta \leq \delta x\right].$$
Next we estimate the second term as
\begin{align*}
&\prob\left[\sum_{t=1}^{m} \xtheta>x, \bigvee_{t=1}^m \xtheta \leq \delta x\right]\\
=& \prob\left[\sum_{t=1}^m \xtheta>x, \bigvee_{t=1}^m\xtheta \leq \delta x, \bigvee_{t=1}^m \xtheta>\frac xm\right]\\
\leq &\sum_{s=1}^m \prob\left[\sum_{t=1}^{m}\xtheta>x, \bigvee_{t=1}^m \xtheta \leq \delta x, \Theta_{s}X_s>\frac xm \right]\\
\leq &\sum_{s=1}^m \prob\left[\sum_{t=1}^{m}\xtheta>x, \Theta_sX_s \leq \delta x, \Theta_sX_s>\frac xm \right]\\
\leq &\sum_{s=1}^m \prob\left[\sum_{\substack{t=1\\t\neq s}}^k \xtheta>(1-\delta)x, \Theta_sX_s>\frac xm \right]\\
\leq &\underset{1\le s\neq t\le m}{\sum\sum} \prob\left[\xtheta>\frac{1-\delta}{m-1}x, \Theta_sX_s>\frac xm \right]\\
\leq &\underset{1\le s\neq t\le m}{\sum\sum} \prob\left[\xtheta>\frac{1-\delta}{m-1}x, \Theta_sX_s>\frac{1-\delta}{m-1}x \right],
\end{align*}
since $\delta>1/2$ and $m\ge2$ imply $(1-\delta)/(m-1)<1/m$.
\end{proof}

With the above three lemmas, we are now ready to show the tail equivalence of the distribution of $\xinf$ and $\sum_{t=1}^\infty \prob[\Theta_tX_t\in\cdot]$.
\begin{proposition} \label{prop: equiv}
Let $\{X_t,t\geq1\}$ be a sequence of identically distributed, pairwise asymptotically independent positive random variables and $\{\Theta_t,t\geq1\}$ be a sequence of positive random variables independent of $\{X_t\}$, such that $\xinf=\sum_{t=1}^{\infty}\Theta_tX_t$ is finite with probability one and has regularly varying tail of index $-\alpha$, where $\alpha>0$. Let $\{\Theta_t,t\geq1\}$ satisfy the appropriate RW condition~\eqref{RW1} or~\eqref{RW2}, depending on the value of $\alpha$. Then, as $x\to\infty$,
$$\sum_{t=1}^\infty \prob[\Theta_tX_t>x] \sim \prob[\xinf>x].$$
\end{proposition}
\begin{proof}
We first show that the tail of $X_1$ can be dominated by a multiple of the tail of $\xinf$, so that Lemmas~\ref{lem: tail} and~\ref{lem: approx indep} apply. Note that the tail of $\xinf$ is bounded and stays bounded away from $0$ on any bounded interval. As $\Theta_1$ is a positive random variable, choose $\eta>0$ such that $\prob[\Theta_1>\eta]>0$. Then, for all $x>0$,
$$\prob[\xinf>\eta x] \ge \prob[\Theta_1X_1>\eta x, \Theta_1>\eta] \ge \prob[X_1>x] \prob[\Theta_1>\eta].$$
Further, using the regular variation of the tail of $\xinf$, $X_1$ satisfies~\eqref{eq: domination} with $R$ as a multiple of $\prob[\xinf>\cdot]$. Thus, from Lemmas~\ref{lem: tail} and~\ref{lem: approx indep}, we have,
\begin{align}
\lim_{m\to\infty} \limsup_{x\to\infty} \frac{\prob[\xupk>x]}{\prob[\xinf>x]} &= 0, \label{eq: neg tail one}\\
\lim_{m\to\infty} \limsup_{x\to\infty} \sum_{t=m+1}^\infty \frac{\prob[\xtheta>x]}{\prob[\xinf>x]} &= 0, \label{eq: neg tail two}
\intertext{and, for any $s\neq t$,}
\lim_{x\to\infty} \frac{\prob[\Theta_s X_s > x, \xtheta > x]}{\prob[\xinf > x]} &= 0. \label{eq: approx indep}
\end{align}

Choose any $\delta>0$. Then
$$\prob\left[\xinf>(1+\delta)x\right]\leq\prob\left[\xdnk>x\right]+\prob\left[\xupk>\delta x\right],$$
and from~\eqref{eq: neg tail one} and the regular variation of the tail of $\xinf$, we have
$$\lim_{m\to\infty}\liminf_{x\to\infty}\frac{\prob[\xdnk>x]}{\prob[\xinf>x]}\geq 1.$$
Further, using the trivial bound $\prob[\xdnk>x] \le \prob[\xinf>x]$, we have
\begin{equation}\label{eq: lower tail}
1 \leq \lim_{m\to\infty} \liminf_{x\to\infty} \frac{\prob[\xdnk>x]}{\prob[\xinf>x]} \leq \lim_{m\to\infty} \limsup_{x\to\infty} \frac{\prob[\xdnk>x]}{\prob[\xinf>x]} \leq 1.
\end{equation}

We next replace $\prob[\xdnk>x]$ in the numerator by $\sum_{t=1}^m \prob[\xtheta>x]$. We obtain the upper bound first. From~\eqref{eq: comp up},~\eqref{eq: approx indep} and~\eqref{eq: lower tail}, we get
$$\limsup_{x\to\infty} \frac{\sum_{t=1}^m \prob[\xtheta>x]}{\prob[\xinf>x]} \leq 1$$ and letting $m\to\infty$, we get the upper bound. The lower bound follows using exactly similar lines, but using~\eqref{eq: comp dn} and the regular variation of the tail of $\xinf$ instead of~\eqref{eq: comp up}. Putting together, we get
\begin{equation}\label{eq: lower tail sum}
1 \leq \lim_{m\to\infty} \liminf_{x\to\infty} \frac{\sum_{t=1}^m \prob[\xtheta>x]}{\prob[\xinf>x]} \leq \lim_{m\to\infty} \limsup_{x\to\infty} \frac{\sum_{t=1}^m \prob[\xtheta>x]}{\prob[\xinf>x]} \leq 1.
\end{equation}

Then the result follows combining~\eqref{eq: neg tail two} and~\eqref{eq: lower tail sum}.
\end{proof}

We are now ready to prove Theorem~\ref{main result}.
\begin{proof}[Proof of Theorem~\ref{main result}]
Let $\nu$ be the law of $X_1$ and define the measure $\rho(\cdot)=\sum_{t=1}^{\infty}\prob[\Theta_t \in \cdot]$. As observed in Remark~\ref{rem: RW}, under the RW conditions, for all values of $\alpha$, we have $\sum_{t=1}^\infty \E[\Theta_t^{\alpha+\epsilon}] <\infty$. Thus, $\rho$ is a $\sigma$-finite measure. Also, by Proposition~\ref{prop: equiv}, we have $\nu\circledast\rho (x,\infty) = \sum_{t=1}^{\infty} \prob[\xtheta>x] \sim \prob[\xinf>x]$. Hence $\nu\circledast\rho$ has regularly varying tail of index $-\alpha$. As $\nu$ is a probability measure, by Remark~2.4 of \cite{JMRS:2009},~\eqref{conditionjmrs} holds. The RW condition implies~\eqref{conditionjmrs0}. Finally,~\eqref{conditionjmrs1} holds, since, for all $\beta\in\mathbb{R}$, we have, from~\eqref{eq: mellin}, $\int_0^{\infty} y^{\alpha+i\beta} \rho(dy) = \sum_{t=1}^{\infty} \E[\Theta_t^{\alpha+i\beta}] \neq 0$. Hence, by Theorem~\ref{JMRS}, $X_1$ has regularly varying tail of index $-\alpha$.
\end{proof}

As in Theorem~\ref{JMRS},~\eqref{eq: mellin} is necessary for Theorem~\ref{main result} and we give its converse below.
\begin{theorem} \label{thm: converse}
Let $\{\Theta_t,t\geq1\}$ be a sequence of positive random variables satisfying the condition~\eqref{RW1} or~\eqref{RW2}, for some $\alpha>0$, but $\sum_{t=1}^{\infty}\E[\Theta_t^{\alpha+i\beta_0}]=0$ for some $\beta_0\in\mathbb{R}$. Then there exists a sequence of i.i.d.\ positive random variables $\{X_t\}$, such that $X_1$ does not have a regularly varying tail, but $\xinf$ is finite almost surely and has regularly varying tail of index $-\alpha$.
\end{theorem}

The proof depends on an analogue of Proposition~\ref{prop: equiv}.
\begin{proposition} \label{prop: equiv alt}
Let $\{X_t,t\geq1\}$ be a sequence of identically distributed, pairwise asymptotically independent positive random variables and $\{\Theta_t,t\geq1\}$ be a sequence of positive random variables satisfying the condition~\eqref{RW1} or~\eqref{RW2} for some $\alpha>0$ and independent of $\{X_t\}$. If $\sum_{t=1}^\infty \prob[\Theta_tX_t>x]$ is regularly varying of index $-\alpha$, then, as $x\to\infty$,
$$\sum_{t=1}^\infty \prob[\Theta_tX_t>x] \sim \prob[\xinf>x]$$
and $\xinf$ is finite with probability one.
\end{proposition}
\begin{proof}
We shall denote $R(x) = \sum_{t=1}^\infty \prob[\Theta_t X_t > x]$. As $\Theta_1$ is a positive random variable, choose $\eta>0$ such that $\prob[\Theta_1>\eta]>0$. Then, for all $x>0$, we have $R(x) \ge \prob[\Theta_1X_1>\eta x, \Theta_1>\eta] \ge \prob[X_1>x]\prob[\Theta_1>\eta]$
and using the regular variation of $R$, the tail of $X_1$ is dominated by a constant multiple of $R$. Also, note that, $R$ is bounded and stays bounded away from $0$ on any bounded interval. Then, from Lemmas~\ref{lem: tail} and~\ref{lem: approx indep}, we have
\begin{align}
\lim_{m\to\infty} \limsup_{x\to\infty} \frac{\prob[\xupk>x]}{R(x)} &= 0, \label{eq: neg tail one alt}\\
\lim_{m\to\infty} \limsup_{x\to\infty} \sum_{t=m+1}^\infty \frac{\prob[\xtheta>x]}{R(x)} &= 0, \label{eq: neg tail two alt}
\intertext{and, for any $s\neq t$,}
\lim_{x\to\infty} \frac{\prob[\Theta_s X_s > x, \xtheta > x]}{R(x)} &= 0. \label{eq: approx indep alt}
\end{align}

Using~\eqref{eq: neg tail two alt}, we have
$$1 \le \lim_{m\to\infty} \liminf_{x\to\infty} \frac{\sum_{t=1}^m \prob[\xtheta>x]}{R(x)} \le \lim_{m\to\infty} \limsup_{x\to\infty} \frac{\sum_{t=1}^m \prob[\xtheta>x]}{R(x)} \le 1.$$
As in the proof of Proposition~\ref{prop: equiv}, using~\eqref{eq: comp up},~\eqref{eq: comp dn} and~\eqref{eq: approx indep alt}, the above inequalities reduce to
$$1 \le \lim_{m\to\infty} \liminf_{x\to\infty} \frac{\prob[\xdnk>x]}{R(x)} \le \lim_{m\to\infty} \limsup_{x\to\infty} \frac{\prob[\xdnk>x]}{R(x)} \le 1$$
and the tail equivalence follows using~\eqref{eq: neg tail one alt} and the regular variation of $R$. Since $R(x)\to 0$, the tail equivalence also shows the almost sure finiteness of $\xinf$.
\end{proof}

Next, we prove Theorem~\ref{thm: converse} using the converse part of Theorem~\ref{JMRS}.
\begin{proof}[Proof of Theorem~\ref{thm: converse}]
Define the measure $\rho(\cdot)=\sum_{t=1}^{\infty}\prob[\Theta_t\in\cdot]$. By the RW moment condition, the measure $\rho$ is $\sigma$-finite. Further, we have, $\int_0^\infty y^{\alpha+i\beta_0} \rho(dy)=0$. Now by converse part of Theorem~\ref{JMRS}, there exists a $\sigma$-finite measure $\nu$, whose tail is not regularly varying, but $\nu\circledast\rho$ has regularly varying tail. Next, define a probability measure $\mu$ using the $\sigma$-finite measure $\nu$ as in Theorem~3.1 of~\cite{JMRS:2009}. Choose $b>1$, such that $\nu(b,\infty)\leq 1$ and define a probability measure on $(0,\infty)$ by
$$\mu(B)=\nu(B\cap(b,\infty))+(1-\nu(b,\infty))\mathbbm{1}_{B}(1),  \text{ where $B$ is Borel subset of $(0,\infty)$.}$$

First observe that
$$\mu(y,\infty) =
\begin{cases}
\nu(y,\infty), &\text{for $y>b$,}\\
\nu(b,\infty), &\text{for $1<y\le b$,}\\
1, &\text{for $y\le 1$.}
\end{cases}$$
Thus, $\mu$ does not have a regularly varying tail and
\begin{align*}
\mu\circledast\rho(x,\infty) = &\int_0^\infty \mu(x/u,\infty) \rho(du)\\
= &\int_0^{x/b} \nu(x/u,\infty) \rho(du) + \nu(b,\infty) \rho[x/b,x) + \rho[x,\infty)\\
= &\nu\circledast\rho(x,\infty) - 2 x^{-\alpha} \int_{x/b}^\infty u^\alpha \rho(du)\\
&\quad + \nu(b,\infty) \rho[x/b,x) + \rho[x,\infty).
\end{align*}
Now, using the bound from~\eqref{eq: bd nu} and~\eqref{eq: prod conv}, the second term is bounded by, for $x>b$,
$$2 \, \frac{\nu\circledast\rho(x,\infty)}{\|\rho\|_\alpha} \int_{x/b}^\infty u^{\alpha+\epsilon} \rho(du) = \lito(\nu\circledast\rho(x,\infty))$$
as $x\to\infty$, since $\int_0^\infty u^{\alpha+\epsilon} \rho(du) < \infty$ by the RW conditions. The sum of the last two terms can be bounded by
$$\frac{1+\nu(b,\infty)b^{\alpha+\epsilon}}{x^{\alpha+\epsilon}} \int_0^\infty u^{\alpha+\epsilon} \rho(du) = \lito(\nu\circledast\rho(x,\infty)),$$
as $x\to\infty$, since $\nu\circledast\rho(x,\infty)$ is regularly varying of index $-\alpha$. Thus, $\mu\circledast\rho(x,\infty) \sim \nu\circledast\rho(x,\infty)$ as $x\to\infty$ and hence is regularly varying of index $-\alpha$.

Let $X_t$ be an i.i.d.\ sequence with common law $\mu$. Then, $X_1$ does not have regularly varying tail. Further, by Proposition~\ref{prop: equiv alt}, $\xinf$ is finite with probability one and $\prob[\xinf>x] \sim \mu\circledast\rho(x,\infty)$ is regularly varying of index $-\alpha$.
\end{proof}
\end{section}

\cleardoublepage
\chapter{Products in CEVM}\label{chap3}

\begin{section}{Introduction} \label{chap3: sec: intro}
The classical multivariate extreme value theory tries to capture the extremal dependence between the components under a multivariate domain of attraction condition and it requires each of the components to be in domain of attraction of a univariate extreme value distribution. The multivariate extreme value theory has a rich theory but has some limitations as it fails to capture the dependence structure. The concept of tail dependence is an alternative way of detecting this dependence. The concept was first proposed by \cite{ledford:tawn:1996,ledford:tawn:1997} and then elaborated upon by \cite{Resnick:2002, Maulik:Resnick:2004}. A different approach towards modeling multivariate extreme value distributions was given by \cite{heffernan:tawn:2004} by conditioning on one of the components to be  extreme. Further properties of this conditional model were subsequently studied by \cite{heffernan:resnick:2007,Das:Resnick:2008}.

An important dependence structure in multivariate extreme value theory is that of asymptotic independence. Recall from Section~\ref{chap1:subsec2} that the joint distribution of two random variable is asymptotically independent if the the nondegenerate limit of suitably centered and scaled coordinate wise partial maximums is a product measure. One of the limitations of the asymptotic independence model is that it is too large a class to conclude anything interesting, for example, about product of two random variables. In another approach, a smaller class was considered by \cite{Maulik:Resnick:Rootzen:2002}. They assumed that $(X,Y)$ satisfy the following vague convergence:
\begin{equation}\label{asymptotic-ind-strong}
t\prob\left[ \left(\frac{X}{a(t)},Y\right)\in\cdot\right]\vaguec(\nu\times H)(\cdot) \text{  on  }M_+((0,\infty]\times[0,\infty]),
\end{equation}
where $M_+((0,\infty]\times[0,\infty])$ denotes the space of nonnegative Radon measures on $(0,\infty] \times [0,\infty]$ and $\nu(x,\infty]=x^{-\alpha}$, for some $\alpha>0$ and $H$ is a probability distribution supported on $(0,\infty)$. The tail behavior of the product $XY$  under the assumption~\eqref{asymptotic-ind-strong} and some further moment conditions was obtained by \cite{Maulik:Resnick:Rootzen:2002}. The conditional model can be viewed as an extension of the above model. Under the conditional model, the limit of the vague convergence need  not be a product measure and it happens in the space $M_+\left([-\infty,\infty]\times \Ebar\right)$ with $\gamma\in \mathbb{R}$, where $\Ebar$ is the right closure of the set $\{x\in\mathbb{R}:1+\gamma x>0\}$. (See Section~\ref{sec: notation} for details.) In this thesis we mainly focus on the product behavior when the limiting measure in the conditional model is not of the product form.

Products of random variables and their domains of attraction are important theoretical issues which have a lot of applications ranging from Internet traffic to insurance models. We study the product of two random variables whose joint distribution satisfies the conditional extreme value model. In particular we try to see the  role of regular variation in determining the behavior of the product. When $\gamma>0$, then it is easy to describe the behavior of the product under certain conditions. However, when $\gamma<0$, some complications arise due to the presence of finite upper end point. We remark that  we do not deal with the other case of $\gamma =0$ in this thesis. Like in the case of Gumbel domain of attraction for maximum of i.i.d.\ observations, the case $\gamma=0$ will require more careful and detailed analysis.

In Section~\ref{sec: notation} we briefly describe the conditional extreme value model and state some properties and nomenclatures, which we use throughout this Chapter. In Section~\ref{sec: overview} we provide an overview of our results, which are presented in the later sections. In Section~\ref{Transformations} we reduce the conditional extreme value model defined in Section~\ref{sec: notation} to simpler forms for special cases. In Section~\ref{mainresults} we describe the behavior of the product of random variables following conditional model under some appropriate conditions. In Section~\ref{section:remarks on moments} we make some remarks on the assumptions used in the results. In the final Section~\ref{section:example} we present an example of a conditional model where the assumptions of the Theorems in Section~\ref{mainresults} do not hold, yet we look at the product behavior.

\end{section}

\begin{section}{Conditional extreme value model} \label{sec: notation}
In this section, we provide the notations used in this Chapter and the basic model. Some of the definitions and notations have already been  introduced in Chapter~\ref{chap:introduction}, but we recall them again together with the new ones to put all the relevant definitions at one place for quick reference.

% If $\mathsf{S}$ is a topological space with $\mathcal{S}$ being its $\sigma$-field, then for non-negative Radon measures $\mu_t$, for $t>0$, and $\mu$ on $(\mathsf S, \mathcal S)$, we say $\mu_t$ converges vaguely to $\mu$ and denote it by $\mu_t \vaguec \mu$, if for all relatively compact sets $C$, which are also $\mu$-continuity sets of $\mu$, that is, $\mu( \partial C ) = 0$, we have $\mu_t (C) \to \mu(C)$, as $t \to \infty$. By $M_+(\mathsf S)$ we shall mean the space of all Radon measures on $\mathsf{S}$ endowed with the topology of vague convergence.
% 

% \begin{definition} \label{def: reg var}
% A measurable function $f:\mathbb{R}_{+}\rightarrow \mathbb{R}_{+}$ is called regularly varying at infinity with index $\alpha$, (we write $f \in RV_{\alpha}$), if for all $t>0$, $$\lim_{x\rightarrow\infty}\frac{f(tx)}{f(x)} = t^\alpha.$$
% If $\alpha = 0$, $f$ is called slowly varying. 
% \end{definition}
% \begin{definition} \label{def: reg var rv}
% We say that a random variable $X$, with distribution function $F$, has a regularly varying tail of index $-\alpha$,  $\alpha \ge 0$, if the tail of its distribution function $\overline F(\cdot) := \prob[X>\cdot]\in RV_{-\alpha}$. 
% \end{definition}
% Definition~\ref{def: reg var rv} is equivalent to the existence of a positive function $a\in RV_{1/\alpha}$ such that $t\prob[X/a(t)\in \cdot]$ has a vague limit in $M_+((0,\infty])$, where the limit is a nondegenerate Radon measure. The limiting measure necessarily takes values $cx^{-\alpha}$ on set $(x,\infty]$.

Let $\Egamma$ be the interval $\{x \in \mathbb{R}: 1+\gamma x > 0\}$ and $\Ebar$ be its right closure in the extended real line $[-\infty, \infty]$. Thus, we have
$$\Egamma =
\begin{cases}
  (-1/\gamma, \infty), &\text{if $\gamma>0$},\\
  (-\infty, \infty), &\text{if $\gamma=0$},\\
  (-\infty, -1/\gamma), &\text{if $\gamma<0$},\\
\end{cases}
\qquad \text{and} \qquad
\Ebar =
\begin{cases}
  (-1/\gamma, \infty], &\text{if $\gamma>0$},\\
  (-\infty, \infty], &\text{if $\gamma=0$},\\
  (-\infty, -1/\gamma], &\text{if $\gamma<0$}.\\
\end{cases}$$

For any $\gamma\in \mathbb{R}$, recall from Definition~\ref{def:gevd}, the generalized extreme value distribution is denoted by $G_\gamma$. It is supported on $\Egamma$ and is  given by, for $x \in \Egamma$,
$$G_\gamma(x) =
\begin{cases}
\exp\left(-(1+\gamma x)^{-\frac1\gamma}\right), &\text{for $\gamma \neq 0$},\\
\exp\left(-\e^{-x}\right), &\text{for $\gamma = 0$}.
\end{cases}$$

% \begin{definition} \label{def: domain of attr}
% We say that a random variable $Y$ with distribution function $F$ is in the domain of attraction of an extreme value distribution $G_{\gamma}$ for some $\gamma\in \mathbb{R}$ (written as $D(G_\gamma)$) if there exists a positive valued function $a$ and a real valued function $b$ such that as $t\rightarrow \infty$, on $\Egamma$,
% \begin{equation}\label{cevm:doa}
% t\prob\left[Y>a(t)y+b(t)\right]=t \overline{F}(a(t)y+b(t))\rightarrow
% \begin{cases}
%       (1+\gamma y)^{-\frac{1}{\gamma}}, &\text{for $\gamma \neq 0$},\\
%       \e^{-y}, &\text{for $\gamma = 0$}.
% \end{cases}
% \end{equation}
% \end{definition}
% When $\gamma\neq 0$, the domain of attraction condition is related to regular variation in the following way.
% 
% If $\gamma>0$, then $F\in D(G_{\gamma})$ if and only if $\overline{F}\in \text{RV}_{-1/\gamma}$.
% 
% If $\gamma<0$, then as $t\rightarrow\infty$, $b(t)\rightarrow b(\infty)<\infty$ and $F\in D(G_{\gamma})$ if and only if $\overline{F}(b(\infty)-1/\cdot)\in \text{RV}_{1/\gamma}.$ Note that in this case $b(\infty)$ becomes the upper end point of the distribution function~$F$.

\begin{definition} [Conditional extreme value model]\label{defn: cevm}
The real valued random vector $(X, Y)$ satisfies \textit{conditional extreme value model} (CEVM) if
\renewcommand{\labelenumi}{(\arabic{enumi})}
\renewcommand{\theenumi}{\arabic{enumi}}
\begin{enumerate}
\item \label{cevm:doa:def}The marginal distribution of $Y$ is in the domain of attraction of an extreme value distribution $G_{\gamma}$, that is, there exists a positive valued function $a$ and a real valued function $b$ such that~\eqref{chap1:def:doa} holds on $\Egamma$, with $G=G_\gamma$.

\item There exists a positive valued function $\alpha$ and a real valued function $\beta$ and a non-null Radon measure $\mu$ on Borel subsets of $(-\infty,\infty) \times \Egamma$ such that
    \begin{enumerate}[(2A)]
      \item \label{basic1} $t \prob \left[ \left(\frac{X-\beta(t)}{\alpha(t)}, \frac{Y-b(t)}{a(t)}\right) \in \cdot \right] \vaguec \mu(\cdot)$ on $[-\infty, \infty] \times \Ebar$, as $t\to\infty$. 
        \item \label{nondegenerate} for each $y \in \Egamma$, $\mu ((-\infty, x] \times (y, \infty))$ is a nondegenerate distribution function in $x$.
    \end{enumerate}

   The function $\alpha$ is regularly varying of index $\rho\in\mathbb R$ and for the function $\beta$,  $\lim_{t\rightarrow\infty}\frac{\beta(tx)-\beta(t)}{\alpha(t)}$ exists for all $x>0$ and is called $\psi(x)$, say, \citep[cf.][]{heffernan:resnick:2007}. We further have \citep[cf.][Theorem B.2.1]{Haan:Ferreira:2006} either $\psi\equiv 0$ or, for some real number $k$,
   $$\psi(x)=\begin{cases}
  \frac k\rho (x^\rho - 1), &\text{when $\rho \neq 0$},\\
  k \log x, &\text{when $\rho = 0$}.
\end{cases}$$

\item \label{cevm:probability} The function $H(x)=\mu((-\infty,x]\times(0,\infty))$ is a probability distribution.
\end{enumerate}
\end{definition}

If $(X,Y)$ satisfy Conditions~\eqref{cevm:doa:def}--\eqref{cevm:probability}, then we say $(X,Y)\in CEVM(\alpha,\beta; a, b;\mu)$ in  $[-\infty,\infty]\times\Ebar$.

Note that, by  Condition~\eqref{nondegenerate}, $H$ is a nondegenerate probability distribution function. Also, Condition~\eqref{basic1} is equivalent to the convergence
$$t \prob [X \le \alpha(t) x + \beta(t), Y > a(t)y + b(t)] \to \mu ((-\infty, x] \times (y, \infty))$$ for all $y \in \Egamma$ and continuity points $(x, y)$ of the measure $\mu$. Note that, for $(X,Y)\in CEVM(\alpha,\beta;a,b;\mu)$ in  $[-\infty,\infty]\times\Ebar$ we have, for all $x\in[-\infty,\infty]$, as $t\rightarrow\infty$,
$$ \prob \left[ \left. \frac{X - \beta(t)}{\alpha(t)} \le x \right| Y > b(t) \right] \to H(x),$$ which motivates the name of the model.

\begin{remark}
Occasionally, we shall also be interested in pair of random variables $(X,Y)$ which satisfy Conditions~\eqref{basic1} and~\eqref{nondegenerate}, without any reference to Conditions~\ref{cevm:doa:def} and~\ref{cevm:probability}. We shall then say that $(X,Y)$ satisfies Conditions~\eqref{basic1} and~\eqref{nondegenerate} with parameters $(\alpha,\beta; a, b;\mu)$ on  $E$, where $\alpha$ and $\beta$ will denote the scaling and centering of $X$, $a$ and $b$ will denote the scaling and centering of $Y$ and $\mu$ will denote the nondegenerate limiting distribution and $E$ is the space on which the convergence takes place. In Definition~\ref{defn: cevm}, we have $E=[-\infty, \infty] \times \Ebar$.
\end{remark}

\begin{definition} \label{def: mrv}
The pair of nonnegative random variables $(Z_1,Z_2)$ is said to be \textit{standard multivariate regularly varying} on  $[0,\infty]\times(0,\infty]$ if, as $t\rightarrow\infty$
$$ t\prob\left[\left(\frac{Z_1}{t},\frac{Z_2}{t}\right)\in \cdot\right]\vaguec\nu(\cdot) \text{ in } M_+([0,\infty]\times(0,\infty]).$$ 
\end{definition}  
In such cases we have $(Z_1,Z_2)\in CEVM(t,0;t,0;\nu)$ in $[0,\infty]\times(0,\infty]$.  The above convergence implies that $\nu(\cdot)$ is homogeneous of order $-1$, that is, $$\nu(c\Lambda)=c^{-1}\nu(\Lambda) \text{ for all  } c>0$$ where  $\Lambda$ is a Borel subset of $[0,\infty]\times(0,\infty]$. By homogeneity arguments it follows that for $r>0,$
\begin{align*}
&\nu\{(x,y)\in[0,\infty]\times(0,\infty]:x+y>r,\frac{x}{x+y}\in\Lambda\}\\
&\qquad =r^{-1}\nu\{(x,y)\in[0,\infty]\times(0,\infty]:x+y>1,\frac{x}{x+y}\in\Lambda\}\\
&\qquad=:r^{-1}S(\Lambda),
\end{align*}
where $S$ is a measure on $\{(x, y): x+y=1, 0\le x<1\}$. The measure $S$ is called the \textit{spectral measure} corresponding to $\nu(\cdot)$, while the measure $\nu$ is called the standardized measure. It was shown in \cite{heffernan:resnick:2007} that  whenever $(X,Y)\in CEVM(\alpha,\beta;a,b;\mu)$ in $[-\infty,\infty]\times\Ebar$ with $(\rho,\psi(x))\neq(0,0)$, we have the standardization  $(f_1(X),f_2(Y))\in CEVM(t,0;t,0;\nu)$ on the cone $[0,\infty]\times(0,\infty]$, for some monotone transformations $f_1$ and $f_2$.  \cite{Das:Resnick:2008} showed that this standardized measure $\nu$ is not a product measure. Throughout this Chapter  we assume that $(\rho,\psi(x))\neq(0,0)$ and consider the product of $X$ and $Y$. We remark that although the model can be standardized in this case, the standardization does not help one to conclude about the behavior of $XY$.
\end{section}

\begin{section}{A brief overview of the results} \label{sec: overview}
In this section we give a brief description of the results. First note that  if $(X,Y)\in CEVM(\alpha,\beta;a,b;\mu)$ on  $[-\infty,\infty]\times\Ebar$, then  $\alpha\in RV_{\rho}$ and $a\in RV_\gamma$, where $\alpha$ and $a$ were the scalings for $X$ and $Y$ respectively. While $Y\in D(G_\gamma)$ necessarily holds, it  need not a priori follow that $X\in D(G_{\rho})$.  We classify the problem according to the parameters $\gamma$ and $\rho$. We break the problem into four cases depending on whether the parameters $\gamma$ and $\rho$ are positive or negative. In Section~\ref{Transformations} we show that, depending on the properties of the scaling and centering parameters, we can first reduce the basic convergence in conditional model to an equivalent convergence with the limiting measure satisfying nondegeneracy condition in an appropriate cone. The reduction of the basic convergence helps us to compute the convergence of the product with ease in Section~\ref{mainresults}.

{\bf Case I: $\rho$ and $\gamma$ positive:} This is an easier case and the behavior is quiet similar to the classical multivariate extreme theory.  In Theorem~\ref{case1[a]}, we show that under appropriate tail condition on $X$, the product $XY$ has regularly varying tail of index $-1/(\rho+\gamma)$. It is not assumed that $X\in D(G_{\rho})$, but in Section~\ref{section:remarks on moments} we show that the tail condition is satisfied when $X\in D(G_{\rho})$. It may happen that $X$ is in some other domain of attraction but still the tail condition holds. We also present a situation where the tail condition may fail.

In all the remaining cases, at least one of the indices $\rho$ and $\gamma$ will be negative. A negative value for $\gamma$ will require that the upper endpoint of $Y$ is indeed $b(\infty)$. However, as has been noted below, the same need not be true for $\rho$, $X$ and $\beta(\infty)$. Yet, we shall assume that whenever $\rho<0$, the upper endpoint of the support of $X$ is $\beta(\infty)$. Further, we shall assume at least one of the factors $X$ and $Y$ to be nonnegative. If both the factors take negative values, then, the product of the negative numbers being positive, their left tails will contribute to the right tail of the product as well. In fact, it will become important in that case to compare the relative heaviness of the two contributions. This can be easily done by breaking each random variable into its negative and positive parts. For the product of two negative parts, the relevant model should be built on $(-X, -Y)$ and the analysis becomes same after that. While these details increase the bookkeeping, they do not provide any new insight into the problem. So we shall refrain from considering the situations where both $X$ and $Y$ take negative values, except in Subcases II(b) and II(d) below, where both $X$ and $Y$ are nonpositive and we get some interesting result about the lower tail of the product $XY$ easily.

{\bf Case II: $\rho$ and $\gamma$ negative:} In Section~\ref{Transformations}, we first reduce the basic convergence to an equivalent convergence where regular variation can play an important role. In this case both $b(t)$ and $\beta(t)$ have finite limits $b(\infty)$ and $\beta(\infty)$ respectively. Since $Y\in D(G_{\gamma})$, $b(\infty)$ is the right end point of $Y$. However, $\beta(\infty)$ need not be the right end point of $X$ in general, yet throughout we shall assume it to be so. In Section~\ref{Transformations}, we reduce the conditional model $(X,Y)\in CEVM(\alpha,\beta;a,b;\mu)$ to $(\widetilde{X},\widetilde{Y})$ which satisfies Conditions~\eqref{basic1} and~\eqref{nondegenerate} with parameters $(\widetilde{\alpha},0;\widetilde{a},0;\nu)$ on $[0,\infty]\times(0,\infty]$, where,
\begin{equation}\label{eq:x-y-transformation}
 \widetilde{X}=\frac1{\beta(\infty)-X}\quad \text{and}\quad \widetilde{Y}=\frac1{b(\infty)-Y},
\end{equation}
and  $\widetilde{\alpha}$ and $\widetilde{a}$ are some appropriate scalings and $\nu$ is a transformed measure.
Regular variation at the right end point plays a crucial role during the determination the product behavior in this case. Depending on the right end point, we break the problem into few subcases which are interesting.

{\bf Subcase II(a): $\beta(\infty)$ and $b(\infty)$ positive:} If the right end points are positive, then, without loss of generality, we assume them to be $1$. In Theorem~\ref{main theorem negative}, we show that if $X$ and $Y$ both have positive right end points, then $(1-XY)^{-1}$ has regularly varying tail of index $-1/|\rho|$, under some further sufficient moment conditions. In Section~\ref{section:example}, we give an example where the moment condition fails, yet the product shows the tail behavior predicted by Theorem~\ref{main theorem negative}.

{\bf Subcase II(b):  $\beta(\infty)$ and $b(\infty)$  zero:} In Theorem~\ref{theorem:both zero} we show that if both right end points are zero, then the product convergence is a simple consequence of the result in Case I. In this case $(XY)^{-1}$ has regularly varying tail of index ${-1/(|\rho|+|\gamma|)}$.

{\bf Subcase II(c): $\beta(\infty)$ zero and $b(\infty)$ positive:} We show in Theorem~\ref{theorem: X zero} that if $Y$ is a nonnegative random variable having positive right end point, then under some appropriate moment conditions $-(XY)^{-1}$ has regularly varying tail of index ${-1/|\rho|}$.

{\bf Subcase II(d): $\beta(\infty)$ and $b(\infty)$  negative:} When both the right end points are negative, then, without loss of generality, we assume them to be $-1$. In Theorem \ref{theo: both right end point negative}, we show that $(XY-1)^{-1}$ has regularly varying tail of index ${-1/|\rho|}$.

There are a few more cases beyond the four subcases considered above, when both $\rho$ and $\gamma$ are negative. For example, consider the case when $Y$ has right end point zero and $X$ has positive right end point $\beta(\infty)$. By our discussion above, $X$ should have the support $[0,\beta(\infty)]$. Again, the product has right end point $0$ and the behavior of $X$ around zero becomes important. Thus, to get something interesting in this case one must have a conditional model which gives adequate information about the behavior of $X$ around the left end point. So it becomes natural to model $(-X, Y)$, which has already been considered in Subcase II(b). A similar situation occurs when $\beta(\infty)<0$ and $b(\infty)>0$. Here, again, the problem reduces to that in Subcase II(d) by modeling $(X, -Y)$.  We refer to Remark~\ref{rem: not done} for a discussion on this subcase.

{\bf Case III: $\rho$ positive and $\gamma$ negative:} In this case we assume $b(\infty)>0$ and also $\alpha(t)\sim 1/a(t)$ which implies that $\rho=-\gamma$. We  show in  Theorem~\ref{theorem case 3} that $XY$ has regularly varying tail of index ${-1/|\gamma|}$.

{\bf Case IV: $\rho$ negative and  $\gamma$ positive:} In Theorem~\ref{case 4} we show that $XY$ has regularly varying tail of index ${-1/\gamma}$.

Finally we end this section by summarizing the results in a tabular form:

\begin{center}
\begin{table}[h]
\caption{Behavior  of  products}	\label{tab:BehaviorOfProducts}
	\begin{tabular}{|c|c|c| c| c| c|}
	\hline \hline
	Index of $\alpha$     & Index of $a$     &Theorem number  & Nature     & Regular variation    \\[0.5ex]
	\hline
	$\rho>0$ & $\gamma >0$ & Theorem~\ref{case1[a]} & $XY$ &$\text{RV}_{-{1}/{(\gamma+\rho)}}$ \\
	\hline
	$\stackrel{\displaystyle \rho>0}{\alpha\sim\frac{1}{a}}$ & $\gamma=-\rho<0$  &  Theorem~\ref{theorem case 3} & $XY$ & $\text{RV}_{-{1}/{|\gamma|}}$ \\
	%& &  & & \\
	\hline
	$\rho<0$ & $\gamma<\rho$ &  Theorem~\ref{main theorem negative}& $(\beta(\infty)b(\infty)-XY)^{-1}$ & $\text{RV}_{-{1}/{|\rho|}}$ \\
	\hline
	$\stackrel{\displaystyle\rho<0}{\alpha\sim a}$ & $\gamma=\rho$ & Theorem~\ref{main theorem negative} &  $(\beta(\infty)b(\infty)-XY)^{-1}$ &$\text{RV}_{-{1}/{|\rho|}}$ \\
	%$$ & & & &   \\
	\hline	
$\rho<0$& $\gamma<0$  & Theorem~\ref{theorem:both zero} & $(XY)^{-1}$ & $\text{RV}_{-{1}/(|\gamma|+|\rho|)}$ \\
\hline
$\rho<0$& $\gamma<0$  & Theorem~\ref{theorem: X zero}& $-(XY)^{-1}$ & $\text{RV}_{-{1}/{|\rho|}}$ \\
\hline
$\rho<0$& $\gamma>0$ & Theorem~\ref{case 4} & XY& $\text{RV}_{1/\gamma}$\\
\hline
\hline
\end{tabular}	
\end{table}
\end{center}

\end{section}

\begin{section}{Some transformations of CEVM according to parameters $\gamma$ and $\rho$} \label{Transformations}
In this section we reduce the basic convergence in Condition~\eqref{basic1} to an equivalent convergence in some appropriate subspace of $\mathbb R^2$ to facilitate  our calculations of  product of two variables following conditional extreme value model. We now discuss the four cases considered in Section~\ref{sec: overview}.

{\bf Case I: $\rho$ and $\gamma$ positive:} In this case we assume $Y$ is nonnegative. Let $(X,Y)\in CEVM(\alpha,\beta;a,b;\mu)$ on $[-\infty,\infty]\times \Ebar$. Now by the domain of attraction condition~\eqref{chap1:def:doa} and  Corollary 1.2.4 of \cite{Haan:Ferreira:2006}, we have,  $b(t)\sim a(t)/\gamma$, as $t\to\infty$. Also, from Theorem 3.1.12 (a),(c) of \cite{bingham:goldie:teugels:1987} it follows that
  $$\lim_{t\rightarrow\infty}\frac{\beta(t)}{\alpha(t)}=
  \begin{cases}
  0 &\text{when $\psi_2=0$} \\
  \frac{1}{\rho} &\text{when $\psi_2\neq 0$}.
  \end{cases}$$
Now using the above conditions and translating $X$ and $Y$ coordinates we get that $(X,Y)$ satisfies Conditions~\eqref{basic1} and~\eqref{nondegenerate} with parameters $(\alpha,0;a,0;\nu)$ on $D:=[-\infty,\infty]\times(0,\infty]$, for some nondegenerate measure $\nu$ which is obtained from $\mu$ by translations on both axes. So in Theorem~\ref{case1[a]} which deals with product $XY$ in Case I, we assume that $(X,Y)$ satisfies Conditions~\eqref{basic1} and~\eqref{nondegenerate} with parameters $(\alpha,0;a,0;\nu)$ on $D=[-\infty,\infty]\times(0,\infty]$ for some nondegenerate Radon measure $\nu$.

{\bf Case II: $\rho$ and $\gamma$ negative:}  Recall that in this case $\Ebar=\left(-\infty,\frac{1}{|\gamma|}\right]$.  Since $Y\in D(G_{\gamma})$ with $\gamma<0$ , it follows from Lemma 1.2.9 of \cite{Haan:Ferreira:2006} that  $\lim_{t\rightarrow\infty}b(t)=:b(\infty)$ exists and is finite and as $t\rightarrow\infty$ we have, $$\frac{b(\infty)-b(t)}{a(t)}\rightarrow\frac{1}{|\gamma|}.$$ Moreover $b(\infty)$ turns out to be the right end point of $Y$. Hence in this case, without loss of generality we take $a(t)=|\gamma|(b(\infty)-b(t))$ and it easily follows that, for $y>0$,
$$\lim_{t\rightarrow\infty}t\prob\left[\frac{\widetilde Y}{a(t)^{-1}}>y\right]=  y^{\frac{1}{\gamma}},$$ where $\widetilde Y$ is defined in~\eqref{eq:x-y-transformation}.  Now observe that $(X,Y)\in CEVM(\alpha,\beta;a,b;\mu)$ on $[-\infty,\infty]\times \Ebar$ gives $(X,\widetilde Y)$ which satisfies Conditions~\eqref{basic1} and~\eqref{nondegenerate} with parameters $(\alpha,\beta;1/a(t),0;\mu_2)$ on $D$, where, $$\mu_2 ([-\infty,x]\times(y,\infty])=\mu( [-\infty,x]\times(\frac{1}{|\gamma|}-\frac{1}{y},\infty]).$$

Now since $\rho<0$, we get by  Theorem B.22 of \cite{Haan:Ferreira:2006} that $\lim_{t\rightarrow\infty}\beta(t)=\beta(\infty)$ exists and is finite. It may happen that $X$ has a different  right end point than $\beta(\infty)$,  but we assume $\beta(\infty)$ to be its right end point to avoid complications. In  $X$ coordinate we can do a similar transformation as the $Y$ variable, to get
\begin{align*}
K_t(\alpha(t)x+\beta(t))&:=\prob\left[  \left. \frac{X-\beta(t)}{\alpha(t)}\leq x \right|\frac{(b(\infty)-Y)^{-1}}{a(t)^{-1}}>y\right]\\
&\rightarrow y^{-\frac1\gamma}\mu_2([-\infty,x]\times(y,\infty])=:K(x).
\end{align*}
When $\psi_2\neq0$, we have, as $t\to\infty$,
\begin{equation}\label{eq:beta:negative:conv}
\frac{\beta(\infty)-\beta(t)}{\alpha(t)}\rightarrow\frac{1}{|\rho|}.\end{equation}

Now by convergence of types theorem (see Theorem~\ref{theo:convergence of types}), as $t\to\infty$, we have, $$K_t(|\rho|(\beta(\infty)-\beta(t))x+\beta(t))\rightarrow K(x).$$ Define,
\begin{equation}\label{eq:negative:alpha:a}
\widetilde{\alpha}(t)=
\begin{cases}
\frac1{|\rho|(\beta(\infty)-\beta(t))} &\text{when~} \psi_2\neq0\\
\frac1{\alpha(t)} &\text{when~} \psi_2=0\\
\end{cases}
\qquad \text{and} \qquad
\widetilde{a}(t)=\frac1{a(t)}.
\end{equation}
Using~\eqref{eq:negative:alpha:a} we get, as $t\rightarrow\infty$,
\begin{equation}
\label{basic2a}
t\prob\left[\frac{\widetilde X}{\widetilde{\alpha}(t)}\leq x,\frac{\widetilde Y}{\widetilde{a}(t)}>y\right]\to
\begin{cases}
\mu_2([-\infty,-\frac{1}{x}+\frac{1}{|\rho|}]\times(y,\infty]) &\text{for }\psi_2\neq 0\\
 \mu_2([-\infty,-\frac{1}{x}]\times(y,\infty]) &\text{for }\psi_2=0.\\
\end{cases}
\end{equation}
In Section~\ref{mainresults}, we deal with Case II by breaking it up into different subcases as pointed out in Section~\ref{sec: overview}. So, in Theorems~\ref{main theorem negative}--\ref{theo: both right end point negative}, we assume that $(\widetilde X,\widetilde Y)$  satisfy Conditions~\eqref{basic1} and~\eqref{nondegenerate} with parameters $(\widetilde{\alpha},0;\widetilde{a},0;\nu)$ on $[0,\infty]\times(0,\infty]$, for some nondegenerate Radon measure $\nu$.

{\bf Case III:  $\rho$ positive and $\gamma$ negative: }
Since $Y\in D(G_{\gamma})$, we can do a transformation similar to that in Case II. So in this case  $(X,\widetilde Y)$  satisfies Conditions~\eqref{basic1} and~\eqref{nondegenerate} with parameters $(\alpha,\beta;\widetilde{a},0;\mu_3)$ on $[-\infty,\infty]\times(0,\infty]$, for some nondegenerate measure $\mu_3$.

Now since $\rho>0$, we can do a translation in the first coordinate to get $(X,\widetilde Y)$  which satisfies Conditions~\eqref{basic1} and~\eqref{nondegenerate} with parameters $(\alpha,0;\widetilde{a},0;\nu)$ on $[-\infty,\infty]\times(0,\infty]$ for some nondegenerate measure $\nu$. In Theorem~\ref{theorem case 3} we deal with product behavior in this case.

{\bf Case IV:  $\rho$ negative and $\gamma$ positive: }
We assume that $\beta(\infty)$, the right end point of X is positive. Now, as in Case II, we use~\eqref{eq:beta:negative:conv} to get the following convergence for $x\geq 0$ and $y>0,$
\begin{align*}
   t\prob\left[\frac{(\beta(\infty)-X)}{\alpha(t)}\leq x,\frac{Y}{a(t)}>y\right]&=t\prob\left[\frac{X-\beta(t)}{\alpha(t)}\geq -x+\frac{\beta(\infty)-\beta(t)}{\alpha(t)},\frac{Y}{a(t)}>y\right]\\
&\rightarrow\mu\left([-x+\frac{1}{|\gamma|},\infty]\times (y,\infty]\right) \text{ as } t\rightarrow\infty.
\end{align*}
So in Theorem~\ref{case 4}, which derives the product behavior in Case IV, we assume  $(\beta(\infty)-X, Y)$  satisfies Conditions~\eqref{basic1} and~\eqref{nondegenerate} with parameters $(\alpha,0;a,0;\nu)$ on $[0,\infty]\times(0,\infty]$, for some nondegenerate Radon measure $\nu$.

We thus observe that if $(X, Y)\in CEVM(\alpha, \beta; a, b; \mu)$, then in Cases I, II, III and IV respectively, $(X, Y)$, $(\widetilde X, \widetilde Y)$, $(X, \widetilde Y)$ and $(\beta(\infty)-X, Y)$ satisfy Conditions~\eqref{basic1} and~\eqref{nondegenerate} with some positive scaling parameters, zero centering parameters and a nondegenerate limiting Radon measure on $D=[-\infty, \infty]\times (0, \infty]$. In future sections, whenever we refer to Conditions~\eqref{basic1} and~\eqref{nondegenerate} with respect to the transformed variables alone without any reference to the CEVM model for the original pair $(X, Y)$, we shall denote, by an abuse of notation, the limiting Radon measure for the transformed random variables as $\mu$ as well.
\end{section}

\begin{section}{Behavior of the product under conditional model}
\label{mainresults}
Now we study the product behavior when $(X,Y)\in CEVM(\alpha,\beta;a,b;\mu)$. In all the cases we assume Conditions~\eqref{basic1} and~\eqref{nondegenerate} on the suitably transformed versions of $(X, Y)$, so that centering is not required.

{\bf Case I: $\rho$ and $\gamma$ positive:} We begin with the case where both $\rho$ and $\gamma$ are positive. 
\begin{theorem}\label{case1[a]} Let $\rho>0,\gamma>0$ and $Y$ be a nonnegative random variable. Assume $(X,Y)$ satisfies Conditions~\eqref{basic1} and~\eqref{nondegenerate} with parameters $(\alpha,0;a,0;\mu)$ on $D:=[-\infty,\infty]\times(0,\infty]$. Also assume
  \begin{equation}
  \label{moment condition}
   \lim_{\epsilon\downarrow0}\limsup_{t\rightarrow\infty}t\prob\left[\frac{|X|}{\alpha(t)}>\frac{z}{\epsilon}\right]=0.
  \end{equation}
   Then, $XY$ has regularly varying tail of index $-{1/(\gamma+\rho)}$ and $t\prob\left[XY/{\left(\alpha(t)a(t)\right)}\in\cdot\right]$ converges vaguely to some nondegenerate Radon measure on $[-\infty,\infty]\setminus \{0\}$.
\end{theorem}
\begin{remark} Please see Section~\ref{section:remarks on moments} for the cases where the condition~\eqref{moment condition} holds.
\end{remark}
\begin{proof}
For $\epsilon>0$ and $z>0$ observe that the set, $$A_{\epsilon,z}=\{(x,y)\in D :xy>z, y>\epsilon \}$$ is a relatively compact set in $D$ and $\mu$ is a Radon measure. Note that,
\begin{multline}
t\prob\left[\left(\frac{XY}{\alpha(t)a(t)},\frac{Y}{a(t)}\right)\in A_{\epsilon,z}\right]\le t\prob\left[\frac{XY}{\alpha(t)a(t)}>z\right] \\ \le  t\prob\left[\left(\frac{XY}{\alpha(t)a(t)},\frac{Y}{a(t)}\right)\in A_{\epsilon,z}\right] + t\prob\left[\frac{XY}{\alpha(t)a(t)}>z,\frac{Y}{a(t)}\leq \epsilon\right]
\end{multline}
First letting $t\to\infty$ and then $\epsilon\downarrow0$ through a sequence such that $A_{\epsilon,z}$ is a $\mu$-continuity set, the left side and the first term on the right side converge to $\mu\{(x,y)\in D: xy>z\}$. The second term on the right side is negligible by the assumed tail condition~\eqref{moment condition}. Combining all, we have the required result when $z>0$.
%Hence there exists a sequence $\epsilon_k\downarrow 0$ such that $\mu(\partial A_{\epsilon_k,z})=0$ for all $k$. For the lower bound, we have,
%$$\liminf_{t\rightarrow\infty}t\prob\left[\frac{XY}{\alpha(t)a(t)}>z\right] \geq \liminf_{t\rightarrow\infty}t\prob\left[\frac{XY}{\alpha(t)a(t)}>z,\frac{Y}{a(t)}>\epsilon_k\right] = \mu(A_{\epsilon_k,z}).$$
%Hence as $k\rightarrow\infty$, $\mu(A_{\epsilon_k,z})\rightarrow \mu\{(x,y)\in D: xy > z\}.$
%Also,
%\begin{align*}
 % \limsup_{t\rightarrow\infty}t\prob\left[\frac{XY}{\alpha(t)a(t)}>z\right]&\leq \limsup_{t\rightarrow\infty}t\prob\left[\frac{XY}{\alpha(t)a(t)}>z,\frac{Y}{a(t)}>\epsilon_k\right]\\ &\qquad+\limsup_{t\rightarrow\infty}t\prob\left[\frac{XY}{\alpha(t)a(t)}>z,\frac{Y}{a(t)}\leq \epsilon_k\right]&\\
  %&\le\mu(A_{\epsilon_k,z})+\limsup_{t\rightarrow\infty}t\prob\left[\frac{|X|}{\alpha(t)}>\frac{z}{\epsilon_k}\right]
%\end{align*}
%Now due to the given tail condition~\eqref{moment condition} on $X$, as $k\rightarrow\infty$ we get,
%$$\limsup_{t\rightarrow\infty} t\prob\left[\frac{XY}{\alpha(t)a(t)}>z\right]\leq \mu\{(x,y)\in D: xy > z\}.$$
By similar arguments one can show the above convergence on the sets of the form $(-\infty,-z)$ with $z>0$, also.
\end{proof}

\noindent{\bf Spectral form for product:}
For simplicity let us assume that $X$ is nonnegative as well. Then the vague convergence in $M_+(D)$ can be thought of as vague convergence in $M_+([0,\infty]\times(0,\infty]).$  By Theorem \ref{case1[a]} we have,
$$\lim_{t\rightarrow\infty}t\prob\left[\frac{XY}{\alpha(t)a(t)}>z\right]= \mu\{(x,y)\in [0,\infty]\times(0,\infty] : xy > z\}.$$
Since $\rho>0$ and $\gamma>0$, it is known that there exists $\overline{\alpha}(t)\sim \alpha(t)$ and $\overline{a}(t)\sim a(t)$ such that they are eventually differentiable and strictly increasing. Also $\frac{\overline{\alpha}(tx)}{\alpha(t)}\rightarrow x^{\rho}$ and $\frac{\overline{a}(tx)}{a(t)}\rightarrow x^{\gamma}$ as $t\rightarrow\infty.$
Hence,
\begin{align*}
   t\prob\left[\frac{\overline{\alpha}^{\leftarrow}(X)}{t}\leq x, \frac{\overline{a}^{\leftarrow}(Y)}{t}>y\right]&=t\prob\left[\frac{X}{\alpha(t)}\leq \frac{\overline{\alpha}(tx)}{\alpha(t)},\frac{Y}{a(t)}>\frac{\overline{a}(ty)}{a(t)}\right]\\
   &\rightarrow \mu([0,x^{\rho}]\times(y^{\gamma},\infty])\\
   &=\mu T_1^{-1}([0,x]\times(y,\infty]),
 \end{align*}
where $T_1(x,y)=(x^{1/\rho},y^{1/\gamma})$.
Let $S$ be the spectral measure for the standardized pair $(\overline{\alpha}^{\leftarrow}(X),\overline{a}^{\leftarrow}(Y))$ corresponding to $\mu T_1^{-1}$. Then,
  \begin{align*}
  &\mu\{(x,y)\in[0,\infty]\times(0,\infty]: xy > z\}\\
  =&\mu T_1^{-1}\{(x,y)\in [0,\infty]\times(0,\infty] : x^{\rho}y^{\gamma} > z\}\\
  =&\int_{\omega\in[0,1)}\int_{r^{\rho+\gamma}\omega^{\rho}(1-\omega)^{\gamma}>z} r^{-2}dr S(d\omega)\\
  =&\int_{\omega\in[0,1)}\int_{r>\frac{z^{\frac{1}{\rho+\gamma}}}{(\omega^{\rho}(1-\omega)^{\gamma})^{\frac{1}{\rho+\gamma}}}}r^{-2}drS(d\omega)\\
  =&z^{-\frac{1}{\rho+\gamma}}\int_{\omega\in[0,1)}\omega ^{\frac{\rho}{\rho+\gamma}}(1-\omega)^{\frac{\gamma}{\rho+\gamma}}S(d\omega).\\
  \end{align*}
So finally we have,
\begin{align*}\lim_{t\rightarrow\infty}t\prob\left[\frac{XY}{\alpha(t)a(t)}>z\right]&= \mu\{(x,y)\in [0,\infty]\times(0,\infty] : xy > z\}\\
&=z^{-\frac{1}{\rho+\gamma}}\int_{\omega\in[0,1)}\omega ^{\frac{\rho}{\rho+\gamma}}(1-\omega)^{\frac{\gamma}{\rho+\gamma}}S(d\omega).\\
\end{align*}

{\bf Case II: $\rho$ and $\gamma$ negative:} As has been already pointed out, $\beta(\infty)$ need not be the right endpoint $X$. However, we shall assume it to be so. The tail behavior of $XY$ strongly depends on the right end points of $X$ and $Y$. There are several possibilities which may arise, but it may not always be possible to predict the tail behavior of $XY$ in all the cases. We shall deal with few interesting cases. See Section~\ref{sec: overview} for a discussion in this regard. Regarding one of the cases left out, see Remark~\ref{rem: not done}. Recall that, from~\eqref{eq:x-y-transformation} we have,
\begin{equation}\label{transformation}
\frac{1}{\beta(\infty)b(\infty)-XY}=\frac{\widetilde{X}\widetilde{Y}}{\beta(\infty)\widetilde{X}+b(\infty)\widetilde{Y}-1}
\end{equation}

{\bf Subcase II(a): $\beta(\infty)$ and $b(\infty)$ positive:} After scaling $X$ and $Y$ suitably, without loss of generality, we can assume that $\beta(\infty)=1=b(\infty).$
\begin{theorem}\label{main theorem negative}
Suppose $X$ and $Y$ are nonnegative and $(\widetilde X,\widetilde Y)$ satisfies Conditions~\eqref{basic1} and~\eqref{nondegenerate} with parameters $(\widetilde{\alpha},0;\widetilde{a},0;\mu)$ on $[0,\infty]\times(0,\infty]$. Assume $\E\left[\widetilde{X}^{{1}/{|\rho|}+\delta}\right]<\infty$ for some $\delta>0$ and  either  $\gamma < \rho $ or  $\widetilde{\alpha}(t)/\widetilde{a}(t)$ remains bounded. Then $(1-XY)^{-1}$ has regularly varying tail of index ${-1/|\rho|}$ and, as $t\to\infty$, $t\prob\left[{(1-XY)^{-1}}/{\widetilde{\alpha}(t)}\in\cdot \right]$ converges vaguely to some nondegenerate Radon measure on $(0,\infty]$.
\end{theorem}

We start with a technical lemma.
\begin{lemma}\label{important lemma}
For $0\leq t_1\leq t_2\leq\infty$ and $z>0$, we denote the set
\begin{equation}\label{eq:set:V}
V_{[t_1,t_2],z}=\{(x,y)\in [0,\infty]\times(0,\infty]:x\in[t_1,t_2], y>z\}.
\end{equation} 
Suppose that $\{(Z_{1t},Z_{2t})\}$ is a sequence of pairs of nonnegative random variables and there exists a Radon measure $\nu(\cdot)$ in $[0,\infty]\times(0,\infty]$  such that they satisfy the following two conditions.\\
Condition A: For $0<t_1<t_2<\infty$ and $z>0$, whenever $V_{[t_1,t_2],z}$ is a $\nu$-continuity set, we have, as $t\rightarrow\infty$,
$$t\prob\left[(Z_{1t},Z_{2t})\in V_{[t_1,t_2],z}\right]\rightarrow \nu(V_{[t_1,t_2],z}).$$
Condition B: For any  $z_0\in(0,\infty)$ we have as $t\rightarrow\infty,$
$$t\prob\left[(Z_{1t},Z_{2t})\in V_{[0,\infty],z_0}\right]\rightarrow f(z_0)\in(0,\infty).$$
Then, as $t\to\infty$,
$$t\prob\left[(Z_{1t},Z_{2t})\in\cdot\right]\vaguec\nu(\cdot)$$ in $M_+([0,\infty]\times(0,\infty])$
\end{lemma}
\begin{proof}
Fix a $z_0\in(0,\infty)$ and define the following probability measures on $[0,\infty]\times(z_0,\infty)$:
$$Q_t(\cdot)=\frac{t\prob\left[(Z_{1t},Z_{2t})\in \cdot\right]}{t\prob\left[(Z_{1t},Z_{2t})\in V_{[0,\infty],z_0}\right]} \text{ and } Q(\cdot)=\frac{\nu(\cdot)}{f(z_0)}.$$ From condition A and B it follows that, $$Q_t(V_{[t_1,t_2],z})\rightarrow Q(V_{[t_1,t_2],z})$$  as $t\rightarrow\infty$, whenever $V_{[t_1,t_2],z}$ is $\nu$-continuity set. Now following the arguments in the proof of Theorem 2.1 of \cite{Maulik:Resnick:Rootzen:2002}, it follows that $Q_t$ converges weakly to $Q$ on $[0,\infty]\times(z_0,\infty)$. Since a Borel set with boundary having zero $Q$ measure is equivalent to having measure zero with respect to the measure $\nu$ we have that $t \prob\left[(Z_{1t},Z_{2t})\in B\right]\rightarrow\nu(B)$ for any Borel subset $B$ of $[0,\infty]\times (z_0,\infty]$, having boundary with zero $\nu$ measure.

Let $K$ be a $\nu$-continuity set as well as a relatively compact set in $[0,\infty]\times(0,\infty]$. Then there exists $z_0>0$ such that $K\subset[0,\infty]\times(z_0,\infty]$. Then $K$ is Borel in $[0,\infty]\times(z_0,\infty)$ and also a $\nu$-continuity set. Hence we have,
$$t\prob\left[(Z_{1t},Z_{2t})\in K\right]=t\prob\left[(Z_{1t},Z_{2t})\in K\right]\rightarrow \nu(K).$$ This shows that $t\prob\left[(Z_{1t},Z_{2t})\in \cdot\right]$ vaguely converges to $\nu$ on $[0,\infty]\times(0,\infty]$.
\end{proof}

From~\eqref{transformation}, the behavior of $XY$ will be determined by the pair $(\widetilde X \widetilde Y, \widetilde X + \widetilde Y - 1)$. So we next prove a result about the joint convergence of product and the sum.
\begin{lemma}\label{lemma:joint: prod:sum}
Let $\gamma<0,\rho<0$ and $(\widetilde X,\widetilde Y)$  satisfies Conditions~\eqref{basic1} and~\eqref{nondegenerate} with parameters $(\widetilde{\alpha},0;\widetilde{a},0;\mu)$. If  $\E\left[\widetilde{X}^{{1}/{|\gamma|}+\delta}\right]<\infty$ for some $\delta>0,$ then~$(\widetilde X\widetilde Y,\widetilde X+\widetilde Y-1)$ also satisfies Conditions~\eqref{basic1} and~\eqref{nondegenerate} with parameters $(\widetilde{\alpha}\widetilde{a},0;\widetilde{a},0;\mu T_2^{-1})$ on  $[0,\infty]\times(0,\infty]$ where $T_2(x,y)=(xy,y)$.
\end{lemma}
\begin{proof}
First observe that from the compactification arguments used in Lemma~\ref{important lemma} and the basic convergence in Condition~\eqref{basic1} satisfied by the pair $(\widetilde X,\widetilde Y)$, it follows that
\begin{equation} \label{T_2 convergence}
t\prob\left[\left(\frac{\widetilde{X}\widetilde{Y}}{\widetilde{\alpha}(t)\widetilde{a}(t)},\frac{\widetilde{Y}}{\widetilde{a}(t)}\right)\in\cdot\right]\vaguec \mu T_2^{-1}(\cdot)\quad \text{in}\quad  M_+([0,\infty]\times(0,\infty]).
\end{equation}
Let $0\leq t_1\leq t_2\leq \infty$ and $z>0$. Assume that $\mu T_2^{-1}(\partial V_{[t_1,t_2],z})=0$, where $V_{[t_1,t_2],z}$ is defined in~\eqref{eq:set:V}. Since $X$ is nonnegative, $\widetilde{X}$ is greater than or equal to $1$ and hence for a lower bound we get,
\begin{multline}
\liminf_{t\rightarrow\infty}t\prob\left[\left(\frac{\widetilde{X}\widetilde{Y}}{\widetilde{\alpha}(t)\widetilde{a}(t)}, \frac{\widetilde{X}+\widetilde{Y}-1}{\widetilde{a}(t)}\right)\in V_{[t_1,t_2],z}\right]\\ \geq\lim_{t\rightarrow\infty}t\prob\left[\frac{\widetilde{X}\widetilde{Y}}{\widetilde{\alpha}(t)\widetilde{a}(t)}\in[t_1,t_2],\frac{\widetilde{Y}}{\widetilde{a}(t)}>z\right]
=\mu T_2^{-1}(V_{[t_1,t_2],z}).\label{eq:negative:lower bound}
\end{multline}
For the upper bound, choose $0<\epsilon<z$, such that $1/|\widetilde a(t)| < \epsilon/2$ (since $\widetilde{a}(t)\in RV_{-\gamma}$) and $\mu T_2^{-1}\left(\partial V_{[t_1,t_2],z-\epsilon}\right)=0$. Then,
\begin{align*}
&t\prob\left[\left(\frac{\widetilde{X}\widetilde{Y}}{\widetilde{\alpha}(t)\widetilde{a}(t)},\frac{\widetilde{X}+\widetilde{Y}-1}{\widetilde{a}(t)}\right)\in V_{[t_1,t_2],z}\right]\\
\leq &t\prob\left[\left(\frac{\widetilde{X}\widetilde{Y}}{\widetilde{\alpha}(t)\widetilde{a}(t)},\frac{\widetilde{X}+\widetilde{Y}-1}{\widetilde{a}(t)}\right)\in V_{[t_1,t_2],z},\frac{\widetilde{X}}{\widetilde{a}(t)}\leq\frac{\epsilon}2\right]\\
& +t\prob\left[\left(\frac{\widetilde{X}\widetilde{Y}}{\widetilde{\alpha}(t)\widetilde{a}(t)},\frac{\widetilde{X}+\widetilde{Y}-1}{\widetilde{a}(t)}\right)\in V_{[t_1,t_2],z},\frac{\widetilde{X}}{\widetilde{a}(t)}>\frac{\epsilon}2\right]\\
\leq &t\prob \left[\frac{\widetilde{X}\widetilde{Y}}{\widetilde{\alpha}(t)\widetilde{a}(t)}\in[t_1,t_2],\frac{\widetilde{Y}}{\widetilde{a}(t)}>z-\epsilon\right]
+t\prob\left[\frac{\widetilde{X}}{\widetilde{a}(t)}>\frac{\epsilon}2\right]\\
\leq &t\prob\left[\left(\frac{\widetilde{X}\widetilde{Y}}{\widetilde{\alpha}(t)\widetilde{a}(t)},\frac{\widetilde{Y}}{\widetilde{a}(t)}\right)\in V_{[t_1,t_2],z-\epsilon}\right] + 2^{1/|\gamma|+\delta} t \frac{\E\left[\widetilde{X}^{{1}/{|\gamma|}+\delta}\right]}{(\widetilde{a}(t)\epsilon)^{\frac{1}{|\gamma|}+\delta}}.
\end{align*}
The first term converges to $\mu T_2^{-1}(V_{[t_1,t_2],z-\epsilon})$, while the second sum converges to zero, since $\widetilde{a}(t)\in RV_{-\gamma}$. Now letting $\epsilon\to0$ satisfying the defining conditions, we obtain the upper bound, which is same as the lower bound~\eqref{eq:negative:lower bound}. Thus we get that,
$$\lim_{t\rightarrow\infty}t\prob\left[(\frac{\widetilde{X}\widetilde{Y}}{\widetilde{\alpha}(t)\widetilde{a}(t)},\frac{\widetilde{X}+\widetilde{Y}}{\widetilde{a}(t)})\in V_{[t_1,t_2],z}\right]= \mu T_2^{-1}(V_{[t_1,t_2],z}).$$
Hence Condition A of Lemma~\ref{important lemma} is satisfied.

Now if we fix $z_0\in(0,\infty)$  and let $\epsilon>0$ satisfy the conditions as in the upper bound above, then
\begin{align*} t\prob\left[\left(\frac{\widetilde{X}\widetilde{Y}}{\widetilde{\alpha}(t)\widetilde{a}(t)},\frac{\widetilde{X}+\widetilde{Y}-1}{\widetilde{a}(t)}\right)\in V_{[0,\infty],z_0}\right]&= t\prob\left[\frac{\widetilde{X}+\widetilde{Y}-1}{\widetilde{a}(t)}>z_0\right]\\
&\leq t\prob\left[\frac{\widetilde{Y}}{\widetilde{a}(t)}>z_0-\epsilon \right]+t\prob\left[\frac{\widetilde{X}}{\widetilde{a}(t)}>\frac{\epsilon}2 \right]\\
&\rightarrow (z_0 -\epsilon)^{\frac{1}{\gamma}}.
\end{align*}
Hence the upper bound for the required limit in Condition B of Lemma~\ref{important lemma} follows by letting $\epsilon\rightarrow 0.$ The lower bound easily follows from the domain of attraction condition on $Y$ and the fact that $\widetilde X\ge 1$. So Condition B is also satisfied by the pair~$(\frac{\widetilde{X}\widetilde{Y}}{\widetilde{\alpha}(t)\widetilde{a}(t)},\frac{\widetilde{X}+\widetilde{Y}-1}{\widetilde{a}(t)})$ and hence the result follows from Lemma \ref{important lemma}.
\end{proof}

\begin{proof}[Proof of Theorem~\ref{main theorem negative}]
Denote $W'=\widetilde{X}\widetilde{Y}$, $W''=\widetilde{X}+\widetilde{Y}-1$ and note that $(1-XY)^{-1}=W'/W''$. So from previous lemma it follows  that,
$$t\prob\left[\left(\frac{W'}{\widetilde{\alpha}(t)\widetilde{a}(t)},\frac{W''}{\widetilde{a}(t)}\right)\in \cdot\right]\vaguec \mu T_2^{-1}(\cdot)\qquad \mbox{as~} t\rightarrow\infty.$$
Let $w\in(0,\infty)$ and $\epsilon>0$ and consider the set,
\begin{equation}\label{eq: neg: A set:descrip}
B_{w,\epsilon}=\{(x,y)\in [0,\infty]\times(0,\infty]: x>yw, y>\epsilon\}.\end{equation}
Then, for $\epsilon>0$, we have,
$$t\prob\left[\frac{W'}{W''}\frac{1}{\widetilde{\alpha}(t)}> w \right] = t \prob\left[\left(\frac{W'}{\widetilde{\alpha}(t)\widetilde{a}(t)}, \frac{W''}{\widetilde{a}(t)}\right)\in B_{w,\epsilon}\right] + t \prob\left[\frac{W'}{W''}\frac{1}{\widetilde{\alpha}(t)}> w, \frac{W''}{\widetilde{a}(t)}\leq \epsilon \right].$$
Since the set is bounded away from both the axes, the first sum converges to $\mu T_2^{-1}(B_{w,\epsilon})$ by the vague convergence of $(W', W'')$. Now since $\widetilde{X}\geq0$ we have $\widetilde{X}+\widetilde{Y}-1\geq \widetilde{Y}-1$, and hence for large $t$ we get,
\begin{align*}
&t\prob\left[\frac{(1-XY)^{-1}}{\widetilde{\alpha}(t)}>w, \frac{\widetilde{Y}-1}{\widetilde{a}(t)}\leq \epsilon_k\right]\\
\leq &t\prob\left[ XY>1-\frac{1}{w\widetilde{\alpha}(t)}, Y\leq 1-\frac{1}{\widetilde{a}(t)\epsilon_k +1} \right]\\
&\leq t\prob\left[\widetilde{X}>\frac{1-\frac{1}{\widetilde{a}(t)\epsilon_k+1}}{\frac{1}{w\widetilde{\alpha}(t)}-\frac{1}{(\widetilde{a}(t)\epsilon_k+1)}}\right]\\
&\leq C(t,w,k)\E\left[\widetilde{X}^{\frac1{|\rho|}+\delta}\right],
\end{align*}
where,
$$C(t,w,k)= \frac{t}{(\widetilde{\alpha}(t))^{\frac{1}{|\rho|}+\delta}}\left({1-\frac{1}{\widetilde{a}(t)\epsilon_k+1}}\right)^{-(\frac{1}{|\rho|}+\delta)} \left(\frac{1}{w}-\frac{1}{(\frac{\widetilde{a}(t)}{\widetilde{\alpha}(t)}\epsilon_k+\frac{1}{\widetilde{\alpha}(t)})}\right)^{\frac{1}{|\rho|}+\delta},$$
which goes to zero as $t\to\infty$, since $\widetilde{\alpha}(t)\in RV_{-\rho}$ and $\widetilde{\alpha}(t)/\widetilde{a}(t)$ remains bounded.
\end{proof}
\begin{remark}
Theorem~\ref{main theorem negative} requires that $(1/|\rho|+\delta)$-th moment of $\widetilde X$ is finite. However, this condition is not necessary. In the final section we give an example where this moment condition is not satisfied but we still obtain the tail behavior of the product.
\end{remark}

{\bf Subcase II(b): $\beta(\infty)=0$ and $ b(\infty)=0$:} In this case, both $X$ and $Y$ are nonpositive, but the product $XY$ is nonnegative. Thus, the right tail behavior of $XY$ will be controlled by the left tail behaviors of $X$ and $Y$, which we cannot control much using CEVM. However, CEVM gives some information about the left tail behavior of $XY$ at $0$, which we summarize below. 

Note that in this case, from~\eqref{eq:x-y-transformation} and~\eqref{eq:negative:alpha:a}, we have $\widetilde{X}=-1/{X}$, $\widetilde{Y}=-{1}/{Y}$ and $\widetilde{\alpha}(t)=-{1}/{(|\rho|\beta(t))}$, $\widetilde{a}(t)={1}/{a(t)}$. From Theorem~\ref{case1[a]}, the behavior of the product $XY$ around zero, or equivalently the behavior of the reciprocal of the product $\widetilde X\widetilde Y$ around infinity, follows immediately.
\begin{theorem}\label{theorem:both zero}
If $\rho<0$, $\gamma<0$ and $(\widetilde X,\widetilde Y)$  satisfies Conditions~\eqref{basic1} and~\eqref{nondegenerate} with parameters $(\widetilde{\alpha},0;\widetilde{a},0;\mu)$.  Also suppose, $$\lim_{\epsilon\downarrow0}\limsup_{t\rightarrow\infty}t\prob\left[\frac{\widetilde{X}}{\widetilde{\alpha}(t)}>\frac{z}{\epsilon}\right]=0$$  Then  $(XY)^{-1}$ has regularly varying tail with index $-1/(|\gamma|+|\rho|)$ and as $t\rightarrow\infty$, $t\prob\left[{(XY)^{-1}}/({\widetilde{\alpha}(t)\widetilde{a}(t)})\in\cdot\right]$ converge to some nondegenerate Radon measure on $(0,\infty]$.
\end{theorem}

{\bf Subcase II(c): $\beta(\infty)=0$ and $b(\infty)=1$:} Now note that from~\eqref{transformation} we have, $$-\frac{1}{XY}=\frac{\widetilde{X}\widetilde{Y}}{\widetilde{Y}-1}.$$
\begin{theorem}\label{theorem: X zero}
Suppose $Y$ is nonnegative and $(\widetilde X,\widetilde Y)$  satisfies Conditions~\eqref{basic1} and~\eqref{nondegenerate} with parameters $(\widetilde{\alpha},0;\widetilde{a},0;\mu)$. If $\E\left[\widetilde{X}^{{1}/{|\rho|}+\delta}\right]<\infty$ for some $\delta>0$, then  $-(XY)^{-1}$ has regularly varying tail of index $-{1/|\rho|}$ and as $t\rightarrow\infty$, $t\prob\left[{-(XY)^{-1}}/{\widetilde{\alpha}(t)}\in\cdot\right]$ converge vaguely to some nondegenerate measure on $(0,\infty]$.
\end{theorem}
\begin{proof}
Under the given hypothesis, and the fact that $\widetilde{a}(t)\to \infty$, it can be shown by arguments similar to Lemma~\ref{lemma:joint: prod:sum} that,
%\begin{equation}\label{eq:theo:vague:pos:neg:case}
$ t\prob\left[\left({\widetilde{X}\widetilde{Y}}/({\widetilde{\alpha}(t)\widetilde{a}(t)}), ({\widetilde{Y}-1})/{\widetilde{\alpha}(t)}\right)\in\cdot\right]$ converges vaguely to some nondegenerate Radon measure in $M_+([0,\infty]\times(0,\infty]$.

Next for $z>0$ and $\epsilon>0$, so that the set $B_{z,\epsilon}$ as in~\eqref{eq: neg: A set:descrip} is a continuity set of the limit measure. Note that
\begin{multline*}
 t\prob\left[\left(\frac{\widetilde{X}\widetilde{Y}}{\widetilde{\alpha}(t)\widetilde{a}(t)}, \frac{\widetilde{Y}-1}{\widetilde{\alpha}(t)}\right)\in B_{z,\epsilon}\right]\leq t\prob\left[\frac{-(XY)^{-1}}{\widetilde{\alpha}(t)}>z\right]\\ \leq t\prob\left[\left(\frac{\widetilde{X}\widetilde{Y}}{\widetilde{\alpha}(t)\widetilde{a}(t)}, \frac{\widetilde{Y}-1}{\widetilde{\alpha}(t)}\right)\in B_{z,\epsilon}\right]+t\prob\left[\frac{\widetilde{X}\widetilde{Y}}{\widetilde{\alpha}(t)(\widetilde{Y}-1)} >z,\frac{\widetilde{Y}-1}{\widetilde{a}(t)}\leq\epsilon \right].
\end{multline*}

The term on the left side and the first term on the right side converge as $t\to\infty$ due to the vague convergence mentioned at the beginning of the proof. Letting $\epsilon\downarrow0$ appropriately, we get the required limit. For the second term on the right side observe that,
 \begin{align*}
t\prob\left[\frac{\widetilde{X}\widetilde{Y}}{\widetilde{\alpha}(t)(\widetilde{Y}-1)} >z,\frac{\widetilde{Y}-1}{\widetilde{a}(t)}\leq\epsilon \right]&
\leq t\prob\left[\frac{-(XY)^{-1}}{\widetilde{\alpha}(t)}>z,\widetilde{Y}\leq \widetilde{a}(t)\epsilon+1\right]\\
&\leq t\prob\left[\frac{-(XY)^{-1}}{\widetilde{\alpha}(t)}>z,1-Y\geq\frac{1}{\widetilde{a}(t)\epsilon+1}\right]\\
&\leq t\prob\left[\frac{XY}{\alpha(t)}>-\frac{1}{z},Y\leq 1-\frac{1}{\widetilde{a}(t)\epsilon+1}\right]\\
&\leq t\prob\left[-X< \frac{\alpha(t)/z}{1-\frac{1}{\widetilde{a}(t)\epsilon+1}}\right]=t\prob\left[\widetilde{X}> \frac{1-\frac{1}{\widetilde{a}(t)\epsilon+1}}{\alpha(t)/z}\right]\\
&\leq t\E\left[\widetilde{X}^{{1}/{|\rho|}+\delta}\right]\left(\frac{\alpha(t)/z}{1-\frac{1}{\widetilde{a}(t)\epsilon+1}}\right)^{\frac{1}{|\rho|}+\delta}.
\end{align*}
The last expression tends to zero as $\widetilde{a}(t)\rightarrow\infty$ and $\alpha\in RV_\rho$.
\end{proof}

{\bf Subcase II(d): $\beta(\infty)$ and $b(\infty)$ negative:} As in Subcase II(a), after suitable scaling, without loss of generality, we can assume that $\beta(\infty)=b(\infty)=-1$.
Again, from~\eqref{transformation}, we have
$$\frac1{XY-1}=\frac{\widetilde{X}\widetilde{Y}}{\widetilde{X}+\widetilde{Y}+1}$$
and to get the behavior of the product around $1$ we first need to derive the joint convergence of $(\widetilde{X}\widetilde{Y}, \widetilde{X}+\widetilde{Y}+1)$. Using an argument very similar to Theorem~\ref{main theorem negative}, we immediately obtain the following result.
\begin{theorem} \label{theo: both right end point negative}
Let $(\widetilde X,\widetilde Y)$  satisfy Conditions~\eqref{basic1} and~\eqref{nondegenerate} with parameters $(\widetilde{\alpha},0;\widetilde{a},0;\mu)$ and $\E\left[\widetilde{X}^{\frac{1}{|\rho|}+\delta}\right]<\infty$ for some $\delta>0$. If either $\gamma < \rho $ or $\widetilde{\alpha}(t)/\widetilde{a}(t)$ remains bounded, then $(XY-1)^{-1}$ has regularly varying tail of index ${-1/|\rho|}$ and as $t\rightarrow\infty$,
$t\prob\left[{(XY-1)^{-1}}/{\widetilde{\alpha}(t)}\in\cdot \right]$ converges vaguely to some nondegenerate Radon measure on $(0,\infty]$.
\end{theorem}

\begin{remark} \label{rem: not done}
The other case when $\beta(\infty)=1$ and $b(\infty)=0$ is not easy to derive from the information about the conditional extreme value model. In this case the right endpoint of the product is zero and behavior of $X$ around zero seems to be important. But the conditional model gives us the regular variation behavior around one and not around zero.
\end{remark}

{\bf Case III: $\rho$ positive and $\gamma$ negative:} In this case we shall assume that $X$ is nonnegative and the upper endpoint of $Y$, $b(\infty)$ is positive. If $b(\infty)\leq0$, then the behavior of $X$ around its lower endpoint will play a crucial role in the behavior of the product $XY$, which becomes negative. However, the behavior of $X$ around its lower endpoint is not controlled by the conditional model and we are unable to conclude about the product behavior when $b(\infty)\leq 0$. So we only consider the case $b(\infty)>0$. We also make the assumption that $\alpha(t)\sim\widetilde a(t)$, which requires that $\rho=|\gamma|$.
\begin{theorem}\label{theorem case 3}
Let $X$ be nonnegative and $Y$ have upper endpoint $b(\infty)>0$. Assume that $(X,\widetilde Y)$  satisfies Conditions~\eqref{basic1} and~\eqref{nondegenerate} with parameters$(\alpha,0;\widetilde{a},0;\mu)$ with $\alpha(t)\sim {1}/{a(t)}=\widetilde{a}(t)$ and $\E\left[X^{{1}/{|\gamma|}+\delta}\right]<\infty$ for some $\delta>0$, then $XY$ has regularly varying tail of index ${-1/|\gamma|}$ and as $t\rightarrow\infty$, we have $t\prob\left[{XY}/\widetilde a(t) \in\cdot\right]$ converges vaguely to a nondegenerate Radon measure on $(0,\infty]$.
\end{theorem}
\begin{proof} 
As $\alpha(t)\sim{\widetilde a(t)}$, convergence of types allows us to change $\alpha(t)$ to $\widetilde{a}(t)$ and hence $(X,\widetilde Y)$ satisfy Conditions~\eqref{basic1} and~\eqref{nondegenerate} with parameters $(\widetilde{a},0;\widetilde{a},0;\mu)$.  Using the fact that $\widetilde{Y}=(b(\infty)-Y)^{-1}$, we have  $$XY=X\frac{b(\infty)\widetilde{Y}-1}{\widetilde{Y}}.$$ Using arguments similar to Proposition~4 of \cite{heffernan:resnick:2007}, it can be shown that as $t\rightarrow\infty$,
$$t\prob\left[\left(\frac{X}{\widetilde{Y}},\frac{b(\infty)\widetilde{Y}}{\widetilde{a}(t)}\right)\in\cdot\right]\vaguec \mu T_3^{-1}(\cdot),$$ in $M_+([0,\infty]\times(0,\infty])$, where $T_3(x,y)=(x/y,b(\infty)y)$. Since $\widetilde{a}(t)\rightarrow\infty$, we further have, as $t\rightarrow\infty,$
$$t\prob\left[\left(\frac{X}{\widetilde{Y}},\frac{b(\infty)\widetilde{Y}-1}{\widetilde{a}(t)}\right)\in\cdot\right]\vaguec \mu T_3^{-1}(\cdot).$$
Now applying the map $T_2(x,y)=(xy,y)$ to the above vague convergence and using compactification arguments similar to that in the proof of Theorem~2.1 of \cite{Maulik:Resnick:Rootzen:2002}, we have,
$$t\prob\left[\left(\frac{X}{\widetilde{Y}} \frac{b(\infty)\widetilde{Y}-1}{\widetilde{a}(t)}, \frac{b(\infty)\widetilde{Y}-1}{\widetilde{a}(t)}\right)\in\cdot\right]\vaguec \mu T_3^{-1}T_2^{-1}(\cdot).$$
Recalling the facts that $XY = X (b(\infty) \widetilde Y - 1)/\widetilde Y$ and $\widetilde a(t)\to\infty$ and reversing the arguments in the second coordinate, we have,
\begin{equation} \label{vague 3}
t\prob\left[\left(\frac{XY}{\widetilde{a}(t)},\frac{\widetilde{Y}}{\widetilde{a}(t)}\right)\in\cdot\right]\vaguec \mu T_3^{-1}T_2^{-1}{\widetilde{T}_3}^{-1}(\cdot),
\end{equation}
where $\widetilde{T}_3(x,y)=(x,\frac{y}{b(\infty)}).$   If $\epsilon>0$ and $z>0$ we have following series of inequalities,
\begin{align*}
   t\prob\left[\frac{XY}{\widetilde{a}(t)}>z,\frac{\widetilde{Y}}{\widetilde{a}(t)}\leq\epsilon\right]&
   =t\prob\left[\frac{XY}{\widetilde{a}(t)}>z,\frac{1}{b(\infty)-Y}\leq \widetilde{a}(t)\epsilon\right]\\
   &=t\prob\left[\frac{XY}{\widetilde{a}(t)}>z,Y\leq b(\infty)-\frac{1}{\widetilde{a}(t)\epsilon}\right]\\
   &\leq t\prob\left[\frac{X}{\widetilde{a}(t)}>z\left(b(\infty)-\frac{1}{\widetilde{a}(t)\epsilon}\right)^{-1}\right]\\
   &\leq t\frac{\E\left[X^{\frac{1}{|\gamma|}+\delta}\right]}{\left(\widetilde{a}(t)z\right)^{\frac{1}{|\gamma|}+\delta}}
   \left(b(\infty)-\frac{1}{\widetilde{a}(t)\epsilon}\right)^{\frac{1}{|\gamma|}+\delta} \to 0,
\end{align*}
since $\widetilde a \in RV_{|\gamma|}$.

Now observe that
\begin{multline*}
t \prob\left[\frac{XY}{\widetilde a(t)}>z, \frac{\widetilde Y}{\widetilde a(t)}>\epsilon\right] \le t \prob\left[\frac{XY}{\widetilde a(t)}>z\right]\\ 
\le  t \prob\left[\frac{XY}{\widetilde a(t)}>z, \frac{\widetilde Y}{\widetilde a(t)}>\epsilon\right] + t \prob\left[\frac{XY}{\widetilde a(t)}>z, \frac{\widetilde Y}{\widetilde a(t)}\le\epsilon\right].\end{multline*}
The last term on the right side is negligible by the previous argument and the left side and the first term on the right side converge due to the vague convergence in~\eqref{vague 3}. The result then follows by letting $\epsilon\to 0$.
\end{proof}

{\bf Case IV: $\rho$ negative and $\gamma$ positive:}
In this case we shall assume that $Y$ is nonnegative, $\beta(\infty)>0$ and $\beta(\infty)$ is the upper endpoint of $X$. Arguing as in Case III, we neglect the possibility that $\beta(\infty)\leq0$. Thus we further assume $0\leq X\leq \beta(\infty)$. Since $X$ becomes bounded, the product of $XY$ inherits its behavior from the tail behavior of $Y$.
\begin{theorem}\label{case 4}
Assume that both $X$ and $Y$ are nonnegative random variables with $\beta(\infty)$ being the upper endpoint of $X$. Let $(\beta(\infty)-X, Y)$   satisfies Conditions~\eqref{basic1} and~\eqref{nondegenerate} with parameters $(\alpha,0;a,0;\mu)$ on $[0,\infty]\times(0,\infty]$, for some nondegenerate Radon measure $\mu$. Then $XY$ has regularly varying tail of index ${-1/\gamma}$ and for all $z>0$, we have
$$t\prob\left[\frac{XY}{a(t)}>z\right]\rightarrow z^{-\frac1\gamma}\beta(\infty)^{\frac1\gamma}, \text{ as } t\rightarrow \infty.$$
\end{theorem}
\begin{proof}
First we prove the upper bound which, in fact, does not use the conditional model. Observe that
\begin{align*}
 t\prob\left[\frac{XY}{a(t)}>z\right]&=t\prob\left[\frac{XY}{a(t)}>z,X<\beta(\infty)\right]\\
&\leq t\prob\left[\frac{Y}{a(t)}>\frac{z}{\beta(\infty)}\right] \to\left(\frac{z}{\beta(\infty)}\right)^{-\frac{1}{\gamma}} \text{ as } t\rightarrow\infty.
\end{align*}

To prove the lower bound we use the basic convergence in Condition~\eqref{basic1} for the pair $(\beta(\infty)-X, Y)$. Before we show the lower bound, first observe that by arguments similar to the proof of Theorem 2.1 of \cite{Maulik:Resnick:Rootzen:2002} we have,
\begin{equation}\label{basic4b}
 t\prob\left[\left(\frac{(\beta(\infty)-X)Y}{\alpha(t)a(t)},\frac{Y}{a(t)}\right)\in\cdot\right]\vaguec \mu T_2^{-1}(\cdot) \text{ as } t\rightarrow\infty
\end{equation}
in $M_+([0,\infty]\times(0,\infty])$, where recall that $T_2(x,y)=(xy,y).$ Now to show the lower bound, first fix a large $M>0$ and $\epsilon>0$ such that $\alpha(t)<\epsilon$ for large $t$ (recall that $\alpha\in RV_{-\rho}$ and $\alpha(t)\rightarrow0$ in this case). Note that
\begin{align*}
 t\prob\left[\frac{XY}{a(t)}>z\right]&=t\prob\left[\frac{(X-\beta(\infty))Y}{\alpha(t)a(t)}\alpha(t)+\frac{\beta(\infty)Y}{a(t)}>z\right]\\
&\geq t\prob\left[\frac{(\beta(\infty)-X)Y}{\alpha(t)a(t)}\leq M, \frac{\beta(\infty)Y}{a(t)}>z+M\alpha(t)\right]\\
&\geq t\prob\left[\frac{(\beta(\infty)-X)Y}{\alpha(t)a(t)}\leq M, \frac{\beta(\infty)Y}{a(t)}>z+M\epsilon\right]\\
&\to \mu T_2^{-1}\left([0,M]\times \left(\frac{z+M\epsilon}{\beta(\infty)},\infty\right]\right),
\end{align*}
using~\eqref{basic4b}. First letting $\epsilon\rightarrow0$ and then letting $M\rightarrow\infty$, so that $$\mu T_2^{-1}\left(\partial \left([0, M]\times (z/\beta(\infty), \infty]\right)\right)=0,$$ we get that
$$\liminf_{t\rightarrow\infty}t\prob[\frac{XY}{a(t)}>z]\geq \mu T_2^{-1}([0,\infty)\times (z/\beta(\infty),\infty])=z^{-\frac1\gamma}\beta(\infty)^{\frac1\gamma}.$$
\end{proof}
\end{section}

\begin{section}{Some remarks on the tail condition~\eqref{moment condition}}\label{section:remarks on moments}
We have already noted that the compactification arguments in Section~\ref{mainresults} require some conditions on the tails of associated random variables. Except in Theorem~\ref{case1[a]}, they have been replaced by some moment conditions, using Markov inequality. The tail condition~\eqref{moment condition} in Theorem~\ref{case1[a]}  can also be replaced by the following moment condition:

If for some $\delta>0$, we have $\E[|X|^{{1}/{\rho}+\delta}]<\infty$, then~\eqref{moment condition} holds, as $\alpha\in RV_{\rho}$.

In general, if $(X,Y)$ follows CEVM model, it need not be true that $X\in D(G_\rho)$. However, if $X\in D(G_{\rho})$  with scaling and centering functions $\alpha(t)$ and $\beta(t)$ and $X\geq0$, then the moment condition \eqref{moment condition}. In fact,
$$\limsup_{t\rightarrow\infty}t\prob\left[\frac{X}{\alpha(t)}>\frac{z}{\epsilon}\right]$$ is a constant multiple of $({z}/{\epsilon})^{-{1}/{\rho}}$, which goes to zero as $\epsilon\rightarrow 0.$

The tail condition~\eqref{moment condition} continues to hold in certain other cases as well. Suppose that $X\geq0$ and $X\in D(G_{\lambda})$ with scaling and centering $A(t)$ and $B(t)$ and $\lambda<\rho$. In this case, $\alpha\in RV_\rho$ and $A\in RV_\lambda$. Thus, $\alpha(t)/A(t)\to\infty$. Hence, for any $\epsilon>0$, we have
$$t\prob\left[\frac{X}{\alpha(t)}>\frac{z}{\epsilon}\right]=t\prob\left[\frac{X}{A(t)}>\frac{z\alpha(t)}{\epsilon A(t)}\right]$$
is of order of $({\alpha(t)}/{A(t)})^{-{1}/{\lambda}}$, which goes to zero and~\eqref{moment condition} holds. However, it would be interesting to see the effect of $A$ and $B$ as scaling and centering in CEVM model. Since, $\alpha$ is of an order higher than $A$, the limit, as expected, becomes degenerate.

\begin{proposition}
Let the pair $(X,Y)\in CEVM(\alpha,\beta;a,b;\mu)$  with $\rho>0$ and $\gamma>0.$ Assume  $X\in D(G_{\lambda})$ with  $\rho>\lambda$ and centering and scaling as $A(t)$ and $B(t)$. If 
$$\lim_{t\rightarrow\infty}t\prob\left[\frac{X-B(t)}{A(t)}\leq x,\frac{Y-b(t)}{a(t)}>y\right]$$ exists for all continuity points $(x,y)\in \mathbb{R}\times \mathbb{E}^{(\gamma)}$, then, for any fixed $y\in \mathbb{E}^{(\gamma)}$, as $x$ varies in $\mathbb R$, the limit measure assigns same values to the sets of the form $[-\infty, x]\times (y,\infty]$.
\end{proposition}
\begin{proof} 
Observe that
$$\frac{A(t)}{\alpha(t)}x+\frac{B(t)-\beta(t)}{\alpha(t)}=\frac{A(t)}{\alpha(t)}\left(x+\frac{B(t)}{A(t)}\right)-\frac{\beta(t)}{\alpha(t)}\rightarrow 0\times(x+\frac{1}{\lambda})-\frac{1}{\rho}.$$
Therefore we have,
\begin{multline*}
\lim_{t\rightarrow\infty} t\prob\left[\frac{X-B(t)}{A(t)}\leq x,\frac{Y-b(t)}{a(t)}>y\right]\\
=\lim_{t\rightarrow\infty} t\prob\left[\frac{X-\beta(t)}{\alpha(t)}\leq \frac{A(t)}{\alpha(t)}x+\frac{B(t)-\beta(t)}{\alpha(t)},\frac{Y-b(t)}{a(t)}>y\right]
=\mu( [-\infty,-\frac{1}{\rho}]\times(y,\infty]),
\end{multline*}
which is independent of $x$.
\end{proof}

We would also like to consider the remaining case, namely, when $X\ge0$ and $X\in D(G_\lambda)$ with scaling $A$ and $\lambda>\rho$. Clearly, we have $\alpha(t)/A(t)\to 0$. Thus, for any $\epsilon>0$, we have 
$$t\prob\left[\frac{X}{\alpha(t)}>\frac{z}{\epsilon}\right]=t\prob\left[\frac{X}{A(t)}>\frac{z\alpha(t)}{\epsilon A(t)}\right]\to\infty$$ Hence~\eqref{moment condition} cannot hold and Theorem~\ref{case1[a]} is of no use. The next result show that in this case we have multivariate extreme value model with the limiting measure being concentrated on the axes, which gives asymptotic independence.
\begin{theorem}\label{theo:hidden}
Let $(X,Y)\in CEVM(\alpha,0;a,0;\mu)$  on the cone $[0,\infty]\times(0,\infty]$ and $X\in D(G_\lambda)$ with $\lambda>\rho$ and scaling $A(t)\in RV_{-\lambda}$.  Also assume that $X\in D(G_\lambda)$ with $\lambda>\rho$ and ${A(t)}/{\alpha(t)}\rightarrow \infty$. Then
\begin{equation}\label{hidden2}
t\prob\left[\left(\frac{X}{A(t)},\frac{Y}{a(t)} \right)\in\cdot\right]
\end{equation}
converges vaguely to a nondegenerate Radon measure on $[0,\infty]^2 \setminus\{(0,0)\}$, which is concentrated on the axes.
\end{theorem}
\begin{proof}
We first show the vague convergence on $[0,\infty]^2 \setminus\{(0,0)\}$. Let $x>0,y>0$. Then,
\begin{multline*}t\prob\left[ \left(\frac{X}{A(t)},\frac{Y}{a(t)}\right)\in ([0,x]\times [0,y])^c\right]\\ = t\prob\left[\frac{X}{A(t)}>x\right]+t\prob\left[\frac{Y}{a(t)}>y\right]-t\prob\left[\frac{X}{A(t)}>x,\frac{Y}{a(t)}>y\right].\end{multline*}
The first two terms converge due to the domain of attraction conditions on $X$ and $Y$. For the last term, by the CEVM conditions and the fact that $A(t)/\alpha(t)\to\infty$, we have
\begin{equation} \label{eq:hidden}
t\prob\left[\frac{X}{A(t)}>x,\frac{Y}{a(t)}>y\right] = t\prob\left[\frac{X}{\alpha(t)}>x\frac{A(t)}{\alpha(t)},\frac{Y}{a(t)}>y\right]\to 0.
\end{equation}
This establishes~\eqref{hidden2}. However, using $x=y$ in~\eqref{eq:hidden}, we find that the limit measure does not put any mass on $(0,\infty]^2$.
\end{proof}
\begin{remark}
Since, we have asymptotic independence, several different behaviors for the product are possible, as it has been illustrated in \cite{Maulik:Resnick:Rootzen:2002}.

While we have established asymptotic independence on the larger cone $[0,\infty]^2\setminus\{0\}$ in Theorem~\ref{theo:hidden}, CEVM gives another nondegenerate limit $\mu$ on the smaller cone $[0,\infty]\times(0,\infty]$. Thus, $(X,Y)$ exhibits hidden regular variation, as described in \cite{Resnick:2002, Maulik:Resnick:2004}.
\end{remark}
\end{section}

\begin{section}{Example}\label{section:example}
We now consider the moment condition in Theorem~\ref{main theorem negative}. We show that the condition is not necessary by providing an example, where the condition fails, but we explicitly calculate the tail behavior of $XY$. 

Let $X$ and $Z$ be two independent random variables, where $X$ follows Beta distribution with parameters $1$ and $a$ and $Z$ is supported on $[0,1]$ and is in $D(G_{-1/b})$, for some $a>0$ and $b>0$. Thus, we have $\prob[X>x]=(1-x)^a$ and $\prob[Z>1-\frac{1}{x}]=x^{-b}L(x)$ for some slowly varying function $L$. Let $G$ denote the distribution function of the random variable $Y=X\wedge Z$. Then $\overline{G}(x)=(1-x)^{a+b}L(\frac{1}{1-x})$ and hence $Y\in D(G_{-1/(a+b)}).$ Clearly, for $\widetilde X= 1/(1-X)$, we have $\E[\widetilde X^{a+b}]=\infty$ and the moment condition in Theorem~\ref{main theorem negative} fails for $\rho=-1/(a+b)$.

We further define  
$$\widetilde{a}(t)= \frac{1}{1-\left({1}/{\overline{G}}\right)^{\leftarrow}(t)},$$ where the left-continuous inverse is defined in Definition~\ref{chap1:def:inverse}. Since, $Y\in D(G_{-1/(a+b)})$, we have from Corollary 1.2.4 of \cite{Haan:Ferreira:2006}, \begin{equation}\label{example relation}
 \frac{{\widetilde{a}(t)}^{a+b}}{L(\widetilde{a}(t))}\sim t.
\end{equation}
Then, for $x>y>0$, we have,
\begin{align*}
t \prob \left[ \frac{\widetilde X}{\widetilde a(t)} \le x, \frac{\widetilde Y}{\widetilde a(t)} > y\right]
= &t\prob \left[\widetilde{a}(t)(X-1)\leq -\frac1x, \widetilde{a}(t)(Y-1)> -\frac1y\right]\\
= &t\prob \left[1-\frac{1}{y\widetilde{a}(t)}<X\leq 1-\frac{1}{x\widetilde{a}(t)}, Z>1-\frac{1}{y\widetilde{a}(t)}\right]\\
= &\frac{t}{\widetilde{a}(t)^{a+b}}[y^{-a}-x^{-a}]y^{-b}L(y{\widetilde{a}(t)})\\
\sim &\frac{t L(\widetilde{a}(t))}{\widetilde{a}(t)^{a+b}}[y^{-a}-x^{-a}]y^{-b}
\sim [y^{-a}-x^{-a}]y^{-b},
\end{align*}
where we use~\eqref{example relation} in the last step. Thus, $(\widetilde X, \widetilde Y)$ satisfies Conditions~\eqref{basic1} and~\eqref{nondegenerate} with $\rho=\gamma=-1/(a+b)$.

Finally, we directly calculate the asymptotic tail behavior of the product. For simplicity of the notations, for $y>0$, let us denote $c(t)=1-{1}/(\widetilde{a}(t)y)$. Since $\widetilde{a}\in RV_{1/(a+b)}$, as $t\to\infty$, we have $c(t)\to 1$. Also, $\widetilde a(t) = 1/((1-c(t))y)$. Then,
\begin{align*}
&t\prob\left[XY>1-\frac{1}{\widetilde{a}(t)y}\right]
= a\cdot t \int\limits_{\sqrt{c(t)}}^{1} \prob \left[Z>\frac{c(t)}{s} \right](1-s)^{a-1}ds\\
\sim &a\cdot {\widetilde{a}}^{\leftarrow} \left(\frac{1}{(1-c(t))y}\right) \int\limits_{\sqrt{c(t)}}^{1}\prob \left[Z>\frac{c}{s}\right] (1-s)^{a-1}ds\\
\sim &a\cdot {\widetilde{a}}^{\leftarrow} \left(\frac{1}{(1-c(t))y}\right) \int\limits_{\sqrt{c(t)}}^{1} \left(1-\frac{c}{s}\right)^{b} (1-s)^{a-1} L\left(\frac{1}{1-c(t)/s}\right)ds
\intertext{substituting $s=1-z(1-c(t))$,}
= &a (1-c(t))^{a+b} {\widetilde{a}}^{\leftarrow} \left(\frac{1}{(1-c(t))y}\right) \int\limits_{0}^{\frac{1}{1+\sqrt{c(t)}}} \frac{(1-z)^{b}z^{a-1}}{(1-(1-c(t))z)^{b}}  L\left(1+\frac{c(t)}{(1-c(t))(1-z)}\right)dz\\
\sim &a (1-c(t))^{a+b} {\widetilde{a}}^{\leftarrow} \left(\frac{1}{(1-c(t))y}\right) L\left( \frac{c(t)}{1-c(t)}\right) \int\limits_{0}^{1/2}(1-z)^b z^{a-1}dz,
\end{align*}
where, in the last step, we use Dominated Convergence Theorem and the facts that $c(t)\to 1$ and 
$$L\left(1+\frac{c(t)}{(1-c(t))(1-z)}\right) \sim L\left(\frac{c(t)}{(1-c(t))(1-z)}\right) \sim L\left( \frac{c(t)}{1-c(t)}\right)$$
uniformly on bounded intervals of $z$.
Finally, using~\eqref{example relation}, definition of $c(t)$,  to get,
$$t\prob\left[\frac{(1-XY)^{-1}}{\widetilde{a}(t)}>y\right]\rightarrow a \cdot y^{-(a+b)}\int\limits_{0}^{1/2}(1-z)^b z^{a-1}dz.$$
\end{section}

\cleardoublepage

\chapter{Sums of free random variables with regularly varying tails}\label{free sub chapter}
\section{Introduction}\label{sec: basic intro}
In this chapter we show an application of regular variations to free probability theory. In Lemma~\ref{lem:one large jump} we saw that  random variables (or distribution functions) with regularly varying tails are in the class of subexponential ones and satisfy the principle of one large jump. In terms of probability measures on $[0,\infty)$ we can reformulate it to say, measures with regularly varying index $-\alpha$, $\alpha\geq 0$ satisfy the following property for all $n\geq 1$,
$$\mu^{*n}(x,\infty)\sim n\mu(x,\infty) \text{ as $x\rightarrow\infty$.}$$

 In Chapter~\ref{chap:introduction} we saw many properties of the subexponential distributions. We shall now show that this property of measures with regularly varying tails is not restricted to classical convolution on the space of probability measures, but they extend to free probability theory as introduced in Section~\ref{chap1:subsec3}. The concept of freeness allows one to describe a convolution on the space of all probability measures which is known as the free convolution. Free convolution between two measures $\mu$ and $\nu$ is denoted by $\mu\boxplus\nu$. In Theorem~\ref{main theorem-1} we show that a probability measure $\mu$ with regularly varying tail on $[0,\infty)$ satisfies the following property for all $n\geq 1$,
$$\mu^{\boxplus n}(x,\infty)\sim n\mu(x,\infty) \text{ as $x\rightarrow\infty$}.$$

  The space of  all measures which satisfy the above property will be called \textit{free subexponential} measures (see Definition~\ref{def: free sub}). Section~\ref{sec: intro} elaborates upon Section~\ref{sec: basic intro} and provides an extended review of the useful notions and results of free probability theory. We refer to the monographs by \cite{voiculescu:dykema:nica:book, speicher:nica:book} for the basic definitions  of free probability. Towards the end of Section~\ref{sec: intro}, once we have introduced the important concepts, we introduce the main problem considered in the rest of the chapter. In Section~\ref{subsec: main results} we also provide a sketch of the chapter and an approach to the proof of Theorem~\ref{main theorem-1}.

\section{Basic setup and results} \label{sec: intro}
\subsection{Free probability}\label{subsec: free}
The notion of freeness was first introduces by \cite{voiculescu1986addition}. The definition was motivated by his study on certain types of von Neumann algebras. The relationship of freeness with random matrices established in \cite{voiculescu:1991} was a major advancement in this area. Some of the notions in this Subsection has already been in introduced in Section~\ref{chap1:subsec3}. We define them here formally and elaborate upon them.

\begin{definition}\label{def: ncp}A non-commutative probability space is a pair $(\mathcal{A},\tau)$ where $\mathcal{A}$ is a unital complex algebra and $\tau$ is a linear functional on $\mathcal{A}$ satisfying $\tau(1)=1.$ 
 \end{definition}
 Given elements $a_1,a_2,\cdots,a_k\in\mathcal A$ and indices $i(1),i(2),\cdots,i(n)\leq k$, the numbers $\tau(a_{i(1)}a_{i(2)}\cdots a_{i(n)})$ denote a moments. In analogy to the fact from the classical probability theory, where, in certain cases, the moments determine a distribution function, the collection of all moments involving $a_1,a_2,\cdots ,a_k$ is called the joint distribution of $a_1,a_2,\cdots a_k$.

Now we give two examples of non-commutative probability spaces. The first example indicates the motivation behind the nomenclature of a probability space.

\begin{example}
 \begin{enumerate}
  \item Let $(\Omega,\mathcal F, \prob)$ be a probability space in the classical sense. Let $\mathcal A=L^\infty(\Omega,\prob)$ be the set of bounded random variables and $$\tau(a)=\int a(\omega)d\prob(w).$$ In this case, the random variables actually commute.

\item Let $\mathcal H$ be a Hilbert space and $\mathcal A=\mathcal B(\mathcal H)$ be the space of all bounded operators on $\mathcal H$. The fundamental states are of the form $\tau(a)=\langle \eta,a\eta\rangle$ where $\eta\in \mathcal H$ is a vector of norm $1$.
 \end{enumerate}

\end{example}

Having defined the non-commutative probability spaces, we now give the definition of freeness.

\begin{definition}
Let $(\mathcal A,\tau)$ be a non-commutative probability space.  A family of unital subalgebras $\{\mathcal{A}_i\}_{i\in I}\subset\mathcal{A}$ is called {\it free} if $\tau(a_1\cdots a_n)=0$ whenever $\tau(a_j)=0,\  a_j\in \mathcal{A}_{i_j}$ and $i_j\neq i_{j+1}$ for all $j$. 
\end{definition}
Next we briefly describe some of the properties of freeness in this setup. 

Freeness allows easy calculation of the mixed moments from the knowledge of moments of the individual elements. For example, if $\mathcal A$ and $\mathcal B$ are free, then one has, for $a,a_1,a_2\in \mathcal A$ and $b,b_1,b_2\in \mathcal B$,
\begin{enumerate}
 \item $\tau(ab)=\tau(a)\tau(b);$ \label{tauab}
\item $\tau(a_1ba_2)=\tau(a_1a_2)\tau(b)$;
\item $\tau(a_1b_1a_2b_2)=\tau(a_1a_2)\tau(b_1)\tau(b_2)+\tau(a_1)\tau(a_2)\tau(b_1b_2)-\tau(a_1)\tau(b_1)\tau(a_2)\tau(b_2)$.
\end{enumerate}

In all the above examples one has to first center the elements and use the freeness condition. If $\hat{a}=a-\tau(a)$ and $\hat{b}=b-\tau(b)$, then $\tau(\hat{a}\hat{b})=0$ by freeness. Also, by linearity of $\tau$, we have $\tau(\hat{a}\hat{b})=\tau(ab)-\tau(a)\tau(b)$ which leads to~\ref{tauab}. Other claims can similarly be  checked.

Much of the calculations of moments using the above methods involve combinatorial arguments. In the combinatorial theory of free probability, it turns out that  calculating moments is often difficult. In some  such cases, the calculation of ``free cumulants'' are more advantageous. In order to describe the free cumulants one uses a subclass of the partitions known as non-crossing partitions. 
\begin{definition}
Let $n$ be a natural number. We call $\pi=\{V_1,\cdots ,V_r\}$ to be a \textit{partition} of $S=\{1,2,\cdots ,n\}$ if and only if $V_i$ for $1\leq i\leq r$ are pairwise disjoint and their union is $S$. $V_1,\cdots,V_r$ are called the \textit{blocks} of the partition $\pi$. If two elements $p$ and $q$ are in the same block then one writes $p\sim_\pi q$. 
\end{definition}
\begin{definition}A partition $\pi$ is called \textit{crossing} if there exists $1\leq p_1<q_1<p_2<q_2\leq n$ such that $p_1\sim_\pi p_2\nsim_\pi q_1\sim_\pi q_2$. A partition which is not crossing, is called a \textit{ non-crossing partition}. The set of non-crossing partitions of $S$ is denoted by $NC(S)$.
\end{definition}
Consider a collection $\{l_n:\mathcal A^n\to \mathbb C\}_{n=1}^\infty$ of multilinear functionals on a fixed complex algebra $\mathcal A$. 
Given a non-empty set $J=\{i_1<i_2<\cdots <i_m\}$ of positive integers, define
 $l(\{a_i\}_{i\in J})=l_m(a_{i_1},\cdots a_{i_m}).$
Further, if $\pi\in NC(J)$, define
$$l_{\pi}(\{a_i\}_{i\in J})=\prod_{V\in\pi }l(\{a_i\}_{i\in V}).$$
% We define $l_\pi(\{a_i\}_{i\in J})\in\mathbb C$ for finite non-empty set $J$ of positive integers, families $\{a_i\}_{i\in J}$ of elements of $\mathcal A$ and $\pi\in NC(J)$ in two steps:
% Let $J=\{i_1<i_2<\cdots <i_m\}$ and define $l(\{a_i\}_{i\in J})=l_m(a_{i_1},\cdots a_{i_m})$; then define
%  $$l_{\pi}(\{a_i\}_{i\in J})=\prod_{V\in\pi }l(\{a_i\}_{i\in V}).$$
We are now ready to introduce the free cumulants.
\begin{definition}
 Let $(\mathcal A,\tau)$ be a non-commutative probability space. The free cumulants are defined as a collection of multilinear functions $\kappa_n:\mathcal A^n\to\mathbb C$ by the following system of equations:
\begin{equation}\label{eq:moment-cumulant}
\tau(a_1a_2\cdots a_n)=\sum_{\pi\in NC(n)}\kappa_\pi(a_1,a_2,\cdots ,a_n).
\end{equation}
\end{definition}

It can be easily shown that free cumulants are well defined.
% Now let us consider an example of $\kappa_\pi$. Consider the non-crossing partition $\pi=\{(1,10),(2,5,9),(3,4),(6),(7,8)\}$ for $n=10$. Then 
%$$\kappa_\pi(a_1,a_2,\cdots, a_{10})=\kappa_2(a_1,a_{10})\kappa_3(a_2,a_5,a_9)\kappa_2(a_3,a_4)\kappa_1(a_6)\kappa_2(a_7,a_8).$$
The first few free cumulants can be calculated as follows. For $n=1$, we have from~\eqref{eq:moment-cumulant}, $\kappa_1(a_1)=\tau(a_1)$ which is the mean. For $n=2$, one gets the covariance, $\kappa_2(a_1,a_2)=\tau(a_1a_2)-\tau(a_1)\tau(a_2)$.

Freeness can also be characterized in terms of cumulants. In fact, freeness is equivalent to vanishing of mixed cumulants. More precisely, we have the following result.
\begin{theorem}[\citealp{speicher:nica:book}, Theorem11.16]\label{free:cumulants:charac}
The elements $\{a_i\}_{i\in I}$ of a complex unital algebra $\mathcal A$ are free if and only if $\kappa_n(a_{i(1)},\cdots,a_{i(n)})=0$ whenever $n\geq 2$ and there exists $k,l$ such that $i(k)\neq i(l)$. 
\end{theorem}
Denote by $\kappa_n(a):=\kappa_n(a,\cdots,a)$ the $n$-th cumulant of $a$. When $a$ and $b$ are free then from Theorem~\ref{free:cumulants:charac} one can conclude that, for all $n\geq 1$,
\begin{equation}\label{kappafree}
 \kappa_n(a+b)=\kappa_n(a)+\kappa_n(b).
\end{equation}

\begin{definition}\label{def: r-transform}
 For an element $a$, the formal power series $$R_a(z)=\sum_{n\geq0}\kappa_{n+1}(a)z^n$$ is called the R transform of $a$.
\end{definition}
 By~\eqref{kappafree}, the R transform in free probability plays a role similar to the logarithm of the characteristic function in classical probability. If $a$ and $b$ are free, we have,
$$R_{a+b}(z)=R_a(z)+R_b(z).$$

The relationship between the moments and cumulants in~\eqref{eq:moment-cumulant} can be translated to the following relationship between the formal power series. For a proof of the following Proposition we refer to Chapter 11 of \cite{speicher:nica:book}.

\begin{proposition}\label{prop:moment-cumulant}
 Let $(m_n)_{n\geq 1}$ and $(\kappa_n)_{n\geq 1}$ be two sequences of complex numbers and consider the corresponding power series:
$$M(z)=1+\sum_{n=1}^\infty m_n z^n \quad\text{and} \quad C(z)=1+\sum_{n=1}^\infty \kappa_nz^n.$$
Then the following statements are equivalent:
\begin{enumerate}
 \item For all $n\geq 1$,
$$m_n=\sum_{\pi\in NC(n)}\kappa_\pi=\sum_{\substack{\pi=\{V_1,\cdots,V_r\} \\ \pi\in NC(n)}}\kappa_{|V_1|}\cdots \kappa_{|V_r|}.$$
\item For all $n\geq 1$, (with $m_0:=1$) one has,
$$m_n=\sum_{s=1}^n\,\,\sum_{\stackrel{i_1,\cdots,i_s\in\{0,1,\cdots,n-s\}}{i_1+\cdots+i_s=n-s}}\kappa_s m_{i_1}\cdots m_{i_s}.$$
\item $C(zM(z))=M(z)$.
\end{enumerate}

\end{proposition}

\begin{definition} For an element $a$ of $(\mathcal A,\tau)$, the Cauchy transform is the formal power series given by
\begin{equation}
 G_a(z):=\sum_{n=0}^{\infty}\frac{m_n(a)}{z^{n+1}}, \quad\text{where } m_n(a)=\tau(a^n).
\end{equation}
\end{definition}
Now from Proposition~\ref{prop:moment-cumulant}, the Cauchy transform and the R transform of a random variable $a$ is related by the following relationship:
$$G_a\left(R_a(z)+\frac1z\right)=z.$$

Although the above things are done in terms of formal power series,  they can be made rigorous when the algebra is a $C^*$-algebra or von Neuman algebra and the element is self-adjoint.
\begin{definition}\label{def:state:tracial}Suppose $\mathcal A$ is a von Neumann algebra. A linear functional $\tau:\mathcal A\to \mathbb C$ with $\tau(1)=1$ is called:
\begin{enumerate}
 \item state: if $\tau(a)\geq 0$ for $a\geq 0$;
\item tracial : if $\tau(ab)=\tau(ba)$;
\item faithful: if $\tau(aa^*)=0$ implies $a=0$;
\item normal : if for any net $\{a_\beta\}$ of nonnegative elements of $\mathcal A$ monotonically decreasing to zero, we have $\inf_\beta\tau(a_\beta)=0$.
\end{enumerate}
\end{definition}
A self-adjoint element of a $C^*$-algebra is often called a random variable or a random element. Consider a self-adjoint element or a random variable $X$ of a $C^*$-algebra $\mathcal A$ equipped with a state $\tau$. Then there exists a compactly supported measure $\mu_X$ on $\mathbb R$, such that, for all $n\in\mathbb N$
$$\tau(X^n)=\int_{\mathbb R} t^n\mu_X(dt).$$

We further have that $G_X$ is an analytic function and $$G_X(z)=\tau\left( (z-X)^{-1}\right)=\int_{\mathbb R}\frac1{z-t}\mu_X(dt), \text{ for $z\in \mathbb C^+$}.$$
In general, the above definition works for any measure $\mu$ on $\mathbb R$. We denote the Cauchy transform of a measure $\mu$ by $\gmu$. Some further properties of the Cauchy transform for general measures is dealt in Section~\ref{sec: transform}.

\subsection{Main results}\label{subsec: main results}
The basic notions in free probability theory described in Subsection~\ref{subsec: free} are well suited for bounded random variables which correspond to compactly supported measures. However, our interest is in probability measures with regularly varying tails which have  unbounded support and hence one has to consider unbounded random variables. In order to deal with unbounded random variables we need to consider a tracial $W^*$-probability space $(A,\tau)$ with a von Neumann algebra $\mathcal A$ and a normal faithful tracial state $\tau$. See Definition~\ref{def:state:tracial} for the related concepts. By unbounded random variables, we shall mean self-adjoint operators affiliated to a von Neumann algebra $\mathcal A$. For definition of affiliated random variables and freeness for unbounded random variables, please see Subsection~\ref{chap1:subsec3} of Chapter~\ref{chap:introduction}.

 Recall that, given two measures $\mu$ and $\nu$, there exists a unique measure $\mu \boxplus \nu$, called the \textit{free convolution} of $\mu$ and $\nu$, such that whenever $X$ and $Y$ are two free random elements on a tracial $W^*$ probability space $(\mathcal A, \tau)$, the random element $X+Y$ has the law $\mu\boxplus\nu$. The definition of subexponential distributions  can be easily extended to the non-commutative setup by replacing the classical convolution powers by free convolution ones and this leads to Definition~\ref{def: free sub} of free subexponential probability measures. 
% We shall define a free subexponential measure on $[0, \infty)$ alone, but the definition can be extended to probability measures on the entire real line, as in the classical case. Formally, we define a free subexponential measure as follows:
%\begin{definition}
%\label{def: free sub}
A probability measure $\mu$ on $[0,\infty)$ is said to be free subexponential if $\mu(x,\infty)>0$ for all $x\geq 0$ and  for all $n$,
$$\mu^{\boxplus n}(x,\infty)=\underbrace{(\mu\boxplus\cdots\boxplus\mu)}_{n \text{ times}}(x,\infty)\sim n\mu(x,\infty) \text{ as $x\rightarrow\infty$.}$$
%\end{definition}
The above definition can be rewritten in terms of distribution functions as well. A distribution function $F$ is called free subexponetial if for all $n \in \mathbb N$, $\overline {F^{\boxplus n}}(x) \sim n \overline F(x)$ as $x\to\infty$. A random element $X$ affiliated to a tracial $W^*$-probability space is called free subexponential, if its distribution is so. One immediate consequence of the definition of free subexponentiality is the principle of one large jump.

\cite{arous2006free} showed that for two distribution functions $F$ and $G$, there exists a unique measure $F\boxv G$, such that whenever $X$ and $Y$ are two free random elements on a tracial $W^*$-probability space, $F\boxv G$ will become the distribution of $X\vee Y$. Here $X\vee Y$ is the maximum of two self-adjoint operators defined using the spectral calculus via the projection-valued operators, see \cite{arous2006free} for details. \cite{arous2006free} showed that $F\boxv G(x) = \max((F+G-1)(x),0)$, and hence $F^{\boxv n}(x) = \max((nF-(n-1))(x),0)$. Then, we have, for each $n$, $\overline{F^{\boxv n}} (x) \sim n \overline F(x)$ as $x\to\infty$. Thus, by definition of free subexponentiality, we have
\begin{proposition}[Free one large jump principle] \label{prop: free one large jump}
Free subexponential distributions satisfy the principle of one large jump, namely, if $F$ is free subexponential, then, for every $n$,
$$\overline{F^{\boxplus n}}(x) \sim \overline{F^{\boxv n}}(x) \text{ as $x\to\infty$.}$$
\end{proposition}

%In spite of the above important property possessed by the free subexponential distributions, it remains to be checked whether such a class is nonempty.
 The distributions with regularly varying tails of index $-\alpha$, with $\alpha\ge0$, form an important class of examples of subexponential distributions in the classical setup.  We show in Theorem~\ref{main theorem-1} that the distributions with regularly varying tails of index $-\alpha$, $\alpha\geq0$ also form a subclass of the free subexponential distributions. 
\begin{theorem}\label{main theorem-1}
If $F$ has regularly varying tail of index ${-\alpha}$ with $\alpha\ge0$, then $F$ is free subexponential.
\end{theorem}

While it need not be assumed that the measure is concentrated on $[0,\infty)$, both the notions of free subexponentiality and regular variation are defined in terms of the measure restricted to $[0,\infty)$. Thus we shall assume the measure to be supported on $[0,\infty)$ except for the definitions of the relevant transforms in the initial part of Subsection~\ref{subsec: transf} and in the statement and the proof of Theorem~\ref{main theorem-1}. 

Due to the lack of coordinate systems and expressions for joint distributions of non-commutative random elements in terms of probability measures, the proofs of the above results deviate from the classical ones. In absence of the higher moments of the distributions with regularly varying tails, we cannot use the usual moment-based approach used in free probability theory. Instead, Cauchy and Voiculescu transforms become the natural tools to deal with the free convolution of measures. While Cauchy transform has already been introduced, we define Voiculescu transform in Section~\ref{sec: transform}. We then discuss the relationship between the remainder terms of Laurent series expansions of Cauchy and Voiculescu transforms of measures with regularly varying tail of index $-\alpha$. We need to consider four cases separately depending on the maximum number $p$ of integer moments that the measure $\mu$ may have. For a nonnegative integer $p$, let us denote the class of all probability measures $\mu$ on $[0,\infty)$ with $\int_0^\infty t^p d\mu(t) <\infty$, but $\int_0^\infty t^{p+1} d\mu(t) =\infty$, by $\mathcal M_p$. We shall also denote the class of all probability measures $\mu$ in $\mathcal M_p$ with regularly varying tail of index $-\alpha$ by $\mathcal M_{p,\alpha}$. Note that, we necessarily have $\alpha\in[p,p+1]$. Theorems~\ref{thm: error equiv}--\ref{thm: error equiv new} summarize the relationships among the remainder terms for various choices of $\alpha$ and $p$. These theorems are the key tools of this Chapter. Section~\ref{sec: transform} is concluded with two Abel-Tauber type results for Stieltjes transform of measures with regularly varying tail. We then prove Theorem~\ref{main theorem-1} in Section~\ref{sec: pf thm one} using Theorems~\ref{thm: error equiv}--\ref{thm: error equiv new}. We use the final two sections to prove Theorems~\ref{thm: error equiv}--\ref{thm: error equiv new}. In Section~\ref{sec: Cauchy}, we collect some results about the remainder term in Laurent series expansion of Cauchy transform of measures with regularly varying tails. In Section~\ref{sec: C-V reln}, we study the relationship between the remainder terms in Laurent expansions of Cauchy and Voiculescu transforms through a general analysis of the remainder terms of Taylor expansions of a suitable class of functions and their inverses or reciprocals. Combining, the results of Sections~\ref{sec: Cauchy} and~\ref{sec: C-V reln}, we prove Theorems~\ref{thm: error equiv}--\ref{thm: error equiv new}.

\section{Some transforms and their related properties} \label{sec: transform}
In this section, we collect some notations, definitions and results to be used later in this Chapter. In Subsection~\ref{subsec: non-tang}, we define the concept of non-tangential limits. Various transforms in non-commutative probability theory, like Cauchy, Voiculescu and $R$ transforms are introduced in Subsection~\ref{subsec: transf}. Theorems~\ref{thm: error equiv}--\ref{thm: error equiv new} regarding the relationship between the remainder terms of Laurent expansions of Cauchy and Voiculescu transforms are given in this subsection as well. Finally, in Subsection~\ref{subsec: karamata}, two results about measures with regularly varying tails are given.
\subsection{Non-tangential limits and notations} \label{subsec: non-tang}
Recall that, the complex plane is denoted by $\mathbb C$ and for a complex number $z$, $\Re z$ and $\Im z$  denotes the real and imaginary parts respectively. We say $z$ goes to infinity (zero respectively) \textit{non-tangentially} to $\mathbb R$ (n.t.), if $z$ goes to infinity (zero respectively), while $\Re z/ \Im z$ stays bounded. We can then define that a function $f$ converges or stays bounded as $z$ goes to infinity (or zero) n.t. To elaborate upon the notion, given positive numbers $\eta$, $\delta$ and $M$, let us define the following cones:
\begin{enumerate}
\item $\Gamma_{\eta}=\{z\in \mathbb{C}^+: |\Re z|<\eta \Im z\}$ and $\Gamma_{\eta,M}=\{z\in\Gamma_{\eta}: |z|>M\}$,
\item $\Delta_{\eta}=\{z\in \mathbb{C}^-: |\Re z|<-\eta \Im z\}$ and $\Delta_{\eta,\delta}=\{z\in\Delta_{\eta}: |z|<\delta\}$,
\end{enumerate}
where $\mathbb C^+$ and $\mathbb C^-$ are the upper and the lower halves of the complex plane respectively, namely, $\mathbb C^+ = \{z\in\mathbb C: \Im z>0\}$ and $\mathbb C^- = - \mathbb C^+$. Then we shall say that $f(z) \to l$ as $z$ goes to $\infty$ n.t., if for any $\epsilon>0$ and $\eta>0$, there exists $M\equiv M(\eta,\epsilon)>0$, such that $|f(z) - l|< \epsilon$, whenever $z \in \Gamma_{\eta, M}$. The boundedness can be defined analogously.

We shall write $f(z)\approx g(z)$, $f(z)=\lito(g(z))$ and $f(z)=\bigo(g(z))$ as $z\to\infty$ n.t.\ to mean that $f(z)/g(z)$ converges to a non-zero limit, ${f(z)}/{g(z)}\rightarrow 0$ and $f(z)/g(z)$ stays bounded as $z\to\infty$ n.t.\ respectively. If the non-zero limit is $1$ in the first case, we write $f(z) \sim g(z)$ as $z\to\infty$ n.t. For $f(z)=\lito(g(z))$ as $z\to\infty$ n.t., we shall also use the notations $f(z)\ll g(z)$ and $g(z)\gg f(z)$ as $z\to\infty$ n.t.

The map $z\mapsto 1/z$ maps the set $\Gamma_{\eta,{1}/{\delta}}$ onto $\Delta_{\eta,\delta}$ for each positive $\eta$ and $\delta$. Thus the analogous concepts can be defined for $z\to 0$ n.t.\ using $\Delta_{\eta,\delta}$.
\subsection{Cauchy and Voiculescu Transform} \label{subsec: transf}
Recall that for a probability measure $\mu$, its Cauchy transform is defined as
$$G_{\mu}(z)=\int_{-\infty}^{\infty}\frac{1}{z-t}d\mu(t),\ \ \  z\in \mathbb{C}^+.$$
Note that $G_{\mu}$ maps $\mathbb{C}^+$ to $\mathbb{C}^{-}$. Set $F_{\mu}=1/\gmu$, which maps $\mathbb{C}^+$ to $\mathbb{C}^+$. We shall be also interested in the function $H_\mu(z) = G_\mu(1/z)$ which maps $\mathbb C^-$ to $\mathbb C^-$.

By Proposition~5.4 and Corollary~5.5 of \cite{bercovici1993free}, for all $\eta>0$ and for all $\epsilon\in (0,\eta\wedge 1)$, there exists $\delta\equiv\delta(\eta)$ small enough, such that $H_\mu$ is a conformal bijection from $\Delta_{\eta,\delta}$ onto an open set $\mathcal{D}_{\eta,\delta}$, where the range sets satisfy
$$\Delta_{\eta-\epsilon,(1-\epsilon)\delta}\subset \mathcal{D}_{\eta,\delta}\subset \Delta_{\eta+\epsilon,(1+\epsilon)\delta}.$$

If we define $\mathcal D = \cup_{\eta>0} \mathcal D_{\eta, \delta(\eta)}$, then we can obtain an analytic function $L_\mu$ with domain $\mathcal D$ by patching up the inverses of $H_\mu$ on $\mathcal{D}_{\eta,\delta(\eta)}$ for each $\eta>0$. In this case $L_{\mu}$ becomes the right inverse of $H_\mu$ on $\mathcal{D}.$ Also it was shown that the sets of type $\Delta_{\eta,\delta}$ were contained in the unique connected component of the set $H_\mu^{-1}(\mathcal{D})$. It follows that $H_\mu$ is the right inverse of $L_{\mu}$ on $\Delta_{\eta,\delta}$ and hence on the whole connected component by analytic continuation.

We then define $R$ and Voiculescu transforms of the probability measure $\mu$ respectively as:
\begin{equation}\label{eq: R phi defn}
R_{\mu}(z)=\frac{1}{L_{\mu}(z)}-\frac{1}{z} \ \ \ \text{and} \ \ \ \phi_{\mu}(z)=R_{\mu}({1}/{z}).
\end{equation}
Also, if $X$ is a (self-adjoint) random element in a $W^*$-algebra $\mathcal A$ with an associated measure $\mu$, then $R_\mu$ coincides with the formal definition given in Definition~\ref{def: r-transform}. 
Arguing as in the case of $\gmu(1/z)$, it can be shown that $F_\mu$ has a left inverse, denoted by $F_\mu^{-1}$ on a suitable domain and, in that case, we have
$$\phi_{\mu}(z)=F_{\mu}^{-1}(z)-z.$$

\cite{bercovici1993free} established the following relation between free convolution and Voiculescu and $R$ transforms. For probability measures $\mu$ and $\nu$,
$$\phi_{\mu\boxplus\nu}=\phi_{\mu}+\phi_{\nu} \ \ \text{and}\ \ R_{\mu\boxplus\nu}=R_{\mu}+R_{\nu},$$
wherever all the functions involved are defined.

We shall also need to analyze the power and Taylor series expansions of the above transforms. For Taylor series expansion of a function, we need to define the remainder term appropriately, so that it becomes amenable to the later calculations. In fact, for a function $A$ with Taylor series expansion of order $p$, we define the remainder term as
\begin{equation} \label{eq: def remainder}
r_A(z) = z^{-p} \left( A(z) - \sum_{i=0}^p a_i z^i \right).
\end{equation}
Note that, we divide by $z^p$ after subtracting the polynomial part.

For compactly supported measure $\mu$, \cite{speicher1994multiplicative} showed that, in an appropriate neighborhood of zero, $R_{\mu}(z)=\sum_{j= 0}^\infty\kappa_{j+1}(\mu)z^j$, where $\{\kappa_j(\mu)\}$ denotes the free cumulant sequence of the probability measure $\mu$. For probability measures $\mu$ with finite $p$ moments, Taylor expansions of $R_\mu$ and $H_\mu$ are given by Theorems~$1.3$ and~$1.5$ of \cite{benaych2006taylor}:
\begin{equation} \label{eq: error}
R_{\mu}(z)=\sum_{j=0}^{p-1}\kappa_{j+1}(\mu)z^j+ z^{p-1} r_{R_\mu}(z), \text{ and } H_\mu(z)=\sum_{j=1}^{p+1}m_{j-1}(\mu)z^j+z^{p+1} r_{H_\mu}(z),
\end{equation}
where the remainder terms $r_{R_\mu}(z) \equiv r_R(z) = \lito(1)$ and $r_{H_\mu}(z) \equiv r_H(z) = \lito(1)$ as $z\rightarrow 0$ n.t.\ are defined along the lines of~\eqref{eq: def remainder}, $\{\kappa_j(\mu):j\le p\}$ denotes the free cumulant sequence of $\mu$ as before and $\{m_j(\mu):j\le p\}$ denotes the moment sequence of the probability measure $\mu$. When there is no possibility of confusion, we shall sometimes suppress the measure involved in the notation for the moment and the cumulant sequences, as well as the remainder terms.
In the study of stable laws and the infinitely divisible laws, the following relationship between Cauchy and Voiculescu transforms of a probability measure $\mu$, obtained in Proposition~2.5 of \cite{bercovici1999stable}, played a crucial role:
\begin{equation} \label{eq: phi mu}
\phi_{\mu}(z)\sim z^2\left[\gmu(z)-\frac{1}{z}\right] \text{ as  } z\rightarrow\infty  \text{ n.t.}
\end{equation}
Depending on the number of moments that the probability measure $\mu$ may have, its Cauchy and Voiculescu transforms can have Laurent series expansions of higher order. Motivated by this fact, for probability measures $\mu\in\mathcal M_p$ (that is, when $\mu$ has only $p$ integral moments), we introduce the remainder terms in Laurent series expansion of Cauchy and Voiculescu transforms (in analogy to the remainder terms in Taylor series expansion):
\begin{align}
r_{G_\mu}(z) \equiv r_G(z) &= z^{p+1} \left(G_\mu(z) - \sum_{j=1}^{p+1}m_{j-1}(\mu)z^{-j}\right)\label{eq:rg defn}
\intertext{and}
r_{\phi_\mu}(z) \equiv r_\phi(z) &= z^{p-1} \left(\phi_\mu(z) - \sum_{j=0}^{p-1}\kappa_{j+1}(\mu)z^{-j}\right),\label{eq: rphi defn}
\end{align}
where we shall again suppress the measure $\mu$ in the notation if there is no possibility of confusion. In~\eqref{eq: rphi defn}, we interpret the sum on the right side as zero, when $p=0$. Using the remainder terms defined in~\eqref{eq:rg defn} and~\eqref{eq: rphi defn} we provide  extensions of~\eqref{eq: phi mu} in Theorems~\ref{thm: error equiv}--\ref{thm: error equiv new} for different choices of $\alpha$ and $p$. We split the statements into four cases as follows: (i) $p$ is a positive integer and $\alpha\in(p,p+1)$, (ii) $p$ is a positive integer and $\alpha=p$, (iii) $p=0$ and $\alpha\in[0,1)$ and (iv) $p$ is a nonnegative integer and $\alpha=p+1$ giving rise to Theorems~\ref{thm: error equiv}--\ref{thm: error equiv new} respectively.

We first consider the case where $p$ is a positive integer and $\alpha\in(p,p+1)$.
\begin{theorem}\label{thm: error equiv}
Let $\mu$ be a probability measure in the class $\mathcal M_p$ and $\alpha\in(p,p+1)$. The following statements are equivalent:
\let\myenumi\theenumi
\renewcommand{\theenumi}{\roman{enumi}}
\begin{enumerate}
\item $\mu(y,\infty)$ is regularly varying of index $-\alpha$. \label{tail}
\item $\Im r_G(iy)$ is regularly varying of index $-(\alpha-p)$. \label{Cauchy remainder}
\item $\Im r_\phi(iy)$ is regularly varying of index $-(\alpha-p)$, $\Re r_\phi(iy)\gg y^{-1}$ as $y\to\infty$ and $r_\phi(z)\gg z^{-1}$ as $z\to\infty$ n.t. \label{Voiculescu remainder}
\end{enumerate}
If any of the above statements holds, we also have, as $z\to\infty$ n.t.,
\begin{equation} \label{rg-rphi}
r_G(z)\sim r_\phi(z) \gg z^{-1};
\end{equation}
as $y\to\infty$,
\begin{equation}
\Im r_\phi(iy) \sim \Im r_G(iy) \sim -\frac{\frac{\pi(p+1-\alpha)}2}{\cos\frac{\pi(\alpha-p)}2} y^p \mu(y,\infty) \gg \frac1y \label{imrg-rphi}
\end{equation}
and
\begin{equation}
\Re r_\phi(iy) \sim \Re r_G(iy) \sim -\frac{\frac{\pi(p+2-\alpha)}2}{\sin\frac{\pi(\alpha-p)}2} y^p \mu(y,\infty) \gg \frac1y. \label{rerg-rphi}
\end{equation}
\end{theorem}

Next we consider the case where $p$ is a positive integer and $\alpha=p$.
\begin{theorem}\label{thm: error equiv eq p}
Let $\mu$ be a probability measure in the class $\mathcal M_p$. The following statements are equivalent:
\let\myenumi\theenumi
\renewcommand{\theenumi}{\roman{enumi}}
\begin{enumerate}
\item $\mu(y,\infty)$ is regularly varying of index $-p$. \label{tail eq p}
\item $\Im r_G(iy)$ is slowly varying. \label{Cauchy remainder eq p}
\item $\Im r_\phi(iy)$ is slowly varying, $\Re r_\phi(iy)\gg y^{-1}$ as $y\to\infty$ and $r_\phi(z)\gg z^{-1}$ as $z\to\infty$ n.t. \label{Voiculescu remainder eq p}
\end{enumerate}
If any of the above statements holds, we also have, as $z\to\infty$ n.t.,
\begin{equation} \label{rg-rphi eq p}
r_G(z)\sim r_\phi(z) \gg z^{-1};
\end{equation}
as $y\to\infty$,
\begin{equation}
\Im r_\phi(iy) \sim \Im r_G(iy) \sim -\frac\pi2 y^p \mu(y,\infty) \gg \frac1y \label{imrg-rphi eq p}
\end{equation}
and
\begin{equation}
\Re r_\phi(iy) \sim \Re r_G(iy) \gg \frac1y. \label{rerg-rphi eq p}
\end{equation}
\end{theorem}

In the third case, we consider $\alpha\in[0,1)$.
\begin{theorem}\label{thm: error equiv-0}
Let $\mu$ be a probability measure in the class $\mathcal M_0$ and $\alpha\in[0,1)$. The following statements are equivalent:
\let\myenumi\theenumi
\renewcommand{\theenumi}{\roman{enumi}}
\begin{enumerate}
\item $\mu(y,\infty)$ is regularly varying of index $-\alpha$. \label{tail-0}
\item $\Im r_G(iy)$ is regularly varying of index $-\alpha$. \label{Cauchy remainder-0}
\item $\Im r_\phi(iy)$ is regularly varying of index $-\alpha$, $\Re r_\phi(iy)\approx \Im r_\phi(iy)$ as $y\to\infty$ and $r_\phi(z)\gg z^{-1}$ as $z\to\infty$ n.t. \label{Voiculescu remainder-0}
\end{enumerate}
If any of the above statements holds, we also have, as $z\to\infty$ n.t.,
\begin{equation} \label{rg-rphi-0}
r_G(z)\sim r_\phi(z) \gg z^{-1};
\end{equation}
as $y\to\infty$,
\begin{equation}
\Im r_\phi(iy) \sim \Im r_G(iy) \sim -\frac{\frac{\pi(1-\alpha)}2}{\cos\frac{\pi\alpha}2}  \mu(y,\infty) \gg \frac1y \label{imrg-rphi-0}
\end{equation}
and
\begin{equation}
\Re r_\phi(iy) \sim \Re r_G(iy) \sim -d_\alpha \mu(y,\infty) \gg \frac1y. \label{rerg-rphi-0},
\end{equation}
where
$$d_\alpha =
\begin{cases}
  \frac{\frac{\pi(2-\alpha)}2}{\sin\frac{\pi\alpha}2}, &\text{when $\alpha>0$,}\\
  1, &\text{when $\alpha=0$.}
\end{cases}$$
\end{theorem}

Finally, we consider the case where $p$ is a nonnegative integer and $\alpha=p+1$.
\begin{theorem}\label{thm: error equiv new}
Let $\mu$ be a probability measure in the class $\mathcal M_p$ and $\beta\in(0,1/2)$. The following statements are equivalent:
\let\myenumi\theenumi
\renewcommand{\theenumi}{\roman{enumi}}
\begin{enumerate}
\item $\mu(y,\infty)$ is regularly varying of index $-(p+1)$. \label{tail new}
\item $\Re r_G(iy)$ is regularly varying of index $-1$. \label{Cauchy remainder new}
\item $\Re r_\phi(iy)$ is regularly varying of index $-1$, $y^{-1} \ll \Im r_\phi(iy)\ll y^{-(1-\beta/2)}$ as $y\to\infty$ and $z^{-1} \ll r_\phi(z) \ll z^{-\beta}$ as $z\to\infty$ n.t. \label{Voiculescu remainder new}
\end{enumerate}
If any of the above statements holds, we also have, as $z\to\infty$ n.t.,
\begin{equation} \label{rg-rphi new}
z^{-1}\ll r_G(z)\sim r_\phi(z)\ll z^{-\beta};
\end{equation}
as $y\to\infty$,
\begin{equation} \label{rerg-rphi new}
y^{-(1+\beta/2)} \ll \Re r_\phi(iy) \sim \Re r_G(iy) \sim -\frac\pi2 y^p \mu(y,\infty) \ll y^{-(1-\beta/2)}
\end{equation}
and
\begin{equation} \label{imrg-rphi new}
y^{-1} \ll \Im r_\phi(iy) \sim \Im r_G(iy) \ll y^{-(1-\beta/2)}.
\end{equation}
\end{theorem}

It is easy to obtain the equivalent statements for $H_\mu$ and $R_\mu$ through the simple observation that $G_\mu(z)=H_\mu(1/z)$ and $\phi_\mu(z)=R_\mu(1/z)$. For $p=0$, Theorems~\ref{thm: error equiv-0} and~\ref{thm: error equiv new} together give a special case of~\eqref{eq: phi mu} for the probability measures with regularly varying tail and infinite mean. However, Theorems~\ref{thm: error equiv}--\ref{thm: error equiv new} give more detailed asymptotic behavior of the real and imaginary parts separately, which is required for our analysis.
\subsection{Karamata type results} \label{subsec: karamata}
We provide here two results for regularly varying functions, which we shall be using in the proofs of our results. They are variants of Karamata's Abel-Tauber theorem for Stieltjes transform (see, Theorem~\ref{chap1:theo:karamata:stieltjes}) which explains the regular variation of Cauchy transform of measures with regularly varying tails.

The first result is quoted from \cite{bercovici1999stable}.
\begin{proposition}[\citealp{bercovici1999stable}, Corollary~5.4]\label{stieltjes}
Let $\rho$ be a positive Borel measure on $[0,\infty)$ and fix $\alpha\in[0,2)$. Then the following statements are equivalent:
\let\myenumi\theenumi
\renewcommand{\theenumi}{\roman{enumi}}
\begin{enumerate}
 \item $y\mapsto \rho[0,y]$ is regularly varying of index $\alpha$.
 \item $y\mapsto\int_0^{\infty}\frac{1}{t^2+y^2}d\rho(t)$ is regularly varying of index $-(2-\alpha)$.
\end{enumerate}
If either of the above conditions is satisfied, then
$$\int_0^{\infty}\frac{1}{t^2+y^2}d\rho(t)\sim \frac{\frac{\pi\alpha}{2}}{\sin\frac{\pi\alpha}{2}}\frac{\rho[0,y]}{y^2} \text{ as } y\rightarrow\infty.$$
The constant pre-factor on the right side is interpreted as $1$ when $\alpha=0$.
\end{proposition}

The second result uses a different integrand.
\begin{proposition}\label{stieltjes2}
Let $\rho$ be a finite positive Borel measure on $[0,\infty)$ and fix $\alpha\in[0,2)$. Then the following statements are equivalent:
\let\myenumi\theenumi
\renewcommand{\theenumi}{\roman{enumi}}
\begin{enumerate}
 \item $y\mapsto \rho(y,\infty)$ is regularly varying of index $-\alpha$.
 \item $y\mapsto\int_0^{\infty}\frac{t^2}{t^2+y^2}d\rho(t)$ is regularly varying of index $-\alpha$.
\end{enumerate}
If either of the above conditions is satisfied, then
$$\int_0^{\infty}\frac{t^2}{t^2+y^2}d\rho(t)\sim \frac{\frac{\pi\alpha}{2}}{\sin\frac{\pi\alpha}{2}}{\rho(y,\infty)} \text{ as } y\rightarrow\infty.$$
The constant pre-factor on the right side is interpreted as $1$ when $\alpha=0$.
\end{proposition}
\begin{proof}
Define $d\widetilde \rho(y)=\rho(\sqrt s,\infty)ds$. By a variant of Karamata's theorem given in Theorem~\ref{chap1:theorem:karamata}, as $\alpha<2$, we have
\begin{equation}\label{eq: stieltjes rho}
\widetilde\rho[0,y]\sim \frac1{1-\frac\alpha2} y \rho(\sqrt{y},\infty)
\end{equation}
is regularly varying of index $1-\alpha/2$. Then, we have,
\begin{multline*}
\int_0^\infty \frac{t^2}{t^2+y^2} d\rho(t) =y^2 \int_0^\infty \int_0^t \frac{2sds}{(s^2+y^2)^2} d\rho(t)\\ = y^2 \int_0^\infty \frac{2s \rho(s,\infty)}{(s^2+y^2)^2} ds = y^2 \int_0^\infty \frac{d\widetilde\rho(s)}{(s+y^2)^2}.
\end{multline*}
Now, first applying Theorem~\ref{chap1:theo:karamata:stieltjes}, as $\widetilde\rho[0,y]$ is regularly varying of index $1-\alpha/2\in(0,2]$ and then~\eqref{eq: stieltjes rho}, we have
$$\int_0^\infty \frac{t^2}{t^2+y^2} d\rho(t) \sim \frac{\left(1-\frac\alpha2\right) \frac{\pi\alpha}{2}}{\sin\frac{\pi\alpha}{2}} y^2\frac{\widetilde\rho[0,y^2]}{y^4} \sim \frac{\frac{\pi\alpha}{2}}{\sin\frac{\pi\alpha}{2}} \rho(y,\infty).$$
\end{proof}

\section{Free subexponentiality of measures with regularly varying tails} \label{sec: pf thm one}
We now use Theorems~\ref{thm: error equiv}--\ref{thm: error equiv new} to prove Theorem~\ref{main theorem-1}. We shall first look at the tail behavior of the free convolution of two probability measures with regularly varying tails and which are tail balanced. Theorem~\ref{main theorem-1} will be proved by suitable choices of the two measures.
\begin{lemma} \label{lem: free tail}
Suppose $\mu$ and $\nu$ are two probability measures on $[0,\infty)$ with regularly varying tails, which are tail balanced, that is, for some $c>0$, we have $\nu(y,\infty)\sim c\,\mu(y,\infty)$. Then
$$\mu\boxplus\nu(y,\infty) \sim (1+c) \mu(y,\infty).$$
\end{lemma}
\begin{proof}
We shall now indicate the associated probability measures in the remainder terms, moments and the cumulants to avoid any confusion. Since $\mu$ and $\nu$ are tail balanced and have regularly varying tails, for some nonnegative integer $p$ and $\alpha\in[p,p+1]$, we have both $\mu$ and $\nu$ in the same class $\mathcal M_{p,\alpha}$. When $\alpha\in[p,p+1)$, depending on the choice of $p$ and $\alpha$, we apply one of Theorems~\ref{thm: error equiv},~\ref{thm: error equiv eq p} and~\ref{thm: error equiv-0} on the imaginary parts of the remainder terms in Laurent expansion of Voiculescu transforms. On the other hand, for $\alpha=p+1$, we apply Theorem~\ref{thm: error equiv new} on the real parts of the corresponding objects. We work out only the case $\alpha\in[p,p+1)$ in details, while the other case $\alpha=p+1$ is similar.

For $\alpha\in[p, p+1)$, by Theorems~\ref{thm: error equiv}--\ref{thm: error equiv-0}, we have
\begin{align}
r_{\phi_\mu}(z) \gg z^{-1} \quad &\text{and} \quad r_{\phi_\nu}(z) \gg z^{-1}\label{Voi}\\
\Re r_{\phi_\mu}(-iy) \gg y^{-1} \quad &\text{and} \quad \Re r_{\phi_\nu}(-iy) \gg y^{-1}\label{real Voi}\\
\Im r_{\phi_\mu}(iy) \sim -\frac{\frac{\pi(p+1-\alpha)}2}{\cos \frac{\pi(\alpha-p)}2} y^{p} \mu(y,\infty) \quad &\text{and} \quad \Im r_{\phi_\nu}(iy) \sim -\frac{\frac{\pi(p+1-\alpha)}2}{\cos \frac{\pi(\alpha-p)}2} y^{p} \nu(y,\infty).\label{imag Voi}
\intertext{For $p=0$ and $\alpha\in[0,1)$, we further have}
\Im r_{\phi_\mu}(iy) \approx \Re r_{\phi_\mu}(iy) \approx \mu(y,\infty) \quad &\text{and} \quad \Im r_{\phi_\nu}(iy) \approx \Re r_{\phi_\nu}(iy) \approx \nu(y,\infty).\label{balance}
\end{align}

We also know that, both Voiculescu transforms and cumulants add up in case of free convolution. Hence,
\begin{equation}\label{sum Voi}
r_{\phi_{\mu\boxplus\nu}}(z) = r_{\phi_\mu}(z) + r_{\phi_\nu}(z).
\end{equation}
Further, we shall have
$\kappa_p(\mu\boxplus\nu)<\infty$, but $\kappa_{p+1}(\mu\boxplus\nu)=\infty$ and similar results hold for the moments of $\mu\boxplus\nu$ as well. Then Theorems~\ref{thm: error equiv}--\ref{thm: error equiv-0} will also apply for $\mu\boxplus\nu$. Thus, applying~\eqref{sum Voi} and its real and imaginary parts evaluated at $z=iy$, together with~\eqref{Voi}--\eqref{balance} respectively, we get,
$$r_{\phi_{\mu\boxplus\nu}}(z) \gg z^{-1} \text{ as $z\to\infty$ n.t.},$$
$$\Re r_{\phi_{\mu\boxplus\nu}}(iy) \gg y^{-1} \text{ as $y\to\infty$}$$ and
\begin{equation}
\Im r_{\phi_{\mu\boxplus\nu}}(iy) \sim -(1+c) \frac{\frac{\pi(p+1-\alpha)}2}{\cos \frac{\pi(\alpha-p)}2} y^{p} \mu(y,\infty) \text{ as $y\to\infty$}, \label{eq: asymp 1}
\end{equation}
which is regularly varying of index $-(\alpha-p)$. Further, for $p=0$ and $\alpha\in[0,1)$, we have
$$\Im r_{\phi_{\mu\boxplus\nu}}(iy) \approx \Re r_{\phi_{\mu\boxplus\nu}}(iy).$$
In the last two steps, we also use the hypothesis that $\nu(y,\infty) \sim c \mu(y,\infty)$ as $y\to\infty$. Thus, again using Theorems~\ref{thm: error equiv}--\ref{thm: error equiv-0}, we have
\begin{equation} \label{eq: asymp 2}
-\frac{\frac{\pi(p+1-\alpha)}2}{\cos \frac{\pi(\alpha-p)}2} y^{p} \mu\boxplus\nu(y, \infty) \sim \Im r_{\phi_{\mu\boxplus\nu}}(iy).
\end{equation}
Combining~\eqref{eq: asymp 1} and~\eqref{eq: asymp 2}, the result follows.
\end{proof}

We are now ready to prove the subexponentiality of a distribution with regularly varying tail.
\begin{proof}[Proof of Theorem~\ref{main theorem-1}]
Let $\mu$ be the probability measure on $[0,\infty)$ associated with the distribution function $F_+$. Then $\mu$ also has regularly varying tail of index $-\alpha$. We prove that
\begin{equation} \label{eq: ind conv}
\mu^{\boxplus n}(y,\infty) \sim n \mu(y,\infty), \text{ as $y\to\infty$}
\end{equation}
by induction on $n$. To prove~\eqref{eq: ind conv}, for $n=2$, apply Lemma~\ref{lem: free tail} with both the probability measures as $\mu$ and the constant $c=1$. Next assume~\eqref{eq: ind conv} holds for $n=m$. To prove~\eqref{eq: ind conv}, for $n=m+1$, apply Lemma~\ref{lem: free tail} again with the probability measures $\mu$ and $\mu^{\boxplus m}$ and the constant $c=m$.
\end{proof}

\section{Cauchy transform of measures with regularly varying tail} \label{sec: Cauchy}
As a first step towards proving Theorems~\ref{thm: error equiv}--\ref{thm: error equiv new}, we now collect some results about $r_G(z)$, when the probability measure $\mu$ has regularly varying tails. These results will be be useful in showing equivalence between the tail of $\mu$ and $r_G(iy)$. It is easy to see by induction that
$$\frac1{z-t} - \sum_{j=0}^p \frac{t^j}{z^{j+1}} = \left(\frac tz\right)^{p+1} \frac1{z-t}.$$
Integrating and multiplying by $z^{p+1}$, we get
\begin{equation}\label{eq: rG}
r_G(z) = \int_0^\infty \frac{t^{p+1}}{z-t} d\mu(t).
\end{equation}
We use~\eqref{eq: rG} to obtain asymptotic upper and lower bounds for $r_G(z)$ as $z\to\infty$ n.t. Similar results about $r_H$ can be obtained easily from the fact that $r_G(z)=r_H(1/z)$, but will not be stated separately. We consider the lower bound first.

\begin{proposition}\label{prop: lower bd rG}
Suppose $\mu\in\mathcal M_p$ for some nonnegative integer $p$, then
$$z^{-1}\ll r_G(z) \text{ as $z\to\infty$ n.t.}$$
\end{proposition}
\begin{proof}
We need to show that, for any $\eta>0$, as $|z|\to\infty$ with $z$ in the cone $\Gamma_\eta$, we have $|zr_G(z)|\to\infty$. Note that, for $z=x+iy\in\Gamma_\eta$, we have $|x|<\eta y$. Now, as $|z-t|^2 = (z-t)(\bar z-t)$ and $z(\bar z-t) = |z|^2 - zt$, using~\eqref{eq: rG}, we have,
$$zr_G(z) = z \int_0^\infty \frac{t^{p+1}}{z-t} d\mu(t) = |z|^2 \int_0^\infty \frac{t^{p+1}}{|z-t|^2} d\mu(t) - z \int_0^\infty \frac{t^{p+2}}{|z-t|^2} d\mu(t),$$
which gives
\begin{align}
\Re(zr_G(z)) &= |z|^2 \int_0^\infty \frac{t^{p+1}}{|z-t|^2} d\mu(t) - \Re z \int_0^\infty \frac{t^{p+2}}{|z-t|^2} d\mu(t)\label{eq: real part} \intertext{and}
\Im(zr_G(z)) &= \Im z \int_0^\infty \frac{t^{p+2}}{|z-t|^2} d\mu(t).\label{eq: imaginary part}
\end{align}

On $\Gamma_\eta$ and for $t\in[0,\eta y]$, $|t-x|\le t+|x|\le 2\eta y$. Thus, we have,
\begin{equation} \label{eq: divergence}
\int_0^\infty \frac{|z|^2 t^{p+1}}{|z-t|^2} d\mu(t) \ge \int_0^{\eta y} \frac{y^2 t^{p+1}}{(t-x)^2+y^2} d\mu(t) \ge \frac1{1+4\eta^2} \int_0^{\eta y} t^{p+1} d\mu(t) \to \infty,
\end{equation}
since $\mu\in\mathcal M_p$.

Now fix $\eta>0$, and consider a sequence $\{z_n=x_n+i y_n\}$ in $\Gamma_\eta$, such that $|z_n|\to\infty$, that is, $|x_n|\le uy_n$ and $y_n\to\infty$. If possible, suppose that $\{|z_n r_G(z_n)|\}$ is a bounded sequence, then both the real and the imaginary parts of the sequence will be bounded. However, then the boundedness of the real part and~\eqref{eq: real part} and~\eqref{eq: divergence} give
$$\left|\Re z_n \int_0^\infty \frac{t^{p+2}}{|z_n-t|^2} d\mu(t)\right|\to\infty.$$ Then, using~\eqref{eq: imaginary part} and the fact that $|\Re z| \le \eta \Im z$ on $\Gamma_\eta$, we have
$$\Im(z_nr_G(z_n)) \ge \frac1\eta \left|\Re z_n \int_0^\infty \frac{t^{p+2}}{|z_n-t|^2} d\mu(t)\right|\to\infty,$$
which contradicts the fact that the imaginary part of the sequence $\{z_nr_G(z_n)\}$ is bounded and completes the proof.
\end{proof}

We now consider the upper bound for $r_G(z)$. The result and the proof of the following proposition are inspired by Lemma~5.2(iii) of \cite{bercovici2000functions}.
\begin{proposition}\label{prop: upper bound rG}
Let $\mu$ be a probability measure in the class $\mathcal M_{p,\alpha}$ for some nonnegative integer $p$ and $\alpha\in(p,p+1]$. Then, for any $\beta\in [0,(\alpha-p)/(\alpha-p+1))$, we have
\begin{equation}\label{eq: upper bound rG}
r_G(z)=\lito(z^{-\beta}) \text{ as $z\to\infty$ n.t.}
\end{equation}
\end{proposition}

\begin{remark}
 We consider the consider principal branch of logarithm of a complex number with positive imaginary part, while defining the fractional powers in~\eqref{eq: upper bound rG} above and elsewhere.
\end{remark}

\begin{remark}
Note that~\eqref{eq: upper bound rG} holds also for $p=\alpha$ with $\beta=0$, which can be readily seen from Theorem~1.5 of \cite{benaych2006taylor}.
\end{remark}

\begin{proof}[Proof of Proposition~\ref{prop: upper bound rG}]
Define a measure $\rho_0$ as $d\rho_0(t)=t^pd\mu(t)$. Since $\mu\in\mathcal M_p$, $\rho_0$ is a finite measure. Further, since $p<\alpha$, using Theorem~\ref{lemma:karamata:df}~\ref{Kara:2}, we have $\rho_0(y,\infty) \sim -\frac\alpha{\alpha-p} y^p \mu(y, \infty)$, which is regularly varying of index $-(\alpha-p)$.

Now fix $\eta>0$. It is easy to check that for $t\ge0$ and $z\in\Gamma_\eta$, $t/|z-t| < \sqrt{1+\eta^2}$. For $z=x+iy$, we have $|z-t|>y$ and hence for $t\in[0,y^{1/(\alpha-p+1)}]$, we have $t/|z-t|<y^{-(\alpha-p)/(\alpha-p+1)}$. Then, using~\eqref{eq: rG} and the definition of $\rho_0$,
\begin{align*}
|r_G(z)| &\le \int_0^{y^{{1}/{(\alpha-p+1)}}} \left|\frac{t}{z-t}\right|d\rho_0(t) + \sqrt{1+\eta^2} \rho_0\left(y^{{1}/{(\alpha-p+1)}}, \infty\right)\\
&\le y^{-(\alpha-p)/(\alpha-p+1)} \int_0^\infty t^p d\mu(t) + \sqrt{1+\eta^2} \rho_0\left(y^{1/(\alpha-p+1)},\infty\right) = \lito(y^{-\beta}),
\end{align*}
for any $\beta\in[0,(\alpha-p)/(\alpha-p+1))$, as the second term is regularly varying of index $-(\alpha-p)/(\alpha-p+1)$. Further, for $z=x+iy\in\Gamma_\eta$, we have $|z|=\sqrt{x^2+y^2}\le y\sqrt{1+\eta^2}$, and hence we have the required result.
\end{proof}

Next we specialize to the asymptotic behavior of $r_G(iy)$, as $y\to\infty$. Observe that
\begin{equation}\label{rG y}
\Re r_G(iy)=-\int_0^\infty \frac{t^{p+2}}{t^2+y^2} d\mu(t) \quad \text{and} \quad \Im r_G(iy)=-y \int_0^\infty \frac{t^{p+1}}{t^2+y^2} d\mu(t).
\end{equation}

\begin{proposition} \label{prop: cauchy tail}
Let $\mu$ be a probability measure in the class $\mathcal M_p$.

If $\alpha\in(p,p+1)$, then the following statements are equivalent:
\let\myenumi\theenumi
\renewcommand{\theenumi}{\roman{enumi}}
\begin{enumerate}
\item $\mu$ has regularly varying tail of index $-\alpha$. \label{prop: cauchy meas tail}
\item $\Re r_G(iy)$ is  regularly varying of index $-(\alpha-p)$. \label{prop: cauchy real tail}
\item $\Im r_G(iy)$ is  regularly varying of index $-(\alpha-p)$. \label{prop: cauchy imag tail}
\end{enumerate}
If any of the above statements holds, then
$$\frac{\sin\frac{\pi(\alpha-p)}2}{\frac{\pi(p+2-\alpha)}2} \Re r_G(iy)\sim \frac{\cos\frac{\pi(\alpha-p)}2}{\frac{\pi(p+1-\alpha)}2} \Im r_G(iy)\sim -y^p \mu(y,\infty) \text{ as $y\to\infty$.}$$
Further, $\Re r_G(iy)\gg y^{-1}$ and $\Im r_G(iy)\gg y^{-1}$ as $y\to\infty$.

If $\alpha=p$, then the statements~\eqref{prop: cauchy meas tail} and~\eqref{prop: cauchy imag tail} above are equivalent. Also, if either of the statements holds, then
\begin{equation}\label{eq: imag equiv}
\Im r_G(iy)\sim -\frac\pi2y^p \mu(y,\infty) \text{ as $y\to\infty$.}
\end{equation}
Further, $\Im r_G(iy)\gg y^{-1}$ as $y\to\infty$.

If $\alpha=p+1$, then the statements~\eqref{prop: cauchy meas tail} and~\eqref{prop: cauchy real tail} above are equivalent. Also, if either of the statements holds, then
\begin{equation}\label{eq: real equiv}
\Re r_G(iy)\sim -\frac\pi2 y^p \mu(y,\infty) \text{ as $y\to\infty$.}
\end{equation}
Further, for any $\varepsilon>0$, $\Re r_G(iy) \gg y^{-(1+\varepsilon)}$ as $y\to\infty$.
\end{proposition}

\begin{remark}
Note that, for $\alpha=p+1$, $\Re r_G(iy)$ is regularly varying of index $-1$ and the asymptotic lower bound $\Re r_G(iy)\gg y^{-1}$ need not hold. This causes some difficulty in the proofs of Propositions~\ref{fraction-taylor} and~\ref{inverse-taylor}. The lack of the asymptotic lower bound has to be compensated for by the stronger upper bound obtained in Proposition~\ref{prop: upper bound rG}, which holds for $\alpha=p+1$. This is reflected in the condition~\eqref{R4} for the class $\mathcal R_{p,\beta}$ with $\beta>0$, defined in Section~\ref{sec: C-V reln}. Further note that, the situation reverses for $\alpha=p$, as Proposition~\ref{prop: upper bound rG} need not hold. The case, where $\alpha\in(p,p+1)$ is not an integer, is simple, as the asymptotic lower bounds hold for both the real and imaginary parts of $r_G(iy)$ (Proposition~\ref{prop: cauchy tail}), as well as, the stronger asymptotic upper bound works (Proposition~\ref{prop: upper bound rG}). However, the case of non-integer $\alpha\in(p,p+1)$ is treated simultaneously with the case $\alpha=p$ as the class $\mathcal R_{p,0}$ (cf. Section~\ref{sec: C-V reln}) in Propositions~\ref{fraction-taylor} and~\ref{inverse-taylor}.
\end{remark}

\begin{proof}[Proof of Proposition~\ref{prop: cauchy tail}]
The asymptotic lower bounds for the real and the imaginary parts of $r_G(iy)$ are immediate from~\eqref{prop: cauchy real tail} and~\eqref{prop: cauchy imag tail} respectively. So, we only need to show~\eqref{eq: real equiv} and the equivalence between~\eqref{prop: cauchy meas tail} and~\eqref{prop: cauchy real tail} when $\alpha\in(p,p+1]$ and~\eqref{eq: imag equiv} and the equivalence between~\eqref{prop: cauchy meas tail} and~\eqref{prop: cauchy imag tail} when $\alpha\in[p,p+1)$.

Let $d\rho_j(t)=t^{p+j}d\mu(t)$, for $j=1, 2$. Then, by Theorem~\ref{lemma:karamata:df}~\ref{Kara:1}, we have, for $\alpha\in[p,p+1)$, $\rho_1[0,y]\sim \alpha/(p+1-\alpha) y^{p+1} \mu(y,\infty)$, which is regularly varying of index $p+1-\alpha\in(0,1]$, and, for $\alpha\in(p,p+1]$, $\rho_2[0,y]\sim \alpha/(p+2-\alpha) y^{p+2} \mu(y,\infty)$, which is regularly varying of index $p+2-\alpha\in[1,2)$. Further, from~\eqref{rG y}, we get
$$\Re r_G(iy) = - \int_{0}^{\infty} \frac1{t^2+y^2}d\rho_2(t) \quad \text{and} \quad \Im r_G(iy) = -y \int_0^\infty \frac1{t^2+y^2} d\rho_1(t).$$
Then the results follow immediately from Proposition~\ref{stieltjes}.
\end{proof}

While asymptotic equivalences between $\Re r_G(iy)$ and tail of $\mu$ for $\alpha=p$ and $\Im r_G(iy)$ and tail of $\mu$ for $\alpha=p+1$ are not true in general, we obtain the relevant asymptotic bounds in these cases. We also obtain the exact asymptotic orders when $p=0$.
\begin{proposition} \label{prop: rG iy}
Consider a probability measure $\mu$ in the class $\mathcal M_p$.

If $\mu$ has regularly varying tail of index $-p$, then, for any $\varepsilon>0$, $\Re r_G(iy)\gg y^{-\varepsilon}$ as $y\to\infty$. Further, if $p=0$, then $\Re r_G(iy)\sim - \mu(y, \infty)$ as $y\to\infty$.

If $\mu$ has regularly varying tail of index $-(p+1)$, then $\Im r_G(iy)$ is regularly varying of index $-1$ and $y^{-1} \ll \Im r_G(iy)\ll y^{-(1-\varepsilon)}$ as $y\to\infty$, for any $\varepsilon>0$.
\end{proposition}

\begin{remark}
Note that, in the case $\alpha=p+1$, the lower bound for $\Im r_G(iy)$ is sharper than $\Re r_G(iy)$ and is same as that of $\Im r_G(iy)$ for the case $\alpha\in[p,p+1)$.
\end{remark}

\begin{proof}[Proof of Proposition~\ref{prop: rG iy}]
First consider the case where $\mu$ has regularly varying tail of index $-p$. Recall from the proof of Proposition~\ref{prop: upper bound rG} that $d\rho_0(t)=t^pd\mu(t)$. However, in the current situation Theorem~\ref{lemma:karamata:df}~\ref{Kara:1} will not apply. If $p=0$, then $\rho_0=\mu$ and $\rho_0(y,\infty)$ is slowly varying. If $p>0$, observe that, as $\int t^p d\mu(t) <\infty$, we have $$\rho_0(y,\infty)=y^p\mu(y,\infty)+p\int_0^y s^{p-1}\mu(s,\infty)ds\sim p\int_0^y s^{p-1}\mu(s,\infty)ds,$$
which is again slowly varying, where we use Theorem~\ref{chap1:theorem:karamata}. Thus, in either case, $\rho_0(y,\infty)$ is slowly varying and converges to zero as $y\to\infty$. Now, from~\eqref{rG y} and Proposition~\ref{stieltjes2}, we also have
$$\Re r_G(iy) = -\int_0^\infty \frac{t^2}{t^2+y^2} d\rho_0(t) \sim -\rho_0(y,\infty)$$
as $y\to\infty$. Since $\rho_0(y,\infty)$ is slowly varying, thus, for any $\varepsilon>0$, we have $|y^\varepsilon \Re r_G(iy)|\to\infty$ as $y\to\infty$. Also, for $p=0$, we have $\Re r_G(iy) \sim -\rho_0(y,\infty)=-\mu(y,\infty)$.

Next consider the case, where $\mu\in\mathcal M_p$ has regularly varying tail of index $-(p+1)$. Define again $d\rho_1(t) = t^{p+1} d\mu(t)$. Then, $$\rho_1[0,y] = (p+1) \int_0^y s^p \mu(s,\infty) ds - y^{p+1} \mu(y,\infty) \sim (p+1) \int_0^y s^p \mu(s,\infty) ds$$ is slowly varying, again by Theorem~\ref{chap1:theorem:karamata}. Then, by~\eqref{rG y} and Proposition~\ref{stieltjes}, we have
$$\Im r_G(iy) = y \int_0^\infty \frac{d\rho_1(t)}{t^2+y^2} \sim \frac1y \rho_1[0,y]$$
is regularly varying of index $-1$. Further, $\rho_1[0,y] \to \int_0^\infty t^{p+1} d\mu(t) = \infty$ as $y\to\infty$. Then the asymptotic upper and lower bounds follow immediately.
\end{proof}

\section{Relationship between Cauchy and Voiculescu transform} \label{sec: C-V reln}
The results of the last section relate the tail of a regularly varying probability measure and the behavior of the remainder term in Laurent series expansion of its Cauchy transform. In this section, we shall relate the remainder terms in Laurent series expansion of Cauchy and Voiculescu transforms. Finally, we collect the results from Sections~\ref{sec: Cauchy} and~\ref{sec: C-V reln} to prove Theorems~\ref{thm: error equiv}--\ref{thm: error equiv new}.

To study the relation between the remainder terms in Laurent series expansion of Cauchy and Voiculescu transforms, we consider a class of functions, which include the functions $H_\mu$ for the probability measures $\mu$ with regularly varying tails. We then show that when the inverse and the reciprocal of the functions in the class appropriately defined, the inverse and the reciprocal are also in the same class.

Let $\mathcal H$ denote the set of  analytic functions $A$ having a domain $\mathcal{D}_A$ such that for all positive $\eta$, there exists $\delta>0$ with $\Delta_{\eta,\delta}\subset \mathcal{D}_A$.

For a nonnegative integer $p$ and $\beta\in[0,1/2)$, let $\mathcal{R}_{p,\beta}$ denote the set of all functions $A\in\mathcal H$ which satisfy the following conditions:

\renewcommand{\labelenumi}{(R\arabic{enumi})}
\renewcommand{\theenumi}{R\myenumi}
\begin{enumerate}
\item $A$ has Taylor series expansion with real coefficients of the form
$$A(z) = z + \sum_{j=1}^p a_j z^{j+1} + z^{p+1} r_A(z),$$ \label{R1}
where $a_1, \ldots, a_p$ are real numbers. For $p=0$, we interpret the sum in the middle term as absent.
\item $z\ll r_A(z)\ll z^\beta$ as $z\to 0$ n.t. \label{R2}
\item $\Re r_A(-iy)\gg y^{1+\beta/2}$ and $\Im r_A(-iy)\gg y$ as $y\to 0+$.\label{R3}
\end{enumerate}
 For $p=0=\beta$, we further require that
\renewcommand{\labelenumi}{(R4$^\prime$)}
\renewcommand{\theenumi}{R4$^\prime$}
\begin{enumerate}
\item $\Re r_A(-iy) \approx \Im r_A(-iy)$ as $y\to 0+$.\label{R4prime}
\end{enumerate}
For $\beta\in(0,1/2)$, we further require that,
\renewcommand{\labelenumi}{(R4$^{\prime\prime}$)}
\renewcommand{\theenumi}{R4$^{\prime\prime}$}
\begin{enumerate}
 \item $\Re r_A(-iy)\ll y^{1-\beta/2} \quad \text{and} \quad \Im r_A(-iy)\ll y^{1-\beta/2} \,\, \text{as $y\to 0+$.}$ \label{R4}
\end{enumerate}
Note that the functions in $\mathcal R_{p,\beta}$ satisfy~\eqref{R1}--\eqref{R3} for $p\geq 1$. For $p=0=\beta$, the functions in $\mathcal R_{p,\beta}$ satisfy~\eqref{R1}--\eqref{R3} as well as~\eqref{R4prime}. Finally, for nonnegative integers $p$ and $\beta\in(0,1/2)$, the functions in $\mathcal R_{p,\beta}$ satify~\eqref{R1}--\eqref{R3} and~\eqref{R4}.

The classes $\mathcal R_{p,\beta}$ as $p$ varies nonnegative integers and $\beta$ varies over $[0,1/2]$, include the functions $H_\mu$ where $\mu\in \mathcal M_{p,\alpha}$ with $p$ varying over nonnegative integers and $\alpha$ varying over $[p,p+1]$.

{\bf Case I:} $p$ positive integer and $\alpha\in[p,p+1)$: By Proposition~\ref{prop: lower bd rG} and~\ref{prop: cauchy tail}, we have $H_\mu\in \mathcal R_{p,0}$.

{\bf Case II:} $p=0$, $\alpha\in[0,1)$: By  Proposition~\ref{prop: lower bd rG}, \ref{prop: cauchy tail} and~\ref{prop: rG iy}, $H_\mu\in \mathcal R_{0,0}$.

Proposition~\ref{prop: rG iy} is required to prove~\eqref{R4prime} for $p=\alpha=0$ only.

{\bf Case III:} $p$ nonnegative integer, $\alpha=p+1$: By Proposition~\ref{prop: upper bound rG} and~\ref{prop: rG iy}, $H_\mu$ will be in $\mathcal R_{p,\beta}$ for any $\beta\in(0,1/2)$.

% From Propositions~\ref{prop: lower bd rG} and~\ref{prop: cauchy tail}, this class with $\beta=0$ will include, in particular, $H_\mu$ for all probability measures $\mu\in\mathcal M_{p,\alpha}$, where $p$ is a positive integer and $\alpha\in[p,p+1)$. We do not impose the condition $\Re r_A(-iy) \approx \Im r_A(-iy)$ for $p>0$, as it may fail for some measures in $\mathcal M_{p,p}$. However, from Propositions~\ref{prop: lower bd rG},~\ref{prop: cauchy tail} and~\ref{prop: rG iy}, this class will also include $H_\mu$ for all probability measures $\mu$ with regularly varying tail of index $-\alpha$ with $\alpha\in[0,1)$. For the case $\alpha=p+1$, we need $\beta>0$. According to Propositions~\ref{prop: upper bound rG} and~\ref{prop: rG iy}, for any $\beta\in(0,1/2)$, $H_\mu$ will be in $\mathcal R_{p,\beta}$, whenever $\mu\in\mathcal M_{p,p+1}$.
We do not impose the condition $\Re r_A(-iy) \approx \Im r_A(-iy)$ for $p>0$, as it may fail for some measures in $\mathcal M_{p,p}$.

The first result deals with the reciprocals. Note that $U(z)$ and $zU(z)$ have the same remainder functions and if one belongs to the class $\mathcal H$, so does the other.
\begin{proposition}\label{fraction-taylor}
Suppose $zU(z)\in\mathcal{H}$ be a function belonging to $\mathcal R_{p, \beta}$ for some nonnegative integer $p$ and $0\le\beta<1/2$, such that $U$ does not vanish in a neighborhood of zero. Further assume that $V=1/U$ also belongs to $\mathcal H$. Then $zV(z)$ is also in $\mathcal R_{p,\beta}$. Furthermore, we have,
\renewcommand{\labelenumi}{(F\arabic{enumi})}
\renewcommand{\theenumi}{F\myenumi}
\begin{enumerate}
 \item  $r_V(z)\sim -r_U(z)$, as $z\to0$ n.t.,\label{taylor1}
\item  $\Re r_V(-iy)\sim -\Re r_U(-iy)$, as  $y\rightarrow 0+$, and \label{taylor2}
\item  $\Im r_V(-iy)\sim -\Im r_U(-iy)$, as  $y\rightarrow 0+$.\label{taylor3}
\end{enumerate}

\end{proposition}

The second result shows that for each of the above classes, when we consider a bijective function from the class, its inverse is also in the same class.
\begin{proposition}\label{inverse-taylor}
Suppose $U\in\mathcal{H}$ be a bijective function with the inverse in $\mathcal{H}$ as well and $U\in\mathcal R_{p,\beta}$ for some nonnegative integer $p$ and $0\le\beta<1/2$. Then the inverse $V$ is also in $\mathcal R_{p,\beta}$. Furthermore, we have,
\renewcommand{\labelenumi}{(I\arabic{enumi})}
\renewcommand{\theenumi}{I\myenumi}
\begin{enumerate}
 \item  $r_V(z)\sim -r_U(z)$, as $z\to0$ n.t.,\label{inverse-taylor1}
\item  $\Re r_V(-iy)\sim -\Re r_U(-iy)$,  as $y\rightarrow 0+$, and\label{inverse-taylor2}
\item  $\Im r_V(-iy)\sim -\Im r_U(-iy)$, as $y\rightarrow 0+$. \label{inverse-taylor3}
\end{enumerate}
\end{proposition}

Next we prove Propositions~\ref{fraction-taylor} and~\ref{inverse-taylor}. In both the proofs, all the limits will be taken as $z\to0$ n.t.\ or $y\to0+$, unless otherwise mentioned and these conventions will not be stated repeatedly. We shall also use that, for any nonnegative integer $p$ and $\beta\in[0,1/2)$, with $U\in\mathcal R_{p,\beta}$, we have
\begin{equation} \label{eq: real imag}
|\Re r_U(-iy)|\le |r_U(-iy)|\ll 1\quad\text{and}\quad|\Im r_U(-iy)|\le |r_U(-iy)|\ll 1.
\end{equation}
The proofs of Propositions~\ref{fraction-taylor} and~\ref{inverse-taylor} will be broken down into cases $p=0$ and $p\geq1$. Each of these cases will be further split into subcases $\beta=0$ and $\beta\in(0,1/2)$. The case $p\geq 1$ is more invloved compared to the case $p=0$. However, the proofs, specially that of Proposition~\ref{inverse-taylor}, has substantial part in common.

We first prove the result regarding the reciprocal.
\begin{proof}[Proof of Proposition~\ref{fraction-taylor}]
Observe that if we verify~\eqref{taylor1}--\eqref{taylor3}, then $zV(z)$ is automatically in $\mathcal R_{p,\beta}$ as well, since $V(z)$ and $zV(z)$ have same remainder functions. We shall prove~\eqref{taylor1}--\eqref{taylor3} using the fact that $V(z)=1/U(z)$ and the properties of $U$ as an element of $\mathcal R_{p,\beta}$.

{\bf Case I: $p=0$.} Let $zU(z)=z+zr_U(z)$ be a function in this class. Then $V(z)=1-r_U(z)+\bigo(|r_U(z)|^2)$. By uniqueness of Taylor expansion from Lemma~A.1 of \cite{benaych2006taylor}, we have
\begin{equation} \label{eq: recip remainder}
r_V(z)=-r_U(z)+\bigo(|r_U(z)|^2).
\end{equation}
Since, by~\eqref{R2}, $r_U(z)\ll 1$, we have $r_V(z)\sim-r_U(z)$, which checks~\eqref{taylor1}.

Further, evaluating~\eqref{eq: recip remainder} at $z=-iy$ and equating the real and the imaginary parts, we have
\begin{align*}
\Re r_V(-iy)&=-\Re r_U(-iy)+\bigo(|r_U(-iy)|^2)\intertext{and}\Im r_V(-iy)&=-\Im r_U(-iy)+\bigo(|r_U(-iy)|^2).
\end{align*}
Thus, to obtain the equivalences~\eqref{taylor2} and~\eqref{taylor3}, it is enough to show that $|r_U(-iy)|^2 = |\Re r_U(-iy)|^2 + |\Im r_U(-iy)|^2$ is negligible with respect to both the real and the imaginary parts of $r_U(-iy)$. We prove the negligibility seperately for two subcases $\beta=0$ and $\beta\in(0,1/2)$.

{\bf Subcase Ia: $p=0$, $\beta=0$.} Using~\eqref{eq: real imag} and $\Re r_U(-iy) \approx \Im r_U(-iy)$ from~\eqref{R4prime}, we have the required negligibility condition.

 {\bf Subcase Ib: $p=0$, $\beta\in(0,1/2)$.} Using~\eqref{R3} and~\eqref{R4}, we have
\begin{align*}
\frac{|\Im r_U(-iy)|^2}{|\Re r_U(-iy)|} &= \frac{y^{1+\beta/2}}{|\Re r_U(-iy)|} \left(\frac{|\Im r_U(-iy)|}{y^{1-\beta/2}}\right)^2 y^{1-3\beta/2}\to0\intertext{and}
\frac{|\Re r_U(-iy)|^2}{|\Im r_U(-iy)|} &= \frac{y}{|\Im r_U(-iy)|} \left(\frac{|\Re r_U(-iy)|}{y^{1-\beta/2}}\right)^2 y^{1-\beta}\to0.
\end{align*}
They, together with~\eqref{eq: real imag}, give the required negligibility condition, thus proving~\eqref{taylor2} and~\eqref{taylor3}.

{\bf Case II: $p\geq1$.} Let $zU(z)= z + \sum_{j=1}^p u_j z^{j+1} + z^{p+1} r_U(z)$ be a function in this class. Note that, as $p\ge1$ and by~\eqref{R2}, as $z\ll r_U(z)$, we have $\sum_{j=1}^p u_j z^j + z^p r_U(z) = u_1 z + \bigo( z r_U(z) )$. Thus, using~\eqref{R2}, we have,
$$V(z)=1+\sum_{j=1}^p (-1)^j \left(\sum_{m=1}^p u_m z^{m} + z^{p} r_U(z)\right)^j+ (-1)^{p+1}u_1^{p+1} z^{p+1} + \bigo( z^{p+1} r_U(z) ).$$

Now we expand the second term on the right side. As $z\ll r_U(z)$ from~\eqref{R2}, all powers of $z$ with indices greater than $(p+1)$ can be absorbed in the last term on the right side. Then collect the $(p+1)$-th powers of $z$ in the second and third terms to get $c_1 z^{p+1}$ for some real number  $c_1$. The remaining powers of $z$ form a polynomial $P(z)$ of degree at most $p$ with real coefficients. Finally we consider the terms containing some power of $r_U(z)$. It will contain terms of the form $z^{l_1} (z^p r_U(z))^{l_2}$ for integers $l_1\ge 0$ and $l_2\ge 1$, with the leading term being $-z^p r_U(z)$. Since $p\ge1$ and from~\eqref{R2} we have $r_U(z)\ll 1$, the remaining terms can be absorbed in the last term on the right side. Thus, we get,
$$V(z)=1+P(z)-z^pr_U(z)+ c_1 z^{p+1} + \bigo(z^{p+1} r_U(z)).$$
By uniqueness of Taylor series expansion from Lemma~A.1 of \cite{benaych2006taylor}, we have
$$r_V(z) = -r_U(z)+c_1 z + \bigo(zr_U(z)).$$

The form of $r_V$ immediately gives $r_V(z) \sim -r_U(z)$, since $z\ll r_U(z)$, by~\eqref{R2}. This proves~\eqref{taylor1}.

 Also, using~\eqref{eq: real imag}, $\Im r_V(-iy) = -\Im r_U(-iy) + \bigo(y)$ and as $y\ll \Im r_U(-iy)$ from~\eqref{R3}, we have $\Im r_V(-iy) \sim -\Im r_U(-iy)$. This shows~\eqref{taylor3}. Further, as $c_1$ is real, $\Re r_V(-iy) =-\Re r_U(-iy) + \bigo( y |r_U(-iy)|)$. Thus, to conclude~\eqref{taylor2}, it is enough to show that $y |r_U(-iy)| \ll \Re r_U(-iy)$, for which it is enough to show that $y \Im r_U(-iy) \ll \Re r_U(-iy)$. We show this seperately for two subcases.

{\bf Subcase IIa: $p\geq 1$, $\beta=0$.} We have by~\eqref{R3}, $$\frac{y \Im r_U(-iy)}{\Re r_U(-iy)} =
\frac{y}{\Re r_U(-iy)} \cdot \Im r_U(-iy).$$

{\bf Subcase IIb: $p\geq1$, $\beta\in(0,1/2)$.} Using the properties~\eqref{R3} and~\eqref{R4} we get,
$$\frac{y \Im r_U(-iy)}{\Re r_U(-iy)} =
\frac{y^{1+\beta/2}}{\Re r_U(-iy)} \cdot \frac{\Im r_U(-iy)}{y^{1-\beta/2}} \cdot y^{1-\beta},$$

It is easy to see that the limit is zero in either subcase.
\end{proof}

Before proving the result regarding the inverse, we provide a result connecting a function in the class $\mathcal H$ and its derivative.
\begin{lemma}\label{lem: deriv}
Let $v\in\mathcal H$ satisfy $v(z)=\lito(z^\beta)$ as $z\to0$ n.t., for some real number $\beta$. Then $v^\prime(z) = \lito(z^{\beta-1})$ as $z\to0$ n.t.
\end{lemma}
\begin{proof}
The result for $\beta=0$ follows from the calculations in the proof of Proposition~A.1(ii) of~\cite{benaych2006taylor}. For the general case, define $w(z) = z^{-\beta} v(z)$. Then $w\in\mathcal H$ and $w(z)=\lito(1)$. So by the case $\beta=0$, we have $w^\prime(z) = -\beta z^{-\beta-1} v(z) + z^{-\beta} v^\prime(z)=\lito(z^{-1})$. Thus, $z w^\prime(z) = -\beta z^{-\beta} v(z) + z^{-(\beta-1)} v^\prime(z)$, where the left side and the first term on the right side are $\lito(1)$ and hence the second term on the right side is $\lito(1)$ as well.
\end{proof}

We are now ready to prove the result regarding the inverse.
\begin{proof}[Proof of Proposition~\ref{inverse-taylor}] We begin with some estimates which work for all values of $p$ and $\beta$ before breaking into cases and subcases.
Since $U$ is of the form
$$U(z) = z + \sum_{j=1}^p u_j z^{j+1} + z^{p+1} r_U(z)$$
and $r_U(z)\ll1$, by Proposition~A.3 of \cite{benaych2006taylor}, the inverse function $V$ also has the same form with the remainder term $r_V$ satisfying
\begin{equation}\label{eq: rv}
r_V(z)\ll1.
\end{equation}
Also note that $V(z)\sim z$. Further, Lemma~A.1 of \cite{benaych2006taylor} shows that the coefficients are determined by the limits of the derivatives of the function at $0$. Hence, the real coefficients of $U$ guarantee that the coefficients of $V$ are real. So we only need to check the asymptotic equivalences of the remainder functions given in~\eqref{inverse-taylor1}--\eqref{inverse-taylor3}. We shall achieve this by analyzing $I(z)=r_U(V(z))-r_U(z)$, the fact that $U(V(z))=z$ and the properties of $U$ as an element in $\mathcal R_{p,\beta}$. For that purpose, we define
$$I(z)=r_U(V(z))-r_U(z)=\int_{\gamma_z} r_U^\prime(\zeta) d\zeta,$$
where $\gamma_z$ is the closed line segment joining $z$ and $V(z)$. By definition of the class $\mathcal H$, given any $\eta>0$, there exists $\delta>0$, such that $\Delta_{\eta, \delta}\subset\mathcal D_U$. Since $V(z)\sim z$ as $z\to0$ n.t., given any $\eta>0$, there exists $\delta>0$, such that both $z$ and $V(z)$ belong to $\Delta_{\eta,\delta}$ and hence $\gamma_z$ is contained in $\Delta_{\eta, \delta}\subset\mathcal D_U$. (Note that $\Delta_{\eta, \delta}$ is a convex set.) Thus $r_U^\prime$ is defined on the entire line $\gamma_z$. We shall need the following estimate that
$$|I(z)| \le |\gamma_z| \sup_{\zeta\in\gamma_z} |r_U^\prime(\zeta)| = |V(z)-z| \sup_{\zeta\in\gamma_z} |r_U^\prime(\zeta)| = |V(z)-z| |r_U^\prime(\zeta_0(z))|,$$
for some $\zeta_0(z)\in\gamma_z$, since $\gamma_z$ is compact. Note that $\zeta_0(z)=z+\theta(z)(V(z)-z)$, for some $\theta(z)\in[0,1]$ and hence $\zeta_0(z)\sim z$. Now, $r_U(z)=\lito(z^\beta)$ by~\eqref{R2} and thus, by Lemma~\ref{lem: deriv}, we have $r_U^\prime(\zeta_0(z)) = \lito(\zeta_0(z)^{\beta-1}) = \lito(z^{\beta-1})$. Further estimates for $I(z)$ depend on the functions of $V(z)$ which are separate for the cases $p=0$ and $p\geq 1$. Using $V(z)=z+zr_V(z)$ for $p=0$ and $V(z)=z+\bigo(z^2)$ for $p\ge 1$, we have,
\begin{equation}\label{eq: I est}
|I(z)| =
\begin{cases}
\lito(z^\beta r_V(z)), &\text{for $p=0$,}\\
\lito(z^{1+\beta}), &\text{for $p\ge 1$.}
\end{cases}
\end{equation}

{\bf Case I: $p=0$.} Then $U(z)=z+zr_U(z)$ and $V(z)=z+zr_V(z)$. Using $U(V(z))=z$ and $I(z)=r_U(V(z))-r_U(z)$, we get $0 = zr_V(z) + (z+zr_V(z))(r_U(z)+I(z))$. Further canceling $z$ and using~\eqref{eq: rv}, we have
\begin{equation}\label{eq: rem p0}
0=r_U(z)+r_V(z)+r_U(z)r_V(z)+\bigo(I(z)).
\end{equation}
Using~\eqref{eq: I est} for $p=0$ and $r_U(z)\ll 1$ from~\eqref{R2}, we have $r_V(z)\sim-r_U(z)$, which proves~\eqref{inverse-taylor1}. Further, using~\eqref{R2} and evaluating at $z=-iy$, we have, for $\beta\in[0,1/2)$,
\begin{equation} \label{eq: r0beta}
|r_V(-iy)|\ll y^\beta.
\end{equation}
Evaluating~\eqref{eq: rem p0} at $z=iy$ and equating the real and the imaginary parts, we have
\begin{align}
0=\Re r_U(-iy)+\Re r_V(-iy) + \bigo(|r_U(-iy)| |r_V(-iy)|) + \bigo(|I(-iy)|)\label{eq: inv real p0}
\intertext{and}
0=\Im r_U(-iy)+\Im r_V(-iy) + \bigo(|r_U(-iy)| |r_V(-iy)|) + \bigo(|I(-iy)|).\label{eq: inv imag p0}
\end{align}
We split the proofs of~\eqref{inverse-taylor2} and~\eqref{inverse-taylor3} for the case $p=0$ into further subcases $\beta=0$ and $\beta\in(0,1/2)$.

{\bf Subcase Ia: $p=0$, $\beta=0$.} By~\eqref{inverse-taylor1} for $z=-iy$ and~\eqref{R4prime}, we have, $$|I(-iy)|\ll |r_V(-iy)|\sim|r_U(-iy)|\approx|\Re r_U(-iy)|\approx|\Im r_U(-iy)|.$$ Thus, the last term on the right hand side of~\eqref{eq: inv real p0} and~\eqref{eq: inv imag p0} are negligible with respect to $\Re r_U(-iy)$ and $\Im r_U(-iy)$ respectively. Then, further using $r_U(-iy)\to0$ from~\eqref{R2}, the third term on the right hand side of~\eqref{eq: inv real p0} and~\eqref{eq: inv imag p0} are negligible with respect to $\Re r_U(-iy)$ and $\Im r_U(-iy)$ respectively and hence we get $\Re r_U(-iy)\sim-\Re r_V(-iy)$ and $\Im r_U(-iy)\sim-\Im r_V(-iy)$, which prove~\eqref{inverse-taylor2} and~\eqref{inverse-taylor3}.

{\bf Subcase Ib: $p=0$, $\beta\in(0,1/2)$.} We have, by~\eqref{R3} and~\eqref{R4},
\begin{align*}
y^\beta \frac{|\Im r_U(-iy)|}{|\Re r_U(-iy)|} &= \frac{|\Im r_U(-iy)|}{y^{1-\beta/2}} \frac{y^{1+\beta/2}}{|\Re r_U(-iy)|} \to 0
\intertext{and}
y^\beta \frac{|\Re r_U(-iy)|}{|\Im r_U(-iy)|} &= \frac{|\Re r_U(-iy)|}{y^{1-\beta/2}} \frac{y}{|\Im r_U(-iy)|} y^{\beta/2} \to 0.
\end{align*}
They, together with~\eqref{eq: real imag}, give $y^\beta |r_U(-iy)|$ is negligible with respect to both the real and the imaginary parts of $r_U(-iy)$. Further, using~\eqref{eq: I est} and~\eqref{eq: r0beta} respectively, we have $$|I(-iy)|\ll y^\beta |r_V(-iy)| \sim y^\beta |r_U(-iy)|\quad\text{and}\quad |r_U(-iy) r_V(-iy)| \ll y^\beta |r_U(-iy)|.$$ Thus, both $|I(-iy)|$ and $|r_U(-iy) r_V(-iy)|$ which are the last two terms of~\eqref{eq: inv real p0} and~\eqref{eq: inv imag p0}, are negligible with respect to both the real and the imaginary parts of $r_U(-iy)$. Then, from~\eqref{eq: inv real p0} and~\eqref{eq: inv imag p0}, we immediately have $\Re r_U(-iy)\sim-\Re r_V(-iy)$ and $\Im r_U(-iy)\sim-\Im r_V(-iy)$, which prove~\eqref{inverse-taylor2} and~\eqref{inverse-taylor3}.

{\bf Case II: $p\ge 1$.} In this case $U(z)=z+\sum_{j=1}^pu_jz^{j+1}+z^{p+1}r_U(z)$ and $V(z) = z + \sum_{j=1}^pv_jz^{j+1} + z^{p+1}r_V(z) = z(1+v_1z(1+\lito(1)))$. Using $z=U(V(z))$ and canceling $z$ on both sides,  we have
\begin{multline} \label{eq: compose}
0 = \sum_{j=1}^pv_jz^{j+1} + z^{p+1}r_V(z) + \sum_{m=1}^p u_m \left( z + \sum_{j=1}^pv_jz^{j+1} + z^{p+1}r_V(z) \right)^{m+1}\\ + z^{p+1} \Big(r_U(z)+I(z)\Big) \Big(1+(p+1)v_1z(1+\lito(1))\Big).
\end{multline}
Note that all the coefficients on the right side are real. We collect the powers of $z$ till degree $p+1$ on the right side in the polynomial $Q(z)$. Let $c'\in\mathbb R$ be the coefficient of $z^{p+2}$ on the right side. The remaining powers of $z$ on the right side will be $\bigo(z^{p+3})$. We next consider the terms with $r_V(z)$ as a factor and observe that $z^{p+1} r_V(z)$ is the leading term and the remaining terms contribute $\bigo(z^{p+2} r_V(z))$. Finally, the last term on the right side gives $z^{p+1} r_U(z) + \bigo(z^{p+2} r_U(z)) + \bigo(z^{p+1} I(z))$. Since $z\ll r_U(z)$ by~\eqref{R2}, the term $\bigo(z^{p+3})$ can be absorbed in $\bigo(z^{p+2} r_U(z))$. Combining the above facts and dividing~\eqref{eq: compose} by $z^{p+1}$, we get,
\begin{equation}
\label{eq: casep1aftercompose}
0=z^{-(p+1)} Q(z) + \big( r_U(z) + c' z + \bigo(I(z)) + \bigo(z r_U(z)) \big) + \big( r_V(z) + \bigo(z r_V(z)) \big).\end{equation}

%Now, the first group of terms within the bracket is $r_U(z)(1+\lito(1))=\lito(1)$,
As $I(z)\ll z^{1+\beta}\ll z\ll r_U(z)$ by~\eqref{eq: I est} and~\eqref{R2}, we have $r_U(z)+c'z+\bigo(I(z))+\bigo(zr_U(z))=r_U(z)(1+\lito(1))$. Also $r_V(z)+\bigo(zr_V(z))=r_V(z)(1+\lito(1))$. Thus, the last two terms on the right hand side of~\eqref{eq: casep1aftercompose} goes to zero.  However, the first term on the right hand side of~\eqref{eq: casep1aftercompose}, $Q$ being a polynomial of degree at most $p$, becomes unbounded unless $Q\equiv0$. So we must have $Q\equiv0$. Thus, ~\eqref{eq: casep1aftercompose} simplifies to
\begin{equation}\label{eq: compose 2}
r_U(z) + c' z + \bigo(I(z)) + \bigo(z r_U(z)) = - r_V(z) + \bigo(z r_V(z)).
\end{equation}
As observed earlier, the left side is $r_U(z) (1+\lito(1))$ and the right side is $- r_V(z) (1+\lito(1))$ giving $r_U(z) \sim - r_V(z)$, which proves~\eqref{inverse-taylor1}.

Further, as in the case $p=0$, we have~\eqref{eq: r0beta} from $r_U(z)\sim-r_V(z)$.  Also,~\eqref{eq: compose 2} becomes
\begin{equation} \label{eq: compose 3}
-r_V(z) = r_U(z) + c' z + \bigo(I(z)) + \bigo(z r_U(z)).
\end{equation}
Evaluating~\eqref{eq: compose 3} at $z=-iy$ and equating the imaginary parts, we have, using~\eqref{eq: I est},
$$-\Im r_V(-iy) = \Im r_U(-iy) + \bigo(y).$$ This gives~\eqref{inverse-taylor3}, that is, $-\Im r_V(-iy) \sim \Im r_U(-iy)$, since $y\ll\Im r_U(-iy)$ by~\eqref{R3}.

Evaluating~\eqref{eq: compose 2} at $z=-iy$ again and now equating the real parts, we have, as $c'$ is real,
$$-\Re r_V(-iy) = \Re r_U(-iy) + \bigo(|I(-iy)|) + \bigo(y|r_U(-iy)|).$$ From~\eqref{eq: I est} and~\eqref{R3}, we have $|I(-iy)|\ll y^{1+\beta} \ll \Re r_U(-iy)$. Thus, to obtain~\eqref{inverse-taylor2}, that is, $-\Re r_V(-iy) \sim \Re r_U(-iy)$, we only need to show that $y|r_U(-iy)|\ll\Re r_U(-iy)$, which follows using $r_U(-iy)\sim-r_V(-iy)$, \eqref{eq: r0beta} and~\eqref{R3}, since
$$\frac{y|r_U(-iy)|}{|\Re r_U(-iy)|} = \frac{y^{1+\beta/2}}{|\Re r_U(-iy)|} \frac{|r_U(-iy)|}{y^\beta} y^{\beta/2}.$$
\end{proof}

We wrap up the Chapter by collecting the results from Sections~\ref{sec: Cauchy} and~\ref{sec: C-V reln} and proving Theorems~\ref{thm: error equiv}--\ref{thm: error equiv new}.
\begin{proof}[Proofs of Theorems~\ref{thm: error equiv}--\ref{thm: error equiv new}]
We shall prove all the theorems together, as the proofs are very similar.

The statements involving the tail of the probability measure $\mu$ and the remainder term in Laurent expansion of Cauchy transform, $r_{G_\mu}$ can be obtained from the results in Section~\ref{sec: Cauchy} as follows:
For all the theorems, the equivalence of the statements~\eqref{tail} and~\eqref{Cauchy remainder} about the tail of the probability measure and Cauchy transform (the imaginary part in Theorems~\ref{thm: error equiv}--\ref{thm: error equiv-0} and the real part in Theorem~\ref{thm: error equiv new}) are given in Proposition~\ref{prop: cauchy tail}. The asymptotic equivalences between the tail of the measure and (the real and the imaginary parts of) the remainder term in Laurent series expansion of Cauchy transform, given in~\eqref{imrg-rphi},~\eqref{rerg-rphi},~\eqref{imrg-rphi eq p},~\eqref{imrg-rphi-0} and~\eqref{rerg-rphi new} are also given in Proposition~\ref{prop: cauchy tail}. The similar asymptotic equivalence in~\eqref{rerg-rphi-0} follows from Propositions~\ref{prop: cauchy tail} and~\ref{prop: rG iy} for the cases $\alpha\in(0,1)$ and $\alpha=1$ respectively. We consider the asymptotic upper and lower bounds next. The asymptotic lower bounds in~\eqref{rg-rphi},~\eqref{rg-rphi eq p},~\eqref{rg-rphi-0} and~\eqref{rg-rphi new} follow from Proposition~\ref{prop: lower bd rG}. The asymptotic upper bound in~\eqref{rg-rphi new} follows from Proposition~\ref{prop: upper bound rG}. The asymptotic lower bounds in~\eqref{imrg-rphi},~\eqref{rerg-rphi},~\eqref{imrg-rphi eq p},~\eqref{imrg-rphi-0} and~\eqref{rerg-rphi new} follow from Proposition~\ref{prop: cauchy tail}. The asymptotic upper bound in~\eqref{rerg-rphi new} follows from the fact that $y^p \mu(y,\infty)$ is a regularly varying function of index $-1$. The asymptotic lower bound in~\eqref{rerg-rphi eq p} follows from Proposition~\ref{prop: rG iy}, while the asymptotic lower bound in~\eqref{rerg-rphi-0} follows as the tail of the measure is regularly varying of index $-\alpha$ with $\alpha\in[0,1)$. Finally both the asymptotic bounds in~\eqref{imrg-rphi new} follow from Proposition~\ref{prop: rG iy}.

To complete the proofs of Theorems~\ref{thm: error equiv}--~\ref{thm: error equiv new}, we need to check the equivalence of the statements~\eqref{Cauchy remainder} and~\eqref{Voiculescu remainder} involving the remainder terms in Laurent expansion of Cauchy and Voiculescu transforms for all the theorems and the asymptotic equivalences between the remainder terms in Laurent series expansion of Cauchy and Voiculescu transforms and their real and imaginary parts given in~\eqref{rg-rphi}--\eqref{imrg-rphi new}. Note that all these claims about Cauchy and Voiculescu transforms of $\mu$ have analogues about $H_\mu$ and $R_\mu$ due to the facts that $r_G(z)=r_H(1/z)$ and $r_\phi(z)=r_R(1/z)$. We shall actually deal with the functions $H_\mu$ and $R_\mu$.

For any probability measure $\mu\in \mathcal M_p$, the function $H\equiv H_\mu$ is invertible, belongs to the class $\mathcal H$ and the leading term of its Taylor expansion is $z$. Further, by Proposition~A.3 of~\cite{benaych2006taylor}, the above statement about $H$ is equivalent to the same statement about its inverse, denoted by $L\equiv L_\mu$. Since the leading term of Taylor expansion of $L$ has leading term $z$, the leading term of Taylor expansion of $L(z)/z$ is $1$ and it is also in $\mathcal H$. Define $K(z)=z/L(z)$. Then $K$ is also in $\mathcal H$ and its Taylor expansion has leading term $1$. We shall also use the following facts obtained from~\eqref{eq: R phi defn}:
\begin{equation} \label{eq: R-K}
zR_\mu(z) = (K(z) - 1) \quad \text{and} \quad zK(z) = z (1+zR_\mu(z)).
\end{equation}
Hence Taylor expansion of $K$ will also lead to a Taylor expansion of $R$ of degree one less than that of $K$. However, due to the definition of the remainder term of Taylor expansion given in~\eqref{eq: def remainder}, the corresponding remainder terms will be related by $r_K\equiv r_R$. Thus, we can move from the function $r_H$ to $r_K (\equiv r_R)$ through inverse and reciprocal and vice versa as follows:
\begin{equation} \label{scheme}
H(z) \xleftrightarrow[\text{Proposition~\ref{inverse-taylor}}]{L(z)=H^{-1}(z)} L(z)=z\cdot\frac{L(z)}z \xleftrightarrow[\text{Proposition~\ref{fraction-taylor}}]{K(z)=\frac{z}{L(z)}} zK(z) \xleftrightarrow[r_K=r_R]{R(z)=\frac{K(z)-1}z} R(z).
\end{equation}
These observations set up the stage for Propositions~\ref{fraction-taylor} and~\ref{inverse-taylor}. We shall use the class $\mathcal R_{p,0}$ for Theorems~\ref{thm: error equiv}--\ref{thm: error equiv-0} and the class $\mathcal R_{p,\beta}$ with any $\beta\in(0,1/2)$ for Theorem~\ref{thm: error equiv new}.

Suppose $\mu\in\mathcal M_p$ with $\alpha\in[p,p+1)$. This condition holds for Theorems~\ref{thm: error equiv}--\ref{thm: error equiv-0} and we prove these three theorems first. In these cases, $H_\mu(z)$ and $zK(z) = z(1+zR_\mu(z))$ necessarily have Taylor expansions of the form given in the hypothesis~\eqref{R1} for the class $\mathcal R_{p,0}$ with $r_H(z)\ll1$ and $r_R(z)\ll1$ as $z\to\infty$.

For all three theorems, first assume the statement~\eqref{Cauchy remainder} that $\Im r_G(iy)$ is regularly varying of index $-(\alpha-p)$. Then, from the already proven lower bounds in~\eqref{rg-rphi}--\eqref{rerg-rphi-0}, we have the asymptotic lower bounds for $r_G(z)$, $\Re r_G(iy)$ and $\Im r_G(iy)$ under the setup of each of the three theorems. They translate to the asymptotic lower bounds for the function $H_\mu$, as required by the hypotheses~\eqref{R2} and~\eqref{R3}. The asymptotic upper bound in~\eqref{R2} holds, as the remainder term in Taylor series expansion of $H$ satisfies $r_H\ll1$. For Theorem~\ref{thm: error equiv-0}, we have $p=0$ and we need to check the extra condition~\eqref{R4prime}, which follows from the already proven asymptotic equivalences~\eqref{imrg-rphi-0} and~\eqref{rerg-rphi-0}. Thus, for each of Theorems~\ref{thm: error equiv}--\ref{thm: error equiv-0}, $H_\mu$ belongs to $\mathcal R_{p,0}$.

We now refer to the schematic diagram given in~\eqref{scheme}. As $H_\mu$ is also invertible with $L=H^{-1}\in\mathcal H$, by Proposition~\ref{inverse-taylor}, we also have $L\in\mathcal R_{p,0}$ and $r_H(z) \sim - r_L(z)$, $\Re r_H(-iy) \sim -\Re r_L(-iy)$ and $\Im r_H(-iy) \sim -\Im r_L(-iy)$. Clearly, then Proposition~\ref{fraction-taylor} applies to the function $L(z)/z$, which has reciprocal $K\in\mathcal H$. Thus, $r_K$ and $r_L$ satisfy the relevant asymptotic equivalences. Furthermore, since, $r_R \equiv r_K$, combining, we have $r_H(z) \sim r_R(z)$, $\Re r_H(-iy) \sim \Re r_R(-iy)$ and $\Im r_H(-iy) \sim \Im r_R(-iy)$. Further, for Theorem~\ref{thm: error equiv-0}, we have $p=0$ and $H\in\mathcal R_{p,0}$ satisfies~\eqref{R4prime}. Hence, we also have $\Re r_R(-iy) \approx \Im r_R(-iy)$. Then $R_\mu$ inherits the appropriate properties from $H_\mu$ and passes them on to $\phi_\mu$, which gives us the statement~\eqref{Voiculescu remainder} about the remainder term in Laurent expansion of Voiculescu transform in each of Theorems~\ref{thm: error equiv}--\ref{thm: error equiv-0}.

Conversely, assume the statement~\eqref{Voiculescu remainder}. Then the assumptions on $r_\phi$ imply the analogous properties for $r_R\equiv r_K$. Further, as $\mu$ is in $\mathcal M_p$, $zK(z) = z(1+zR_\mu(z))$ satisfies the hypothesis~\eqref{R1} for the class $\mathcal R_{p,0}$. Also, the remainder term of Taylor series expansion of $zK(z)$ is also given by $r_R\equiv r_K\ll1$. The lower bound for the imaginary part of the remainder term in the hypothesis~\eqref{R3} follows from its regular variation and the fact that $\alpha\in[p, p+1)$. The lower bound in the hypothesis~\eqref{R2} is part of the statement~\eqref{Voiculescu remainder}. The lower bound for the real part of the remainder term in the hypothesis~\eqref{R3} is also a part of the statement~\eqref{Voiculescu remainder} for Theorems~\ref{thm: error equiv} and~\ref{thm: error equiv eq p}, while it follows from the statement~\eqref{Voiculescu remainder-0} for Theorem~\ref{thm: error equiv-0}, as both the real and imaginary parts become asymptotically equivalent and regularly varying of index $\alpha$ with $\alpha\in[0,1)$. Finally, the asymptotic equivalence in~\eqref{R4prime} for Theorem~\ref{thm: error equiv-0} is a part of the statement~\eqref{Voiculescu remainder-0}. Thus, again for each of Theorems~\ref{thm: error equiv}--\ref{thm: error equiv-0}, $zK(z)$ belongs to $\mathcal R_{p,0}$. Then apply Proposition~\ref{fraction-taylor} on $K$ and then Proposition~\ref{inverse-taylor} on $z/K(z)=L(z)$ to obtain $H_\mu(z)$. Arguing, by checking the asymptotic equivalences as in the direct case, we obtain the required conclusions about $r_H$ and hence $r_G$ given in the statement~\eqref{Cauchy remainder} for each of Theorems~\ref{thm: error equiv}--\ref{thm: error equiv-0}.

The argument is same in the case $\alpha=p+1$, which applies to Theorem~\ref{thm: error equiv new}, with the observation that the stronger bounds required in the hypotheses~\eqref{R2},~\eqref{R3} and~\eqref{R4} with $\beta>0$ is assumed for $r_\phi$ and hence for $r_R$ and is proved for $r_G$ and hence for $r_H$ in Proposition~\ref{prop: upper bound rG} and~\ref{prop: rG iy}.
\end{proof}

\cleardoublepage

%\include{reference1}
%\input{synopsis}
%\cleardoublepage
%\printindex
\fancyhead[LE]{\sl\leftmark}
\bibliographystyle{abbrvnat}
%\bibliography{reference}

\end{document}